\documentclass[10pt,letterpaper]{article}
\usepackage[top=0.85in,left=2.75in,footskip=0.75in]{geometry}

\usepackage{amsmath,amssymb}

\usepackage{changepage}

\usepackage[utf8x]{inputenc}

\usepackage{textcomp,marvosym}

\usepackage{cite}

\usepackage{nameref,hyperref}

\usepackage[right]{lineno}

\usepackage{microtype}
\DisableLigatures[f]{encoding = *, family = * }

\usepackage[table]{xcolor}

\usepackage{array}

\newcolumntype{+}{!{\vrule width 2pt}}

\newlength\savedwidth

\raggedright
\usepackage[aboveskip=1pt,labelfont=bf,labelsep=period,justification=raggedright,singlelinecheck=off]{caption}

\makeatletter
\renewcommand{\@biblabel}[1]{\quad#1.}
\makeatother

\usepackage{lastpage,fancyhdr,graphicx}
\usepackage{epstopdf}
\fancyheadoffset[L]{2.25in}
\fancyfootoffset[L]{2.25in}
\usepackage{float}
\usepackage{adjustbox}
\usepackage{tikz,ifthen}
\usepackage{amsthm}
\usetikzlibrary{positioning,calc,patterns}
\usepgfmodule{nonlineartransformations}

\theoremstyle{plain}
\newtheorem{theo}{Theorem}[subsection]

\makeatletter
\tikzdeclarecoordinatesystem{polar}{
	\tikz@scan@one@point\relax(#1)
	\polartransformation
}
\def\polartransformation{% from the pgfmanual section 103.4.2
	\pgfmathsincos@{\pgf@sys@tonumber\pgf@x}%
	\pgf@x=\pgfmathresultx\pgf@y% 
	\pgf@y=\pgfmathresulty\pgf@y%
} % note: the following should work with arbitrary (nonlinear) transformations

\usepackage{cleveref}

\crefformat{section}{\S#2#1#3} % see manual of cleveref, section 8.2.1
\crefformat{subsection}{\S#2#1#3}
\crefformat{subsubsection}{\S#2#1#3}

\newif\ifcomment

\begin{document}
	\vspace*{0.2in}
	
	\begin{flushleft}
		{\Large
			\textbf\newline{Ginzburg-Landau energy and placement of singularities in generated cross fields} % Please use "sentence case" for title and headings (capitalize only the first word in a title (or heading), the first word in a subtitle (or subheading), and any proper nouns).
		}
		\newline
		\\
		Alexis Macq\textsuperscript{1,2},
		Maxence Reberol\textsuperscript{1},
		François Henrotte\textsuperscript{1},
		Pierre-Alexandre Beaufort\textsuperscript{1,3},
		Alexandre Chemin\textsuperscript{1},
		Jean-François Remacle\textsuperscript{1},
		Jean Van Schaftingen\textsuperscript{1}
		\\
		\bigskip
		\textsuperscript{1} UCLouvain, Louvain-la-Neuve, Belgium\\
		\textsuperscript{2} Haute Ecole Galilée - ISPG, Brussels, Belgium\\
		\textsuperscript{3} University of Bern, Switzerland\\
		\bigskip
		
		%% Insert additional author notes using the symbols described below. Insert symbol callouts after author names as necessary.
		%% 
		%% Remove or comment out the author notes below if they aren't used.
		%%
		%% Primary Equal Contribution Note
		%\Yinyang These authors contributed equally to this work.
		%
		%% Additional Equal Contribution Note
		%% Also use this double-dagger symbol for special authorship notes, such as senior authorship.
		%\ddag These authors also contributed equally to this work.
		%
		%% Current address notes
		%\textcurrency Current Address: Dept/Program/Center, Institution Name, City, State, Country % change symbol to "\textcurrency a" if more than one current address note
		%% \textcurrency b Insert second current address 
		%% \textcurrency c Insert third current address
		%
		%% Deceased author note
		%% Group/Consortium Author Note
		%\textpilcrow Membership list can be found in the Acknowledgments section.
		
		% Use the asterisk to denote corresponding authorship and provide email address in note below.
		Contact : alexis.macq@galilee.be
		
	\end{flushleft}
	\section*{Abstract}
	
	Cross field generation is often used as the basis for the construction of block-structured quadrangular meshes, and the field singularities have a key impact on the structure of the resulting meshes.
	In this paper, we extend Ginzburg-Landau cross field generation methods with a new formulation that allows a user to impose inner singularities. 
	The cross field is computed via the optimization of a linear objective function with localized quadratic constraints.
	This method consists in fixing singularities in small holes drilled in the computational domain with specific degree conditions on their boundaries, which leads to non-singular cross fields on the drilled domain.
	We also propose a way to calculate the Ginzburg-Landau energy of these cross fields on the perforated domain by solving a Neumann linear problem.
	This energy converges to the energy of the Ginzburg-Landau functional as epsilon and the radius of the holes tend to zero.
	To obtain insights concerning the sum of the inner singularity degrees, we give: (i) an extension of the Ginzburg-Landau energy to the piecewise smooth domain allowing to identify the positions and degrees of the boundary singularities, and (ii) an interpretation of the Poincaré-Hopf theorem focusing on internal singularities.

	\section*{Introduction}
	
	Automatic block-structured quadrilateral meshing techniques have made significant progress
	in the last decade~\cite{Ray2006,Palacios2007,Ray2008} thanks to broad use of smooth cross fields~\cite{quadcover,bommes2013,Campen2015} aligned with
	the model boundaries. The structure
	of the resulting quadrilateral meshes is mainly determined by the
	cross field properties, for instance, the mesh irregular vertices match the
	cross field singularities and the quad sizes are related to the size map
	inherent to the cross field. In this context, understanding and controlling
	the cross field properties is the natural way to influence the quad mesh
	and to tailor it for engineering applications.  The optimization of the
	mesh characteristics (size map, number of structured blocks, position of
	irregular vertices) via the cross field is still largely unexplored but of
	crucial importance for numerical simulation.
	
	Throughout this paper, a cross field will be considered as a discrete object that assigns four orthogonal directions and a norm to each point of a discretized domain. 
%	From smooth cross fields aligned with the domain boundaries there exists a lot of methods to extract a block-structured quadrangular mesh. 
	Those fields may include singularities
	which are points where four orthogonal directions can not be defined while keeping the field smooth. 
	Our paper focuses on the management of the singularities which is primordial to adapt to the needs of a high-precision numerical simulation. Indeed, the singularities appear to coincide with the corners of the blocks of the subsequent block-structured quadrangular mesh (for such results linked with Ginzburg-Landau theory see~\cite{Beaufort, Viertel}). 
	The aim of this paper is to let a user impose inner singularities in a cross field while giving him methods to evaluate his choices in terms of Ginzburg-Landau energy. 
	
	This paper begins with a short literature review~\ref{RelatedWork}.
	We then propose mathematical background~\ref{TheoBackground} in which we give a problem to evaluate the choice of singularities as perforations in the computational domain~\ref{SingEvaluation}. 
	The asymptotic equivalence between this energy and the Ginzburg-Landau energy functional is given. We also highlight a strong link between the energetic Ginzburg-Landau functional and an energy that involves only some positions and degrees.
	Indeed, the positions and degrees that minimize this last energy are the same as the positions and degrees of the singularities of the minimizers of the Ginzburg-Landau  functional~\ref{GLEnergies}. 
	Next, we introduce the main proposed method of this article~\ref{SingPlacement}.
	It consists in trapping the singularities in small holes being drilled in the computational domain and holding specific boundary conditions. 
	By doing so, a user can choose a configuration of inner singularities as holes in a computed non-singular smooth cross field.
	Unfortunately, our main method does not allow to impose boundary singularities.
	This have lead us to propose some theoretical novelties. They bring clues to understand which imposed inner singularity configurations will lead to respect the Poincaré-Hopf theorem in case of angular boundaries.
	In terms of theoretical inputs, we extend the Ginzburg-Landau energy to piecewise-smooth closed domains and give a strategy for managing singularity configurations based on the minimization of this new Ginzburg-Landau boundary energy~\ref{GLPiecewise}. We also give new interpretation of the generalized Poincaré-Hopf theorem focusing on internal singularities~\ref{GenPH}.
	As further thoughts, we present a Neumann problem with zero-radius holes as singularities~\ref{FurThou}. 
	Those dot-sized holes are the same as the ones used in the computation of the asymptotic energy of the drilled domains.
	The solution of this Neumann problem is a single-valued scalar field that is strongly linked to the multivalued angles of a cross field. 
	Indeed, we can define an angle for each cross of a smooth cross field.
	The computed single-valued field and this phase field have perpendicular gradients and form together a complex harmonic field almost everywhere. This Neumann problem could lead to a new simpler way to represent cross fields and help to extract useful size information from it. 
	Its inputs consist of configurations of both the inner and the boundary singularities. 
	This last method offers more freedom to a user than the main method of our paper while this last one creates a boundary singularity configuration that minimizes our generalized Ginzburg-Landau energy for piecewise-smooth domain.

	\section{Related work}\label{RelatedWork}
    
     Let us briefly examine the literature that is either directly related to our work or that follows the same objectives. 
     For broader literature reviews see the papers of Vaxman et al.~\cite{Vaxman2016} and Bommes et al.~\cite{Bommes2012, bommes2013}. Concerning the theory on mesh generation in two dimensions, we refer the reader to a paper of Bunin~\cite{Bunin2008}.
    
    In the literature, there are two main ways to generate a block-structured quad mesh. Either a cross field is computed to extract a quad mesh via a parameterization or the parameterization is directly computed.
    However, there is a strong links between the two approaches~\cite{Myles2013}.
    The equivalent of singularities in the context of cross fields are cones in the context of direct parameterization.
    In either case, the quality of the resulting object depends on the amount of distortion in it, the number of singularities/cones, and the alignment with features and boundaries~\cite{Myles2013}. 
    The methods based on cross fields always emphasizes the boundary alignment whereas parameterization methods are often used for domain without boundaries~\cite{Ben-Chen2008} or  without boundary alignment (e.g. for texture mapping~\cite{Sawhney:2017:BFF}).
    There are integer constraints in parameterization to ensure the possibility to extract a block-structured quadrangular mesh generation from it. These constraints are called quantization and form a complex combinatorial problem (see \cite{Campen2015} for a way to solve them).
    Thus, the first quality that we seek in cross field for block-structured quadrangular meshing is the alignment with the domain boundaries.
    On one hand, only singularities that make sense in a quadrangular mesh are used in the cross field context, i.e. singularities that are such that their degree implies an integer number of neighbors for the corresponding node in the underlying mesh. 
    On the other hand, in parameterization, the correct connection is not always requested and depends on the use of integer parameters. 
    In parameterization for four-sided zone construction, we have two perpendicular fields and if we want an alignment to the domain boundary, an iso value of one of the two fields must correspond to the boundary.
    With regard to the cross fields, alignments with boundaries or features are straightforward. Indeed it just imply that the tangents of the boundary or feature curves corresponds to one of the vectors of the cross field at each point.
    When we extract a parameterization from a cross field, the cones are located at field singularities and
    gradients of parametric coordinates are aligned
    with the field vectors \cite{Myles2013}.

%    "An internal
%    node of a quad mesh is said to be irregular if its valence is not four, and a boundary
%    node if its valence is not three. A quad layout is similar to a quad mesh, but in
%    place of straight edges piecewise smooth curves are allowed. A mesh or quad layout
%    is conforming if any two faces share at most a single vertex or an entire edge."

    Myles and Zorin~\cite{Myles2013} have proposed a method that place singularities to minimize the size distortion in a parameterization. 
    Regarding boundary-aligned parameterization, they show that metric distortion may be reduced by cone chains but that there will always have a trade-off, which can be arbitrated by a regularization factor, between the distortion and the number of cones.
    In our context, this paper therefore confirms that adding singularities can help to limit the difference in size on computed field or parameterization and therefore on block-structured quadrangular mesh extracted from them.
    The same tradeoff appears in the Ginzburg-Landau theory. 
    It can be controlled either by a factor $\epsilon$ or by the radius of singular holes. 
    It is important to have a tradeoff because placing a lot of singularities can lead to the appearance of very small blocks compared to the average block size in the block-structured mesh that we can extract from a cross field.
    As these factors tend to zero, we obtained asymptotic Ginzburg-Landau energy for which the main minimization concerns the number and order of singularities.
    Contrary to what we propose, the method Myles and Zorin \cite{Myles2013} does not allow a user to impose singularities.
	One point that seems particularly interesting to us is that they include metric distortion directly in the optimization energy.
	The paper of Myles and Zorin~\cite{Myles2013} does not consider additional integer constraints needed for quadrangulation.

%	\comment{TODO: describe the X contributions/parts of the paper, in one sentence}
%	
%	\comment{TODO: one paragraph describing each contribution/part (for each 2/3 sentences max)}
%   
%    
%    \comment{Ben-Chen 2008 does not consider alignment constraints}

%    Not any 1-form
%    corresponds to a cross field, which leads to additional conditions
%    on the integrals of the 1-form on loops. For a closed loop, the full
%    change in this angle should be of the form $k\pi/2$, as after a full circle
%    the cross field should be mapped to itself.
%    Clearly, there are many cross-
%    fields satisfying the constraints above, as these restrict the field ori-
%    entation on the boundary curves only. We can choose a R unique field
%    among these by picking the one for which $\|\omega \|^2_2 = \int \omega \wedge \star \omega$ is
%    minimal, i.e. the integral of the rotation rate over all directions and
%    all points is as small as possible. This is exactly the functional used
%    in [Crane et al. 2010] as well as in [Ray et al. 2008; Bommes et al.
%    2009] in different variables.
 
    In addition to cross-field-based methods for constructing parameterization, for the purpose of block-structured mesh generation, methods based on streamlines are common.
    These streamlines extend the direction adjacent to field singularities. 
    Contrary to parameterization, streamline tracing methods have a direct connection with the topology
    of the cross field.
    Leveraging the Ginzburg-Landau theory, Viertel and Osting~\cite{Viertel} claim that Ginzburg-Landau cross field generation methods can be used to produce a cross field whose separatrices divide the domain into four sided regions.
    They investigate the mathematics of cross field generation and cross field guided quad meshing via streamline tracing.
	They also made the observation that many of the computational methods currently used for cross field design attempt to minimize an energy close to the Ginzburg-Landau energy.
	They also highlight the difficulty to compute numerically asymptotic minimizers of the energetic Ginzburg-Landau functional.
	They suggest
	that this difficulty is link with the fact that the Ginzburg-Landau energy assigns too much weight to singularities.
	They further emphasize a slight influence of the refinement of the initial discretization in the discretely computed energy.
    An important point of this paper is that Viertel and Osting use ad-hoc methods to define boundary singularities and hope to see the Ginzburg-Landau theory extended to domains with piecewise-smooth boundaries. 
    Precisely, later in this paper, we propose such an extension of the Ginzburg-Landau theory.

    Beaufort et al.~\cite{Beaufort} propose to directly solve the energetic Ginzburg-Landau functional with a finite element formulation and with the use of edge-based Crouzeix-Raviart
    interpolation functions.
    They also give theoretical results.
    Particularly, they show that the positions of the singularities in the computed cross fields coincide with the positions of irregular vertices in the mesh that can be extracted from it.
    
%    Inspiration.Many algorithms in computer graphics and geometry processing use two orthogonal smooth direction fields (unit tangent vector fields) defined over a surface. For instance, these direction fields are used in texture synthesis, in geometry processing or in nonphotorealistic rendering to distribute and orient elements on the surface. Such direction fields can be designed in fundamentally different ways, according to the symmetry requested: inverting a direction or swapping two directions might be allowed or not. Despite the advances realized in the last few years in the domain of geometry processing, a unified formalism is still lacking for the mathematical object that characterizes these generalized direction fields. As a consequence, existing direction field design algorithms are limited to using nonoptimum local relaxation procedures. 
	Ray et al.~\cite{Ray2008} give a definition of the singularities of N-symmetry direction fields close to the one that we propose on this article. The proposed definition allows to relate singularities to the topology of the surface. Specifically, they provide a generalization of the Poincaré-Hopf theorem to N-symmetry direction fields on 2-manifolds.
    The proposed algorithm allows to produce N-symmetry  directional fields based on user-defined sets of singularities.
    The approach of Ray et al. uses a greedy algorithm to constrain singularities through the topology.
    The variables used to represent the directions are angles relatively measured to a given arbitrary direction.
    This implies an inherent difficulty since the fields of phases are often multivalued.
    The researchers formulate the generation of cross fields as a quadratic form minimization solved via conjugate gradient.
    However, no quality estimations of the user-based locations of singularities are given. A lot of topological analysis of interest are done in the article of Ray et al.~\cite{Ray2008}.
    In direct link with the work that we will present, Ray et al. introduce the idea of replacing singularities with holes, so that they can be handled as borders. 
    However, this paper present a different approach, as
    our formulation has strong links with Ginzburg-Landau theory. Furthermore, our work is not based on angles, which imply harder numerical computations. Last but not least, we work with an efficient linear problem with some very localized quadratic constraints.

	Crane et al. \cite{Crane:2010:TCD} have developed a particular efficient way to impose singularities in direction fields. 
	They based their method on consistent connection and the dual mesh of the input triangular mesh. Such triangular meshes are classically used as a base in direction field generation. We also use them as a base for all the numerical computations that we described in our article.
%	A connection specifies the infinitesimal rotation associated with the motion in any direction and could be used to describe how any tangent vector changes as it moves along the surface.
	% They work with special connections that are said consistent because they ensure that the map of a vector from one point to another is independent of the path taken.
	% These connections allow to compare tangent vectors that are originated at different points of the computational manifold and provide consistent parallel transport.
	% They are defined on the dual mesh of the triangle input mesh.
	% Holonomy : another way to look at consistency: transport vector around a closed loop; may not end up back where we started. If loop bounds a region, holonomy determines curvature over this region.
	% holonomy around contractible loop => curvature over bounded region
	Crane et al. demonstrate that such consistent connections can be computed easily and efficiently.
	They use homology consistent constraints on particular cycles, contractible ones and non-contractible ones, to construct their consistent connections.
	Such as for the cones in parameterization, they concentrate curvature on singularities but in the same time keep a consistent parallel transportation law.
	Their method summarizes in just finding the minimum 2-norm solution of a constrained linear system.
	%Each of their solutions give a trivial connection closest to the Levi-Civita one among all connections with the prescribed set of singularities.
	% Gauss-Bonnet says int_M K = 2pi*chi
	Their utilization of connection on dual meshes gives interesting links between smooth and discrete geometry.
	Just like us, they highlight the importance of topological constraints.
	An important point is the enlightened tradeoff between total control over the singularities and simplicity of the methods for singularity placements.
	A remarkable point is that their method allows to never generate non-wanted
	singularities.
	However, their method does not focus on boundaries or features alignment that are essential in the purpose of block-structured quadrangular mesh generation for high performance numerical simulation.

	\section{Theoretical background}\label{TheoBackground}
	
	A cross field is a particular case of a $n$-symmetry directional field with $n=4$ (for an historical paper on this matter see~\cite{Ray2008}).
	It can be seen as a quadruple of same norm vectors forming a regular cross,
	i.e., that are either orthogonal or opposite to each other.
	A cross field $c$ on a planar region $G$ is a mapping 
	that associates a cross $c(P)$ to every point $P \in G$.
	The four vectors composing a cross $c(P)$ can be regarded as 
	the $4^{th}$ roots of a unique complex number $u(P)= c^4(P)$ (for a justification see~\cite{Beaufort}). This representative complex number (or 2-d vector) is in general not parallel to any of the branches of the cross it represents. 
	A smooth cross field can thus be represented unambiguously 
	as a continuous $u\,: G \mapsto \mathcal{C}$ complex function. %\comment{simplement connexe ? ou ok sans le simplement connexe ? A creuser}.
	The cross field $c$ is then the field of the $4^{th}$ roots of $u$ as illustrated in the Figure \ref{fig:DiskAsExample}.
	\begin{figure}[ht] 
		\begin{minipage}[b]{.5\linewidth}
			\centering
			\includegraphics[width=\linewidth]{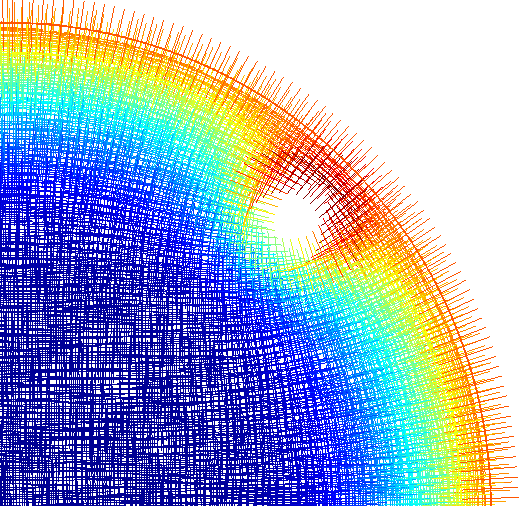}
		\end{minipage}%%
		\begin{minipage}[b]{.5\linewidth}
			\centering
			\includegraphics[width=\linewidth]{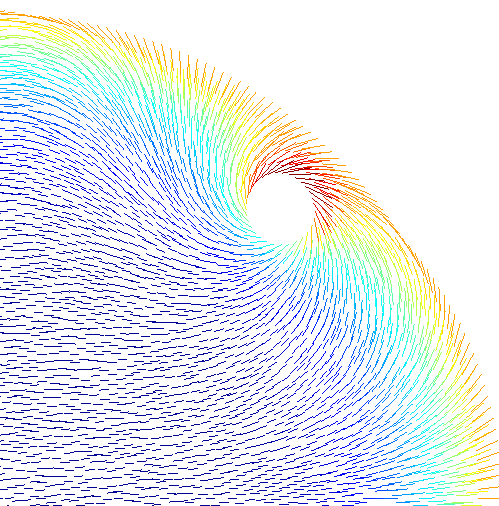} 
			%\vspace{4ex}
		\end{minipage}%% 
		\caption{The crosses of the cross field represented on the left image are the $4^{th}$ roots of the vectors of the vector field represented on the right image.}
		\label{fig:DiskAsExample} 
	\end{figure}
	
	We will take advantage of this link to build discrete cross field analysis on the very established and rich continuous vector field analysis. In particular, one singularity of the representing vector field will correspond to four singularities of the derived cross field. These four singularities will have an order equals to one fourth of the one of the vector field (for a justification see~\cite{Beaufort}).\\
	
	Vector fields and complex functions are equivalent concepts on planar regions
	and we shall use indifferently one or the other according to the context.
	We will see that the notions of degrees and indices are strongly linked and that they allow, via the Poincaré-Hopf theorem, to characterize the total index of a field on the basis of the Euler-Poincaré characteristic of the domain in which it is calculated.
	
	\subsection{Brouwer degree}
	
	The Brouwer degree $\text{deg}(v,\gamma)$ 
	of a vector field $v$ on a smooth closed curve $\gamma$ 
	plays a central role in the context of this paper.
	In mathematical terms, 
	it is defined as the degree of the application $v$
	regarded as a map from $\gamma$ into $S^1$,
	where $S^1$ represents the unit circle in the complex plane. % \comment{Cela me semble confus: on définit un degré à partir d’un degré. Il est important de mentionner qu’il faut que v ne s’annule pas.}. 
	Such an application is not always defined which will
	in those cases be the same for the Brouwer degree.
	Nevertheless, said application and Brouwer degree are well-defined 
	whenever the field $v$ and the curve $\gamma$ 
	are locally planar,
	so that a canonical mapping $S^1\mapsto S^1$ can be found.
	In those instances, the Brouwer degree 
	has a simple and practical geometrical interpretation 
	as the value of the integral 
	\begin{equation}
	\text{deg}(v,\gamma) = \frac{1}{2\pi} \int_{\gamma} d\theta(v)
	\label{eq:brouwerdeg}
	\end{equation}
	%\comment{Cette formule peut s’écrire sans définir l’angle, en prenant le déterminant de v et de sa dérivée divisé par la norme.}
	where $\theta(v)$ is the angle formed by the vector field $v$ 
	and a fixed reference vector, 
	e.g., the x-axis of a local planar coordinate system.
	Note that equivalent formula without the definition of an angle exist (for definitions either via the determinant of a matrix, via an exterior product or via lifting see~\cite{Hatcher2002}).
	That angle is counted positively according to the orientation of the curve $\gamma$. 
	The Brouwer degree is integer-valued for vector fields and it is independent of the orientation of $\gamma$.

	\subsection{Index of a singularity}
	
	\begin{figure}[ht]
		\begin{center}
			
			\tikzset{every picture/.style={line width=0.75pt}} %set default line width to 0.75pt        
			
			\begin{tikzpicture}[x=0.75pt,y=0.75pt,yscale=-1,xscale=1]
			%uncomment if require: \path (0,465); %set diagram left start at 0, and has height of 465
			
			%Shape: Circle [id:dp1926704496599122] 
			\draw  [color={rgb, 255:red, 208; green, 2; blue, 27 }  ,draw opacity=1 ] (229.8,134.8) .. controls (229.8,120.99) and (240.99,109.8) .. (254.8,109.8) .. controls (268.61,109.8) and (279.8,120.99) .. (279.8,134.8) .. controls (279.8,148.61) and (268.61,159.8) .. (254.8,159.8) .. controls (240.99,159.8) and (229.8,148.61) .. (229.8,134.8) -- cycle ;
			%Straight Lines [id:da10435778588853373] 
			\draw [color={rgb, 255:red, 208; green, 2; blue, 27 }  ,draw opacity=1 ]   (271.74,153.02) -- (266.52,154.47) ;

			%Straight Lines [id:da7367281041846702] 
			\draw [color={rgb, 255:red, 208; green, 2; blue, 27 }  ,draw opacity=1 ]   (269.63,158.02) -- (271.74,153.02) ;

			%Shape: Circle [id:dp641039627622315] 
			\draw  [color={rgb, 255:red, 208; green, 2; blue, 27 }  ,draw opacity=1 ] (380.8,134.6) .. controls (380.8,120.79) and (391.99,109.6) .. (405.8,109.6) .. controls (419.61,109.6) and (430.8,120.79) .. (430.8,134.6) .. controls (430.8,148.41) and (419.61,159.6) .. (405.8,159.6) .. controls (391.99,159.6) and (380.8,148.41) .. (380.8,134.6) -- cycle ;
			%Straight Lines [id:da5596505114090224] 
			\draw [color={rgb, 255:red, 208; green, 2; blue, 27 }  ,draw opacity=1 ]   (422.74,152.82) -- (417.52,154.27) ;

			%Straight Lines [id:da022135949722710713] 
			\draw [color={rgb, 255:red, 208; green, 2; blue, 27 }  ,draw opacity=1 ]   (420.63,157.82) -- (422.74,152.82) ;

			%Shape: Circle [id:dp7373707040600409] 
			\draw  [color={rgb, 255:red, 208; green, 2; blue, 27 }  ,draw opacity=1 ] (229,254.8) .. controls (229,240.99) and (240.19,229.8) .. (254,229.8) .. controls (267.81,229.8) and (279,240.99) .. (279,254.8) .. controls (279,268.61) and (267.81,279.8) .. (254,279.8) .. controls (240.19,279.8) and (229,268.61) .. (229,254.8) -- cycle ;
			%Straight Lines [id:da9793998200960502] 
			\draw [color={rgb, 255:red, 208; green, 2; blue, 27 }  ,draw opacity=1 ]   (270.94,273.02) -- (265.72,274.47) ;

			%Straight Lines [id:da02910424344204865] 
			\draw [color={rgb, 255:red, 208; green, 2; blue, 27 }  ,draw opacity=1 ]   (268.83,278.02) -- (270.94,273.02) ;

			%Shape: Circle [id:dp0746374173750538] 
			\draw  [color={rgb, 255:red, 208; green, 2; blue, 27 }  ,draw opacity=1 ] (380,255) .. controls (380,241.19) and (391.19,230) .. (405,230) .. controls (418.81,230) and (430,241.19) .. (430,255) .. controls (430,268.81) and (418.81,280) .. (405,280) .. controls (391.19,280) and (380,268.81) .. (380,255) -- cycle ;
			%Straight Lines [id:da964239225098446] 
			\draw [color={rgb, 255:red, 208; green, 2; blue, 27 }  ,draw opacity=1 ]   (421.94,273.22) -- (416.72,274.67) ;

			%Straight Lines [id:da9139367719898019] 
			\draw [color={rgb, 255:red, 208; green, 2; blue, 27 }  ,draw opacity=1 ]   (419.83,278.22) -- (421.94,273.22) ;

			%Straight Lines [id:da2946621571089083] 
			\draw    (254.8,109.8) -- (255.11,83.47) ;
			\draw [shift={(255.13,81.47)}, rotate = 450.67] [fill={rgb, 255:red, 0; green, 0; blue, 0 }  ][line width=0.75]  [draw opacity=0] (10.72,-5.15) -- (0,0) -- (10.72,5.15) -- (7.12,0) -- cycle    ;
			
			%Straight Lines [id:da9282701558536903] 
			\draw    (405.47,187.93) -- (405.78,161.6) ;
			\draw [shift={(405.8,159.6)}, rotate = 450.67] [fill={rgb, 255:red, 0; green, 0; blue, 0 }  ][line width=0.75]  [draw opacity=0] (10.72,-5.15) -- (0,0) -- (10.72,5.15) -- (7.12,0) -- cycle    ;
			
			%Straight Lines [id:da5621739562195777] 
			\draw    (278.83,268.97) -- (279.14,242.63) ;
			\draw [shift={(279.17,240.63)}, rotate = 450.67] [fill={rgb, 255:red, 0; green, 0; blue, 0 }  ][line width=0.75]  [draw opacity=0] (10.72,-5.15) -- (0,0) -- (10.72,5.15) -- (7.12,0) -- cycle    ;
			
			%Straight Lines [id:da44983503775066824] 
			\draw    (429.83,269.17) -- (430.14,242.83) ;
			\draw [shift={(430.17,240.83)}, rotate = 450.67] [fill={rgb, 255:red, 0; green, 0; blue, 0 }  ][line width=0.75]  [draw opacity=0] (10.72,-5.15) -- (0,0) -- (10.72,5.15) -- (7.12,0) -- cycle    ;
			
			%Straight Lines [id:da12925206948206192] 
			\draw    (237.47,116.8) -- (220.09,96.97) ;
			\draw [shift={(218.77,95.47)}, rotate = 408.76] [fill={rgb, 255:red, 0; green, 0; blue, 0 }  ][line width=0.75]  [draw opacity=0] (10.72,-5.15) -- (0,0) -- (10.72,5.15) -- (7.12,0) -- cycle    ;
			
			%Straight Lines [id:da09413562146669296] 
			\draw    (441.44,174.16) -- (424.06,154.33) ;
			\draw [shift={(422.74,152.82)}, rotate = 408.76] [fill={rgb, 255:red, 0; green, 0; blue, 0 }  ][line width=0.75]  [draw opacity=0] (10.72,-5.15) -- (0,0) -- (10.72,5.15) -- (7.12,0) -- cycle    ;
			
			%Straight Lines [id:da13933946643181716] 
			\draw    (442.93,293.16) -- (425.55,273.33) ;
			\draw [shift={(424.23,271.82)}, rotate = 408.76] [fill={rgb, 255:red, 0; green, 0; blue, 0 }  ][line width=0.75]  [draw opacity=0] (10.72,-5.15) -- (0,0) -- (10.72,5.15) -- (7.12,0) -- cycle    ;
			
			%Straight Lines [id:da8140258527738816] 
			\draw    (278.37,245.23) -- (260.99,225.4) ;
			\draw [shift={(259.67,223.9)}, rotate = 408.76] [fill={rgb, 255:red, 0; green, 0; blue, 0 }  ][line width=0.75]  [draw opacity=0] (10.72,-5.15) -- (0,0) -- (10.72,5.15) -- (7.12,0) -- cycle    ;
			
			%Straight Lines [id:da20168338647909723] 
			\draw    (229.8,134.8) -- (202.23,134.73) ;
			\draw [shift={(200.23,134.73)}, rotate = 360.14] [fill={rgb, 255:red, 0; green, 0; blue, 0 }  ][line width=0.75]  [draw opacity=0] (10.72,-5.15) -- (0,0) -- (10.72,5.15) -- (7.12,0) -- cycle    ;
			
			%Straight Lines [id:da09510441953873228] 
			\draw    (460.37,134.67) -- (432.8,134.6) ;
			\draw [shift={(430.8,134.6)}, rotate = 360.14] [fill={rgb, 255:red, 0; green, 0; blue, 0 }  ][line width=0.75]  [draw opacity=0] (10.72,-5.15) -- (0,0) -- (10.72,5.15) -- (7.12,0) -- cycle    ;
			
			%Straight Lines [id:da41734337557381807] 
			\draw    (268.79,229.84) -- (241.21,229.77) ;
			\draw [shift={(239.21,229.76)}, rotate = 360.14] [fill={rgb, 255:red, 0; green, 0; blue, 0 }  ][line width=0.75]  [draw opacity=0] (10.72,-5.15) -- (0,0) -- (10.72,5.15) -- (7.12,0) -- cycle    ;
			
			%Straight Lines [id:da13501035069230138] 
			\draw    (419.79,280.04) -- (392.21,279.97) ;
			\draw [shift={(390.21,279.96)}, rotate = 360.14] [fill={rgb, 255:red, 0; green, 0; blue, 0 }  ][line width=0.75]  [draw opacity=0] (10.72,-5.15) -- (0,0) -- (10.72,5.15) -- (7.12,0) -- cycle    ;
			
			%Straight Lines [id:da7508914452991003] 
			\draw    (254.47,159.8) -- (254.47,188.5) ;
			\draw [shift={(254.47,190.5)}, rotate = 270] [fill={rgb, 255:red, 0; green, 0; blue, 0 }  ][line width=0.75]  [draw opacity=0] (10.72,-5.15) -- (0,0) -- (10.72,5.15) -- (7.12,0) -- cycle    ;
			
			%Straight Lines [id:da5680321301126873] 
			\draw    (405.8,78.9) -- (405.8,107.6) ;
			\draw [shift={(405.8,109.6)}, rotate = 270] [fill={rgb, 255:red, 0; green, 0; blue, 0 }  ][line width=0.75]  [draw opacity=0] (10.72,-5.15) -- (0,0) -- (10.72,5.15) -- (7.12,0) -- cycle    ;
			
			%Straight Lines [id:da5950405178175229] 
			\draw    (229,239.45) -- (229,268.15) ;
			\draw [shift={(229,270.15)}, rotate = 270] [fill={rgb, 255:red, 0; green, 0; blue, 0 }  ][line width=0.75]  [draw opacity=0] (10.72,-5.15) -- (0,0) -- (10.72,5.15) -- (7.12,0) -- cycle    ;
			
			%Straight Lines [id:da8645427959306156] 
			\draw    (380,239.65) -- (380,268.35) ;
			\draw [shift={(380,270.35)}, rotate = 270] [fill={rgb, 255:red, 0; green, 0; blue, 0 }  ][line width=0.75]  [draw opacity=0] (10.72,-5.15) -- (0,0) -- (10.72,5.15) -- (7.12,0) -- cycle    ;
			
			%Straight Lines [id:da03978347386246006] 
			\draw    (279.8,134.8) -- (307.12,134.87) ;
			\draw [shift={(309.12,134.87)}, rotate = 180.14] [fill={rgb, 255:red, 0; green, 0; blue, 0 }  ][line width=0.75]  [draw opacity=0] (10.72,-5.15) -- (0,0) -- (10.72,5.15) -- (7.12,0) -- cycle    ;
			
			%Straight Lines [id:da7020914539221407] 
			\draw    (351.48,134.53) -- (378.8,134.6) ;
			\draw [shift={(380.8,134.6)}, rotate = 180.14] [fill={rgb, 255:red, 0; green, 0; blue, 0 }  ][line width=0.75]  [draw opacity=0] (10.72,-5.15) -- (0,0) -- (10.72,5.15) -- (7.12,0) -- cycle    ;
			
			%Straight Lines [id:da43705183176543383] 
			\draw    (239.34,279.76) -- (266.66,279.83) ;
			\draw [shift={(268.66,279.84)}, rotate = 180.14] [fill={rgb, 255:red, 0; green, 0; blue, 0 }  ][line width=0.75]  [draw opacity=0] (10.72,-5.15) -- (0,0) -- (10.72,5.15) -- (7.12,0) -- cycle    ;
			
			%Straight Lines [id:da4525164703347795] 
			\draw    (390.34,229.96) -- (417.66,230.03) ;
			\draw [shift={(419.66,230.04)}, rotate = 180.14] [fill={rgb, 255:red, 0; green, 0; blue, 0 }  ][line width=0.75]  [draw opacity=0] (10.72,-5.15) -- (0,0) -- (10.72,5.15) -- (7.12,0) -- cycle    ;
			
			%Straight Lines [id:da7074797047358623] 
			\draw    (273.23,117.23) -- (292.04,99.12) ;
			\draw [shift={(293.48,97.73)}, rotate = 496.08] [fill={rgb, 255:red, 0; green, 0; blue, 0 }  ][line width=0.75]  [draw opacity=0] (10.72,-5.15) -- (0,0) -- (10.72,5.15) -- (7.12,0) -- cycle    ;
			
			%Straight Lines [id:da6556967810971079] 
			\draw    (367.03,171.43) -- (385.84,153.32) ;
			\draw [shift={(387.28,151.93)}, rotate = 496.08] [fill={rgb, 255:red, 0; green, 0; blue, 0 }  ][line width=0.75]  [draw opacity=0] (10.72,-5.15) -- (0,0) -- (10.72,5.15) -- (7.12,0) -- cycle    ;
			
			%Straight Lines [id:da6296403647315417] 
			\draw    (260.82,282.77) -- (279.63,264.66) ;
			\draw [shift={(281.07,263.27)}, rotate = 496.08] [fill={rgb, 255:red, 0; green, 0; blue, 0 }  ][line width=0.75]  [draw opacity=0] (10.72,-5.15) -- (0,0) -- (10.72,5.15) -- (7.12,0) -- cycle    ;
			
			%Straight Lines [id:da7330872923779359] 
			\draw    (423.03,237.03) -- (441.84,218.92) ;
			\draw [shift={(443.28,217.53)}, rotate = 496.08] [fill={rgb, 255:red, 0; green, 0; blue, 0 }  ][line width=0.75]  [draw opacity=0] (10.72,-5.15) -- (0,0) -- (10.72,5.15) -- (7.12,0) -- cycle    ;
			
			%Straight Lines [id:da7132515763309917] 
			\draw    (274.62,150.97) -- (296.71,170.25) ;
			\draw [shift={(298.22,171.57)}, rotate = 221.12] [fill={rgb, 255:red, 0; green, 0; blue, 0 }  ][line width=0.75]  [draw opacity=0] (10.72,-5.15) -- (0,0) -- (10.72,5.15) -- (7.12,0) -- cycle    ;
			
			%Straight Lines [id:da506175224556171] 
			\draw    (363.75,96.5) -- (385.84,115.78) ;
			\draw [shift={(387.35,117.1)}, rotate = 221.12] [fill={rgb, 255:red, 0; green, 0; blue, 0 }  ][line width=0.75]  [draw opacity=0] (10.72,-5.15) -- (0,0) -- (10.72,5.15) -- (7.12,0) -- cycle    ;
			
			%Straight Lines [id:da5264133890342532] 
			\draw    (363.75,215.9) -- (385.84,235.18) ;
			\draw [shift={(387.35,236.5)}, rotate = 221.12] [fill={rgb, 255:red, 0; green, 0; blue, 0 }  ][line width=0.75]  [draw opacity=0] (10.72,-5.15) -- (0,0) -- (10.72,5.15) -- (7.12,0) -- cycle    ;
			
			%Straight Lines [id:da06618149867840173] 
			\draw    (226.35,264.1) -- (248.44,283.38) ;
			\draw [shift={(249.95,284.7)}, rotate = 221.12] [fill={rgb, 255:red, 0; green, 0; blue, 0 }  ][line width=0.75]  [draw opacity=0] (10.72,-5.15) -- (0,0) -- (10.72,5.15) -- (7.12,0) -- cycle    ;
			
			%Straight Lines [id:da02104309563131246] 
			\draw    (236.6,152.8) -- (219.07,170.29) ;
			\draw [shift={(217.65,171.7)}, rotate = 315.08000000000004] [fill={rgb, 255:red, 0; green, 0; blue, 0 }  ][line width=0.75]  [draw opacity=0] (10.72,-5.15) -- (0,0) -- (10.72,5.15) -- (7.12,0) -- cycle    ;
			
			%Straight Lines [id:da7760977460372194] 
			\draw    (442.6,97.6) -- (425.07,115.09) ;
			\draw [shift={(423.65,116.5)}, rotate = 315.08000000000004] [fill={rgb, 255:red, 0; green, 0; blue, 0 }  ][line width=0.75]  [draw opacity=0] (10.72,-5.15) -- (0,0) -- (10.72,5.15) -- (7.12,0) -- cycle    ;
			
			%Straight Lines [id:da8342181485434075] 
			\draw    (245.4,228.2) -- (227.87,245.69) ;
			\draw [shift={(226.45,247.1)}, rotate = 315.08000000000004] [fill={rgb, 255:red, 0; green, 0; blue, 0 }  ][line width=0.75]  [draw opacity=0] (10.72,-5.15) -- (0,0) -- (10.72,5.15) -- (7.12,0) -- cycle    ;
			
			%Straight Lines [id:da923460649866533] 
			\draw    (387.2,272.8) -- (370.19,290.89) ;
			\draw [shift={(368.82,292.35)}, rotate = 313.23] [fill={rgb, 255:red, 0; green, 0; blue, 0 }  ][line width=0.75]  [draw opacity=0] (10.72,-5.15) -- (0,0) -- (10.72,5.15) -- (7.12,0) -- cycle    ;

			% Text Node
			\draw (289.8,151.8) node   {$\textcolor[rgb]{0.82,0.01,0.11}{\gamma }\textcolor[rgb]{0.82,0.01,0.11}{_{1}}$};
			% Text Node
			\draw (440.8,151.6) node   {$\textcolor[rgb]{0.82,0.01,0.11}{\gamma }\textcolor[rgb]{0.82,0.01,0.11}{_{2}}$};
			% Text Node
			\draw (289,271.8) node   {$\textcolor[rgb]{0.82,0.01,0.11}{\gamma }\textcolor[rgb]{0.82,0.01,0.11}{_{3}}$};
			% Text Node
			\draw (440,272) node   {$\textcolor[rgb]{0.82,0.01,0.11}{\gamma }\textcolor[rgb]{0.82,0.01,0.11}{_{4}}$};

			\end{tikzpicture}

		\end{center}
		\caption{The first three vector fields represented in black have a Brouwer degree equals to 1 around the red paths. We have 
			a source on $\gamma_1$, a sink on $\gamma_2$ and a vortex on $\gamma_3$. 
			The Brouwer degree of the fourth vector field $v$ is -1 on $\gamma_4$ (saddle).}
		\label{fig:saddleSinkEtc}
	\end{figure}
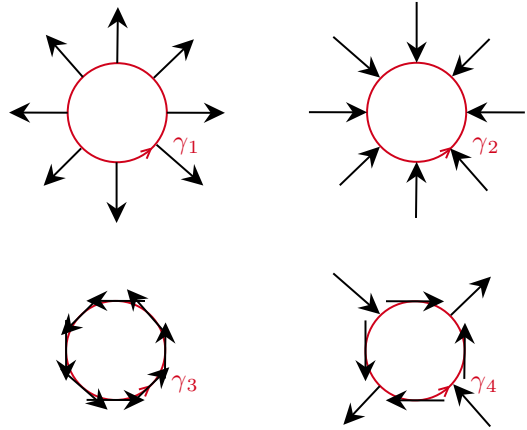
	
	A first situtation where the Brouwer degree exists
	is when the curve $\gamma$,
	defined on a smooth surface $M$ ($\gamma \subset M$), is infinitesimal,
	and therefore can be considered locally planar. That implies that the presented method could easily be extended to manifolds.
	Let $v$ be a vector field defined on $M$.
	The index of an isolated singularity of $v$
	located at a point $P_i \in M$, is defined as the Brouwer degree
	\begin{equation}
	\text{index}(P_i) = \text{deg}(v,\gamma_i(P_i)) = \frac{1}{2\pi} \int_{\gamma_i(P_i)} d\theta(v),
	\label{eq:indexsing}
	\end{equation}
	where $\gamma_i \subset M$ is a circle centered at $P_i$ of infinitesimal radius,
	so that it encloses no other singularity.
	The index is 1 if the singularity is of the type source, vortex or sink, 
	and it is -1 for a saddle type singularity (see figure \ref{fig:saddleSinkEtc}).
	
	\subsection{Poincar\'e-Hopf theorem}
	
	The notion of index entails a very fruitful characterization 
	for the isolated singularities of vector fields 
	(and also for cross fields as we shall see further).
	However it has also a topological significance
	culminating in the Poincar\'e-Hopf theorem (see~\cite{Ray2008,Beaufort} for adaptations of this theorem in the context of cross field).
	This theorem states that
	for a vector field $v$ with isolated singularities
	defined on a closed differentiable manifold $M$
	%and normal to the boundary $\partial M$,
	one has 
	\begin{equation}
	\sum\limits_{i=1}^N \text{index}(P_i) = \chi(M) 
	\label{eq:PoinHopf}
	\end{equation}
	where $\chi(M)$ is the topological Euler-Poincar\'e characteristic 
	of the manifold $M$.
	In dimension 2, the characteristic of $M$ is given by 
	\begin{equation}
	\chi(M) = 2 - 2g - b
	\label{eq:EulerChar}
	\end{equation}
	where $b$ is the number of connected components of the boundary $\partial M$, 
	and $g$ is the genus of the surface, i.e., the maximum number of cuttings along
	non-intersecting closed curves that will {\em not} make the surface disconnected.
	
	\subsection{The topological constraint for planar vector fields on smooth domains}
	
	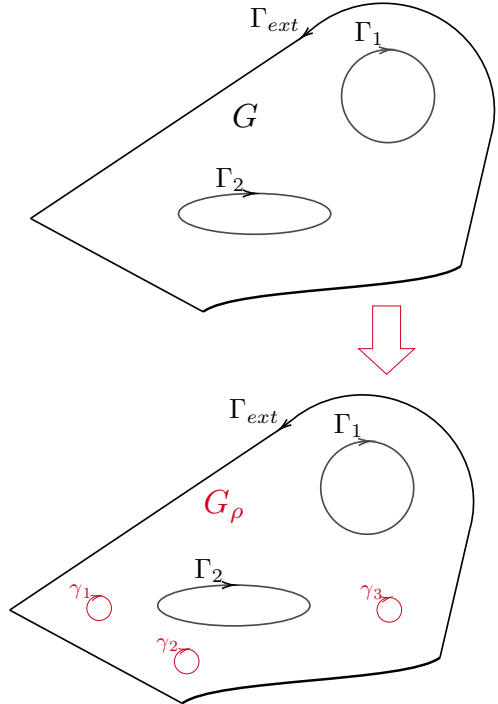
\begin{figure}[ht]
		\centering
		
		\begin{adjustbox}{max width=.5\textwidth}

			\tikzset{every picture/.style={line width=0.75pt}} %set default line width to 0.75pt        
			
			\begin{tikzpicture}[x=0.75pt,y=0.75pt,yscale=-1,xscale=1]
			%uncomment if require: \path (0,941); %set diagram left start at 0, and has height of 941
			
			%Straight Lines [id:da6864071555661981] 
			\draw [color={rgb, 255:red, 0; green, 0; blue, 0 }  ,draw opacity=1 ][line width=1.5]    (59.5,270) -- (266.5,382) ;

			%Curve Lines [id:da3385339632982092] 
			\draw [color={rgb, 255:red, 0; green, 0; blue, 0 }  ,draw opacity=1 ][line width=2.25]    (266.5,382) .. controls (306.5,352) and (535.5,357) .. (575.5,327) ;

			%Straight Lines [id:da12973784124355225] 
			\draw [color={rgb, 255:red, 0; green, 0; blue, 0 }  ,draw opacity=1 ][line width=1.5]    (611.5,172) -- (575.5,327) ;

			%Shape: Arc [id:dp09142027483996673] 
			\draw  [draw opacity=0][line width=1.5]  (399.65,40.6) .. controls (430.86,15.34) and (474.31,4.93) .. (517.31,16.36) .. controls (587.24,34.94) and (629.52,104.2) .. (611.76,171.05) .. controls (611.67,171.37) and (611.59,171.68) .. (611.5,172) -- (485.15,137.41) -- cycle ; \draw  [color={rgb, 255:red, 0; green, 0; blue, 0 }  ,draw opacity=1 ][line width=1.5]  (399.65,40.6) .. controls (430.86,15.34) and (474.31,4.93) .. (517.31,16.36) .. controls (587.24,34.94) and (629.52,104.2) .. (611.76,171.05) .. controls (611.67,171.37) and (611.59,171.68) .. (611.5,172) ;
			%Shape: Circle [id:dp7755034286274565] 
			\draw  [color={rgb, 255:red, 74; green, 74; blue, 74 }  ,draw opacity=1 ][line width=1.5]  (432.5,123.25) .. controls (432.5,92.46) and (457.46,67.5) .. (488.25,67.5) .. controls (519.04,67.5) and (544,92.46) .. (544,123.25) .. controls (544,154.04) and (519.04,179) .. (488.25,179) .. controls (457.46,179) and (432.5,154.04) .. (432.5,123.25) -- cycle ;
			%Shape: Ellipse [id:dp15321652514948836] 
			\draw  [color={rgb, 255:red, 74; green, 74; blue, 74 }  ,draw opacity=1 ][line width=1.5]  (237,264.5) .. controls (237,250.97) and (277.85,240) .. (328.25,240) .. controls (378.65,240) and (419.5,250.97) .. (419.5,264.5) .. controls (419.5,278.03) and (378.65,289) .. (328.25,289) .. controls (277.85,289) and (237,278.03) .. (237,264.5) -- cycle ;
			%Straight Lines [id:da9071678080697864] 
			\draw [color={rgb, 255:red, 74; green, 74; blue, 74 }  ,draw opacity=1 ][line width=1.5]    (488.25,67.5) -- (488.76,67.46) ;
			\draw [shift={(491.75,67.25)}, rotate = 535.9100000000001] [color={rgb, 255:red, 74; green, 74; blue, 74 }  ,draw opacity=1 ][line width=1.5]    (14.21,-4.28) .. controls (9.04,-1.82) and (4.3,-0.39) .. (0,0) .. controls (4.3,0.39) and (9.04,1.82) .. (14.21,4.28)   ;
			
			%Straight Lines [id:da980589938446214] 
			\draw [line width=1.5]    (328.25,240) ;
			\draw [shift={(328.25,240)}, rotate = 180] [color={rgb, 255:red, 0; green, 0; blue, 0 }  ][line width=1.5]    (14.21,-4.28) .. controls (9.04,-1.82) and (4.3,-0.39) .. (0,0) .. controls (4.3,0.39) and (9.04,1.82) .. (14.21,4.28)   ;
			
			%Straight Lines [id:da9066746206299134] 
			\draw [color={rgb, 255:red, 0; green, 0; blue, 0 }  ,draw opacity=1 ][line width=1.5]    (34.5,740) -- (241.5,852) ;

			%Curve Lines [id:da8132787207387806] 
			\draw [color={rgb, 255:red, 0; green, 0; blue, 0 }  ,draw opacity=1 ][line width=2.25]    (241.5,852) .. controls (281.5,822) and (510.5,827) .. (550.5,797) ;

			%Straight Lines [id:da06191397469152138] 
			\draw [color={rgb, 255:red, 0; green, 0; blue, 0 }  ,draw opacity=1 ][line width=1.5]    (586.5,642) -- (550.5,797) ;

			%Shape: Arc [id:dp8674840207869711] 
			\draw  [draw opacity=0][line width=1.5]  (374.9,510.4) .. controls (406.09,485.29) and (449.42,474.96) .. (492.31,486.36) .. controls (562.24,504.94) and (604.52,574.2) .. (586.76,641.05) .. controls (586.67,641.37) and (586.59,641.68) .. (586.5,642) -- (460.15,607.41) -- cycle ; \draw  [color={rgb, 255:red, 0; green, 0; blue, 0 }  ,draw opacity=1 ][line width=1.5]  (374.9,510.4) .. controls (406.09,485.29) and (449.42,474.96) .. (492.31,486.36) .. controls (562.24,504.94) and (604.52,574.2) .. (586.76,641.05) .. controls (586.67,641.37) and (586.59,641.68) .. (586.5,642) ;
			%Shape: Circle [id:dp9911660356173148] 
			\draw  [color={rgb, 255:red, 74; green, 74; blue, 74 }  ,draw opacity=1 ][line width=1.5]  (407.5,593.25) .. controls (407.5,562.46) and (432.46,537.5) .. (463.25,537.5) .. controls (494.04,537.5) and (519,562.46) .. (519,593.25) .. controls (519,624.04) and (494.04,649) .. (463.25,649) .. controls (432.46,649) and (407.5,624.04) .. (407.5,593.25) -- cycle ;
			%Shape: Ellipse [id:dp7406600968332641] 
			\draw  [color={rgb, 255:red, 74; green, 74; blue, 74 }  ,draw opacity=1 ][line width=1.5]  (212,734.5) .. controls (212,720.97) and (252.85,710) .. (303.25,710) .. controls (353.65,710) and (394.5,720.97) .. (394.5,734.5) .. controls (394.5,748.03) and (353.65,759) .. (303.25,759) .. controls (252.85,759) and (212,748.03) .. (212,734.5) -- cycle ;
			%Straight Lines [id:da9505732815995692] 
			\draw [color={rgb, 255:red, 0; green, 0; blue, 0 }  ,draw opacity=1 ][line width=1.5]    (374.9,510.4) -- (360.63,521.28) ;
			\draw [shift={(358.25,523.1)}, rotate = 322.66999999999996] [color={rgb, 255:red, 0; green, 0; blue, 0 }  ,draw opacity=1 ][line width=1.5]    (14.21,-4.28) .. controls (9.04,-1.82) and (4.3,-0.39) .. (0,0) .. controls (4.3,0.39) and (9.04,1.82) .. (14.21,4.28)   ;
			
			%Straight Lines [id:da11633375455314832] 
			\draw [color={rgb, 255:red, 74; green, 74; blue, 74 }  ,draw opacity=1 ][line width=1.5]    (463.25,537.5) -- (463.76,537.46) ;
			\draw [shift={(466.75,537.25)}, rotate = 535.9100000000001] [color={rgb, 255:red, 74; green, 74; blue, 74 }  ,draw opacity=1 ][line width=1.5]    (14.21,-4.28) .. controls (9.04,-1.82) and (4.3,-0.39) .. (0,0) .. controls (4.3,0.39) and (9.04,1.82) .. (14.21,4.28)   ;
			
			%Straight Lines [id:da8120623080076318] 
			\draw [line width=1.5]    (303.25,710) ;
			\draw [shift={(303.25,710)}, rotate = 180] [color={rgb, 255:red, 0; green, 0; blue, 0 }  ][line width=1.5]    (14.21,-4.28) .. controls (9.04,-1.82) and (4.3,-0.39) .. (0,0) .. controls (4.3,0.39) and (9.04,1.82) .. (14.21,4.28)   ;
			
			%Down Arrow [id:dp5381025506863675] 
			\draw  [color={rgb, 255:red, 208; green, 2; blue, 27 }  ,draw opacity=1 ] (453,426.25) -- (469.75,426.25) -- (469.75,375) -- (503.25,375) -- (503.25,426.25) -- (520,426.25) -- (486.5,457) -- cycle ;
			%Shape: Circle [id:dp2803100996202812] 
			\draw  [color={rgb, 255:red, 208; green, 2; blue, 27 }  ,draw opacity=1 ] (127,737.75) .. controls (127,729.6) and (133.6,723) .. (141.75,723) .. controls (149.9,723) and (156.5,729.6) .. (156.5,737.75) .. controls (156.5,745.9) and (149.9,752.5) .. (141.75,752.5) .. controls (133.6,752.5) and (127,745.9) .. (127,737.75) -- cycle ;
			%Straight Lines [id:da42839974604961917] 
			\draw [color={rgb, 255:red, 208; green, 2; blue, 27 }  ,draw opacity=1 ]   (141.75,723) ;
			\draw [shift={(143.05,722.65)}, rotate = 524.9300000000001] [color={rgb, 255:red, 208; green, 2; blue, 27 }  ,draw opacity=1 ][line width=0.75]    (10.93,-3.29) .. controls (6.95,-1.4) and (3.31,-0.3) .. (0,0) .. controls (3.31,0.3) and (6.95,1.4) .. (10.93,3.29)   ;
			
			%Shape: Circle [id:dp20970497692029022] 
			\draw  [color={rgb, 255:red, 208; green, 2; blue, 27 }  ,draw opacity=1 ] (475,739.75) .. controls (475,731.6) and (481.6,725) .. (489.75,725) .. controls (497.9,725) and (504.5,731.6) .. (504.5,739.75) .. controls (504.5,747.9) and (497.9,754.5) .. (489.75,754.5) .. controls (481.6,754.5) and (475,747.9) .. (475,739.75) -- cycle ;
			%Straight Lines [id:da3980521273077636] 
			\draw [color={rgb, 255:red, 208; green, 2; blue, 27 }  ,draw opacity=1 ]   (489.75,725) ;
			\draw [shift={(491.05,724.65)}, rotate = 524.9300000000001] [color={rgb, 255:red, 208; green, 2; blue, 27 }  ,draw opacity=1 ][line width=0.75]    (10.93,-3.29) .. controls (6.95,-1.4) and (3.31,-0.3) .. (0,0) .. controls (3.31,0.3) and (6.95,1.4) .. (10.93,3.29)   ;
			
			%Shape: Circle [id:dp2888148214077172] 
			\draw  [color={rgb, 255:red, 208; green, 2; blue, 27 }  ,draw opacity=1 ] (232,801.75) .. controls (232,793.6) and (238.6,787) .. (246.75,787) .. controls (254.9,787) and (261.5,793.6) .. (261.5,801.75) .. controls (261.5,809.9) and (254.9,816.5) .. (246.75,816.5) .. controls (238.6,816.5) and (232,809.9) .. (232,801.75) -- cycle ;
			%Straight Lines [id:da7090594227312668] 
			\draw [color={rgb, 255:red, 208; green, 2; blue, 27 }  ,draw opacity=1 ]   (246.75,787) ;
			\draw [shift={(248.05,786.65)}, rotate = 524.9300000000001] [color={rgb, 255:red, 208; green, 2; blue, 27 }  ,draw opacity=1 ][line width=0.75]    (10.93,-3.29) .. controls (6.95,-1.4) and (3.31,-0.3) .. (0,0) .. controls (3.31,0.3) and (6.95,1.4) .. (10.93,3.29)   ;
			
			%Straight Lines [id:da7023840931170456] 
			\draw [line width=1.5]    (358.25,523.1) -- (34.5,740) ;

			%Straight Lines [id:da8657832266113327] 
			\draw [color={rgb, 255:red, 0; green, 0; blue, 0 }  ,draw opacity=1 ][line width=1.5]    (399.65,40.6) -- (386.58,51.21) ;
			\draw [shift={(384.25,53.1)}, rotate = 320.94] [color={rgb, 255:red, 0; green, 0; blue, 0 }  ,draw opacity=1 ][line width=1.5]    (14.21,-4.28) .. controls (9.04,-1.82) and (4.3,-0.39) .. (0,0) .. controls (4.3,0.39) and (9.04,1.82) .. (14.21,4.28)   ;
			
			%Straight Lines [id:da8355558672158654] 
			\draw [line width=1.5]    (384.25,53.1) -- (59.5,270) ;

			% Text Node
			\draw (353,31) node [scale=2.5]  {$\textcolor[rgb]{0,0,0}{\Gamma }\textcolor[rgb]{0,0,0}{_{ext}}$};
			% Text Node
			\draw (317,146) node [scale=3,color={rgb, 255:red, 0; green, 0; blue, 0 }  ,opacity=1 ]  {$\textcolor[rgb]{0,0,0}{G}$};
			% Text Node
			\draw (299,223) node [scale=2.5]  {$\Gamma _{2}$};
			% Text Node
			\draw (466,48) node [scale=2.5]  {$\Gamma _{1}$};
			% Text Node
			\draw (328,501) node [scale=2.5]  {$\textcolor[rgb]{0,0,0}{\Gamma }\textcolor[rgb]{0,0,0}{_{ext}}$};
			% Text Node
			\draw (292,616) node [scale=3,color={rgb, 255:red, 0; green, 0; blue, 0 }  ,opacity=1 ]  {$\textcolor[rgb]{0.82,0.01,0.11}{G}\textcolor[rgb]{0.82,0.01,0.11}{_{\rho }}$};
			% Text Node
			\draw (274,693) node [scale=2.5]  {$\Gamma _{2}$};
			% Text Node
			\draw (441,518) node [scale=2.5]  {$\Gamma _{1}$};
			% Text Node
			\draw (120,718) node [scale=2]  {$\textcolor[rgb]{0.82,0.01,0.11}{\gamma }\textcolor[rgb]{0.82,0.01,0.11}{_{1}}$};
			% Text Node
			\draw (224,782) node [scale=2]  {$\textcolor[rgb]{0.82,0.01,0.11}{\gamma }\textcolor[rgb]{0.82,0.01,0.11}{_{2}}$};
			% Text Node
			\draw (469,719) node [scale=2]  {$\textcolor[rgb]{0.82,0.01,0.11}{\gamma }\textcolor[rgb]{0.82,0.01,0.11}{_{3}}$};

			\end{tikzpicture}
			
		\end{adjustbox}
		\caption{Illustration of $G$, $\Gamma_{ext}$, the $\Gamma_i$ and the $\gamma_j$.}
		\label{fig:Gfrontieres}
	\end{figure}

	The Poincaré-Hopf theorem will be developed so that only degrees and indices will be paired, taking advantage of the fact that we work in planar.
	This will lead to the definition of a global constraint that we will use to construct our method for placing singularities within a domain.
	
	The  Poincar\'e-Hopf theorem highlights a profound link existing 
	between a purely analytical concept
	(the zeros of a vector field)
	and a purely topological one 
	(the Euler-Poincar\'e characteristic of the domain of definition of that vector field).
	This link is extremely important in practice
	as it restricts the set of configurations (degrees and positions) of singularities that a surface $M$ can contain based on its Euler-Poincar\'e characteristic $\chi(M)$.
	The theorem is however only applicable to closed surfaces,
	or, if the surface has a boundary
	with fields $v$ normal to that boundary.
	For the purposes of this paper,
	a new interpretation of the generalizations of (\ref{eq:PoinHopf}), valid for wedge-shaped geometries, is developed later~\ref{GenPH}.
	
	Let $G \subset \mathbb{R}^2$ be a planar region, whose boundary
	$$
	%\partial G = \Gamma_{\text{ext}} \cup  \Gamma_1 \cup \cdots \cup \Gamma_{N_{\text{int}}} 
	\partial G = \Gamma_{\text{ext}} - \Gamma_1 - \cdots - \Gamma_{N_{\text{int}}} 
	$$ 
	is composed of one external boundary $\Gamma_{\text{ext}}$
	oriented counter-clockwise,
	and of $N_{\text{int}} = b-1$ internal boundaries $\Gamma_k$
	oriented clockwise (see Figure~\ref{fig:Gfrontieres}).
	The genus of a planar surface being zero, 
	(\ref{eq:PoinHopf}) writes now
	\begin{equation}
	\sum\limits_{i=1}^N \text{index}(P_i) = 2- b = 1-N_{\text{int}}
	\end{equation}
	for planar surfaces.
	On these surfaces, the Brouwer degree of a vector field $v$
	can be defined for any closed curve $\gamma$, even non-infinitesimal. 
	Indeed, a parametrisation of the curve and a global coordinate system then
	allows representing the vector field on any curve as a $S^1 \mapsto S^1$ mapping.  
	The Brouwer degree $\text{deg}(v, \gamma)$  is zero 
	if $v$ is smooth in the interior of $\gamma$,
	and it is otherwise equal to the sum of the indices 
	of the singularities enclosed in $\gamma$.
	
	For smooth boundaries, if
	the vector field $v$ is normal to $\partial G$ 
%	in accordance to the hypothesis of the Poincar\'e-Hopf theorem,
	one can then evaluate the degrees $\text{deg}(v, \Gamma_{\text{ext}}) =1$
	and $\text{deg}(v, \Gamma_k)= -1$ for  $k=1, \dots, N_{\text{int}}$,
	so that the relationship
	\begin{equation}
	\text{deg}(v, \Gamma_{\text{ext}}) 
	= \sum\limits_{i=1}^N \text{index}(P_i) 
	+ \sum_{k=1}^{N_{\text{int}}} \text{deg}(v, \Gamma_k)
	\label{eq:IndexVecField}
	\end{equation}
	holds as an alternative expression 
	of the Poincar\'e-Hopf theorem for planar vector fields (see Figure \ref{fig:Gfrontieres}). 
	
	It has been proven by Bethuel et al.~\cite{BBH} that this relationship 
	remains valid if the field $v$ 
	has arbitrary degrees $\text{deg}(v, \Gamma_{\text{ext}})$ and $\text{deg}(v, \Gamma_k)$
	on the boundary of $G$. 
%	We will see how this generalized topological constraint
%	can also be used for the computation of cross fields. 

	\subsection{Planar cross field and Ginzburg-Landau functional}
	
%	As we said previously, 2D-cross is a particular case of a $n$-symmetry directional field, with $n=4$~\cite{Ray2008}
%	that can be seen as a quadruple of unit norm vectors forming a regular cross,
%	i.e., that are either orthogonal or opposite to each other.
%	A 2D-cross field $c$ on a planar discretize region $G$ is a function 
%	that associates a cross $c(P)$ to every point $P \in G$.
%	The four vectors composing a cross $c(P)$ can be regarded as 
%	the $4^{th}$ roots of a unique complex number $u(P)= c^4(P)$ of unit norm~\cite{Beaufort}, 
%	which is in general not parallel to any of the branches of the cross it represents. 
%	A smooth cross field can thus be represented unambiguously 
%	as a continuous $u\,: G \mapsto S^1$ complex function.
%	The cross field $c$ is then the field of the $4^{th}$ roots of $u$. 

	Let us present the mathematical theory of Ginzburg-Landau adapted to the cross fields.
	 
	We will consider any Euler-Poincaré characteristic and require that one branch of each boundary cross
	is parallel to the local outer normal $\mathbf n$ as boundary conditions.
	Note that this condition shall not be respected on corners of piecewise smooth boundary but the implication in terms of degrees will be tackled via the definition of boundary singularities on such corners.
	For regular boundary crosses (that are present almost everywhere on our piecewise smooth domains of interest), the boundary condition writes
	\begin{equation}
	c(P)\cdot \mathbf n \in \{-1, 0, 1\}, \forall P \in \{\Gamma_{\text{ext}}, \Gamma_{1}, \ldots, \Gamma_{N_{\text{int}}} \},
	\label{eq:boundarycond}
	\end{equation}
	i.e., two branches of a unit boundary cross are aligned with the boundary and the other two branches of this unit boundary cross are opposite and parallel to the local unit outward normal to the boundary. 
	% \comment{Noter l'interprétation!}
	
	Let's temporary assume that we are working on smooth domains.
	If $g_c\,: \{\Gamma_{\text{ext}}, \Gamma_{1}, \ldots, \Gamma_{N_{\text{int}}} \} \rightarrow S^1$ 
	is the boundary condition, smooth on each boundary $\Gamma_i$, then,
	one observes that its Brouwer degree is fixed, $\deg(g_c, \{\Gamma_{\text{ext}}, \Gamma_{1}, \ldots, \Gamma_{N_{\text{int}}} \}) = 1 - N_{\text{int}}$.
	
	However, the branches of a cross are redundant
	as it is enough to know any one of them to reconstruct the whole cross.
	At the same time, the representation of a cross by one of its branches is also ambiguous
	as there exist no systematic rule 
	to select the representative branch of a cross.
	For these reasons, it is convenient to solve
	indirectly for the field $u=c^4$.
	The $4^{th}$ power maps all four branches of the cross
	onto a unique complex number,
	which thus offers a smooth and unambiguous representation for the cross.
	The boundary condition becomes for this vector field
	$g = g_c^4\,: \Gamma_{\text{ext}} \rightarrow S^1$,
	with the Brouwer degree 
	\begin{equation}
	\deg(g, \{\Gamma_{\text{ext}}, \Gamma_{1}, \ldots, \Gamma_{N_{\text{int}}} \}) = 4 - 4 N_{\text{int}}.
	\label{eq:brouwer4}
	\end{equation}
	
	%FH inutile de faire cette distinction 
	% \begin{equation}
	% 	\min\limits_{H^1_g(G, S^1)} \int_G |\nabla u|^2
	% 	\label{eq:Dirichlet}
	% \end{equation}
	% We note that the smoothness of the boundary is crucial for the previous argument, 
	% as it may not be valid if $\partial G$ is only continuous (for example, in a square, see \cref{PHGen}).
	% However, as mentioned in the previous section, there are topological obstructions which may imply that $H^1_g(G, S^1) = \emptyset$. The following relaxation is proposed in~\cite{BBH}
	
	Consequently, to compute a cross field aligned with the domain boundaries and minimizing the Ginzburg-Landau functional, one is led to solve the variational problem
	\begin{equation}
	\min\limits_{H^1_g(G, \mathbb{R}^2)} \int_G |\nabla u|^2 + \frac{1}{4 \epsilon^2} (|u|^2-1)^2.
	\label{eq:GL}
	\end{equation}
	with %$\deg(g, \Gamma_{\text{ext}}) = 4$ and 
	$H_g^1 = \{ u \in H^1(G, \mathbb{R}^2); \; u=g \text{  on  } \{\Gamma_{\text{ext}}, \Gamma_{1}, \ldots, \Gamma_{N_{\text{int}}} \} \}$.
	The gradient term ensures smoothness,
	while the penalty term ensures $|u|$ to be as close to 1 as possible.
	The minimized term of \eqref{eq:GL} is called the Ginzburg-Landau energy of the vector field $u$.
	If $\epsilon$ is small, the minimizers $u_*$ will therefore be close to 1
	outside a small set of isolated singularities, whose cumulated indices sum up to $4 - 4N_{\text{int}}$. 
	Mathematically, when we minimize~\eqref{eq:GL} as $\epsilon \rightarrow 0$ there is talk of asymptotic minimizers of the energetic Ginzburg-Landau functional~\eqref{eq:GL}.
	In the discrete model, singularities are regions of size $\epsilon$ 
	where $|u|$ deviates significantly from 1 towards zero. 
	Note that an adaptation of this energy for two-dimensional manifolds that only implies linear added terms (linked to the Christoffel symbols) has been proposed in~\cite{Macq2018}. 
	
	Now (\ref{eq:IndexVecField}) and (\ref{eq:brouwer4}) imply that 
	\begin{equation}
	\sum\limits_{i=1}^N \text{index}(P_i) 
	= \text{deg}(v, \{\Gamma_{\text{ext}}, \Gamma_{1}, \ldots, \Gamma_{N_{\text{int}}} \})  = 4 - 4 N_{\text{int}}.
	\label{eq:fourSing}
	\end{equation}
	If we count in the points $P_i$ the internal singularities and also the boundary singularities, the formula \eqref{eq:fourSing} holds for domains with piecewise smooth boundary.
	
	Functionals of the form (\ref{eq:GL}) were originally introduced by Ginzburg and Landau 
	in the study of the phase transition problems occurring in superconductivity (see an overview of their research in~\cite{Ginzburg2009}).
	
	The approach of Beaufort et al. (see~\cite{Beaufort}) sees the optimal singularity configurations as the zeros and associated degrees of the minimizers of $E_\epsilon$ for $\epsilon$ small enough.
	Unfortunately, this formulation can not numerically ensure that only a minimum number of singularities, necessary to fulfill \eqref{eq:PoinHopf}, appears and thus is not always able to reproduce the asymptotic mathematical results obtained in~\cite{BBH}.
	
	\begin{figure}[H] 
		\begin{minipage}[b]{.5\linewidth}
			\centering
			\includegraphics[width=\linewidth]{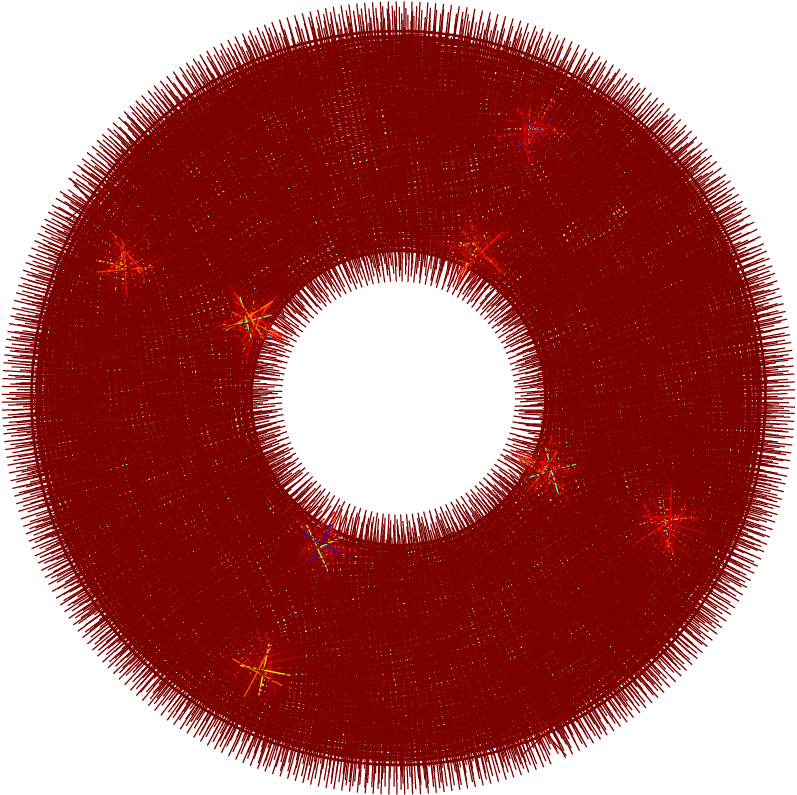} 
		\end{minipage}%%
		\begin{minipage}[b]{.5\linewidth}
			\centering
			\includegraphics[width=\linewidth]{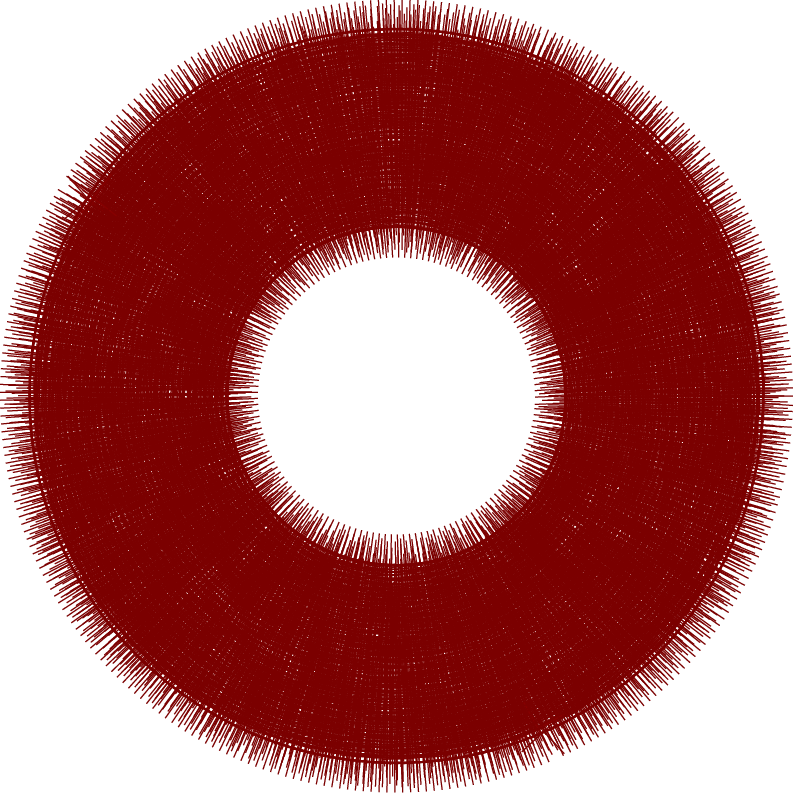} 
		\end{minipage}%% 
		\caption{Here are some results obtained via the method described in~\cite{Beaufort} based on the Ginzburg-Landau functional \eqref{eq:GL} for rings with different ratios of the small radius $r_1$ to the large radius $r_2$. For a parameter  $\epsilon = 1/100$, we obtain, on the left figure (ring with $r_1/r_2 = 0.4$), 8 singularities (four of degree equals to 1/4 and four of degree equals to -1/4) in our cross field and on the right figure (ring with $r_1/r_2 = 0.46$) zero singularity. For a given $\epsilon$, when the ring becomes too thick, singularities appear.}
		\label{fig:ringRatio} 
	\end{figure}
	
	The ring is an example where this phenomenon can occur (see figure \ref{fig:ringRatio}). 
	Although its Euler-Poincar\'e characteristic is $0$, numerically, singularities appear if the ratio of the small radius to the great radius is lower or equal to $0.4$ for $\epsilon = 1/100$ (see figure \ref{fig:ringRatio}).
	With a smaller $\epsilon$ we may have converged to the radial solution with no singularities. However, the use of $\epsilon$ is a problem in itself: the $\epsilon$ needed to obtain a numerical result analogue to the asymptotic mathematical minimizer of \eqref{eq:GL} depends strongly on the domain and must possibly be very small, which may be hard to work with numerically.
	
	On the other hand, the supplementary singularities can be viewed as an advantage for the quadrangular mesh generation because they can be used to limit the size distortions of an elementary quad cutting.
	Let us compare the two rough quad components of Figure \ref{fig:meshedRing}. 
	For thick ring, the radial quad decomposition of the right figure of \ref{fig:meshedRing} has great size difference between its elements. On the contrary, the singularities of opposite sign in the left figure of \ref{fig:meshedRing} enable to thwart those deformations.
	
	\begin{figure}[H] 
		\begin{minipage}[b]{.5\linewidth}
			\centering
			\includegraphics[width=\linewidth]{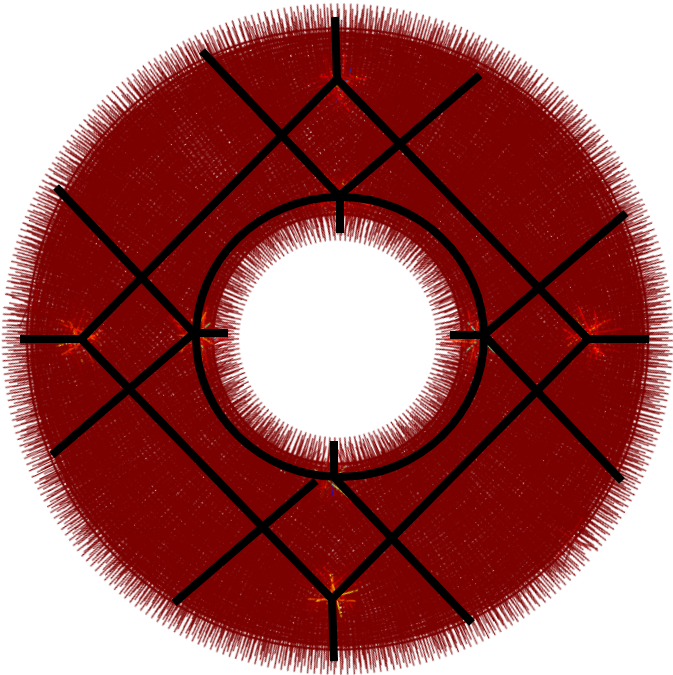} 
		\end{minipage}%%
		\begin{minipage}[b]{.5\linewidth}
			\centering
			\includegraphics[width=\linewidth]{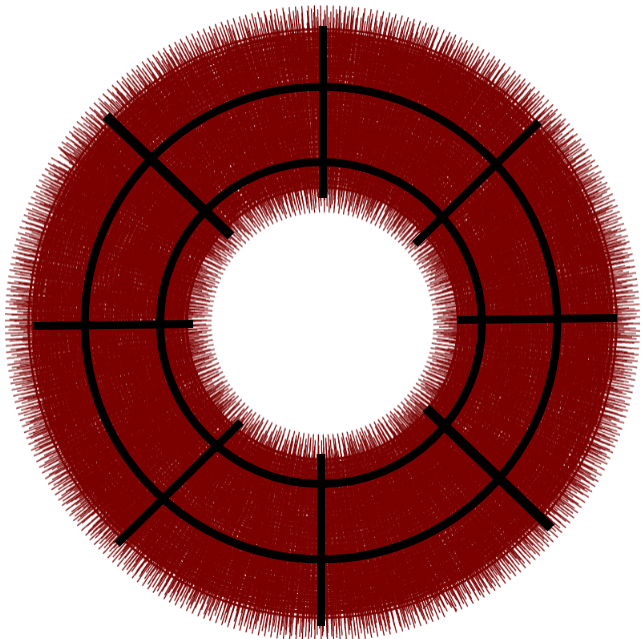} 
		\end{minipage}%% 
		\caption{Obtained meshes via rough directions propagation (priority is given to propagating the directions adjacent to the singularities) for the cross fields of figure \ref{fig:ringRatio}. For a parameter  $\epsilon = 1/100$, we obtain, on the left figure (ring with $r_1/r_2 = 0.4$), a quad mesh with less deformations (especially if we refine it) than the quad mesh of the right figure (ring with $r_1/r_2 = 0.46$) that contains no singularity.}
		\label{fig:meshedRing} 
	\end{figure}
	Judging which of the two cross fields is best from a numerical point of view depends on the properties sought after - for example, emphasis on preferential directions, or, on the contrary, search for uniformity of the elements. Furthermore, the addition of singularities should potentially make it possible to design mesh size for the block-structured quadrangular meshes resulting from the computed cross fields.
	
	Here we have a working but rigid method to construct cross fields on surfaces. In the section \ref{MaterialsAndMethods} we will develop an alternative method in order to choose the configurations of inner singularities.
	
	\subsection{How to evaluate the choice of a singularity location}\label{SingEvaluation}
	
	To find the optimal singularity configuration for various radius of the singular holes with the method presented in \ref{MaterialsAndMethods}, we propose the following energy that is strongly linked with the Ginzburg-Landau energy.
	This energy is the energy of the scalar field that is solution of the following problem:
	\begin{equation}
	\begin{cases}
	-\Delta \Phi = 0 &\text{  in  } G_\rho,\\
	\Phi = \text{Const.} = C_i &\text{  on  } \gamma_i, \; i = 1, 2, \ldots N,\\
	\Phi = 0 &\text{ on } \partial G,\\
	\int_{\gamma_i} \frac{\partial \Phi}{\partial \nu} = 2\pi d_i & i = 1, 2, \ldots, N, \label{eq:Phi}
	\end{cases}
	\end{equation}
	where $G_\rho$ is the domain $G$ in which we have drilled holes of radius $\rho$ with boundaries $\gamma_i$ such as $\chi(G_\rho) = 0$ and where $\nu$ is the outward normal to the $\gamma_i$ (see Figure \ref{fig:Gfrontieres} for an illustration). Here the $C_i$ are not given but are unknown constants that are part of the problem (see~\cite{BBH}).  
	
	Theorem I.1 in~\cite{BBH} tells us that 
	\begin{equation*}
	E_\rho = \int_{G_\rho} | \nabla \Phi|^2 = \inf\limits_{v \in H^1(\mathbb{R}^2,G_\rho)} \int_{G_\rho} |\nabla v|^2.
	\end{equation*}
	
	That means that the solution of the proposed problem have the same energy as the one of the regular cross field proposed in~\ref{MaterialsAndMethods}.
	
	As the problem \eqref{eq:Phi} is a simple linear Neumann problem, it gives a simple way to cheaply compares singularities configurations that could be used with the main method of this paper that will be presented in the beginning of the section \ref{MaterialsAndMethods}.
	
	But, if we are not interested in the impact of the singularity radius or if we let those radius tend to zero, we could even compare the singularity in term of renormalized Ginzburg-Landau energy by direct computation.

	\subsubsection{The unifying renormalized Ginzburg-Landau energy}\label{GLEnergies}
	
	As we take the $\gamma_i$ as being disk of radius $\rho$, we have
	\[
	\lim\limits_{\epsilon \rightarrow 0} \min\limits_{H^1_g(G, \mathbb{R}^2)}  \frac{1}{2} \int_G | \nabla v |^2 + \frac{1}{4 \epsilon^2}  \int_G (|v|^2 - 1)^2 =
	\lim\limits_{\rho \rightarrow 0} \inf\limits_{v_\rho \in \mathcal{E}} \int_{G_\rho} |\nabla u|^2 =  + \infty
	\]
	where
	\begin{equation*}
	\mathcal{E} = \left\lbrace v \in H^1(G_\rho; S^1) \left| \begin{aligned}
	v = g \text{  on  } \partial G \text{    and  }\\
	\deg(v, \gamma_i) = 1 \; \; \forall i
	\end{aligned} \right. \right\rbrace.
	\end{equation*}
	But as explained in~\cite{BBH}, there is a way to compare energetically the two approaches corresponding with those energies by removing the infinite core energy from them. The resulting common energy is called \textit{renormalized} energy. 
	
	Theorem I.2 and Theorem I.7 of~\cite{BBH} prove that there exists a unique minimizer $u_\rho$ for the problem
	\begin{equation}
	\min\limits_{u \in \mathcal{E}} \int_{G_\rho} | \nabla u_\rho |^2
	\end{equation}
	and that the following expansion holds:
	\begin{equation*}
	\frac{1}{2} \int_{G_\rho} | \nabla u_\rho|^2 = \pi d |\log \rho| + W(\{P\}) + O(\rho) \text{  as } \rho \rightarrow 0
	\end{equation*}
	where $W$, called the renormalized energy, stays bounded as $\rho \rightarrow 0$ and where $\{P\} = (P_1, P_2, \ldots P_d)$ is the configuration of the centers of the circles $\gamma_i$. We call singular ``core energy" $\pi d |\log \rho|$. Other examples of removing singular core energy appear in physics (see ~\cite{Kleman} for another example).
	
	The renormalized energy $W$ does not depend on $\rho$ and is the same as the renormalized energy of the energetic Ginzburg-Landau functional.	
	More specifically, it is possible to computed the common renormalized energy $W$ via a set of point locations and degrees rather than via a function such as in \eqref{eq:GL}.
	Indeed the positions and degrees of the singularities of the minimizing field of the energetic Ginzburg-Landau functional, which is defined for maps $u\in H^1(G; \mathbb{C})$,
	\[
	E_\epsilon (u) = \frac{1}{2} \int_G |\nabla u|^2 + \frac{1}{4\epsilon^2} \int_G (|u|^2 -1)^2
	\]
	or the positions and degrees of the singularities that minimize
	\[
		\min\limits_{u \in \mathcal{E}} \int_{G_\rho} | \nabla u_\rho |^2
	\]
	are the same as the positions of the points of the
	configuration $b = (b_1, b_2, \ldots, b_d)$ that minimize the renormalized Ginzburg-Landau energy, as $\epsilon \rightarrow 0$ and if $d = \deg(g, \partial G)$,\\
	\[
	W(b) = - \pi \sum\limits_{i\neq j} \log |b_i - b_j| + \frac{1}{2} \int_{\partial G} \Phi(g \times g_\tau) - \pi \sum\limits_{i=1}^d R(b_i)
	\]
	where $\Phi$ is the solution of the linear Neumann problem
	\[
	\begin{cases}
	\Delta \Phi = 2\pi \sum\limits_{i=1}^d \delta_{b_i} & \text{ in } G,\\
	\frac{\partial \Phi}{\partial \nu} = g \times g_\tau & \text{ on } \partial G,
	\end{cases}
	\]
	where $\nu$ is the outward normal to $\partial G$ and $\tau$ is a unit tangent vector to $\partial G$ such that $(\nu, \tau)$ is direct, the $\delta_{b_i}$ corresponds to the singularity degrees and
	\[
	R(x) = \Phi(x) - \sum\limits_{i=1}^{d} \log|x-b_i|.
	\]
	
	This renormalized energy teaches us that whether a formulations uses holes with a certain radius or the factor epsilon, they all can be compared in terms of the asymptotic Ginzburg-Landau energy by providing their singularities positions and degrees.
	Indeed, those results show us that the minimal configuration of holed singularities from \eqref{eq:DirichletHoled} as $\rho \rightarrow 0$ is the same as the configuration of the zeros of the minimizers of \eqref{eq:GL}. 
	Furthermore, they give a way to compare different choices of singularity configuration in terms of renormalized Ginzburg-Landau energy.
	In particular, Theorem 0.2 of~\cite{BBH} proves that the zeros of minimizers of \eqref{eq:GL} converge to minimizers of $W$.
	
	A scalar field adapted from $\Phi(x)$ to work with cross field is solved in the section \ref{FurThou} and can be used to compute the proposed renormalized energy. One element that makes this a particularly simple problem is that the equivalent of our holes materializes as dots and the equivalent of our degrees as deltas.
	
	Without numerical computation, the analysis of the renormalized energy $W$ teaches us the following things that could be use to compare singularity configurations. 
	First, $W \rightarrow +\infty$ as two singularities $P_i$, $P_j$ of both positive or negative degrees coalesce
	Secondly, $W \rightarrow + \infty$ as a $P_i$ tends to $\partial G$. 
	Third, $W$ decrease as one singularity of positive degree and one singularity of negative degree coalesce.
	Those properties imply that numerically singularities of configuration minimizing the Ginzburg-Landau energies (the renormalized ones and the one of \eqref{eq:GL}) will rather take position moderately far to concave boundaries, more far from convex boundaries and will repulse or attract each other following their degree signs.
	
	While the classical energetic Ginzburg-Landau functional involves finding a function representing a vector field that minimizes a functional, the second consists only in finding a configuration of points as a minimizer. Moreover, the calculation of $\Phi$ in the renormalized energy can be done for any configuration of singularities even though it would not minimize the Ginzburg-Landau renormalized energy. It is therefore a major lead for efficient cross field computations.
	
	\subsection{Pinning strategy}
	
	Let us now consider another way of looking at the imposition of singularities. This kind of method is of great interest in physics but is less interesting in the context of cross fields, as we will explain briefly.
	
	The pinning strategy consists in modifying the Ginzburg-Landau functional in order to pin singularities in desired places. We can either change the non-linear term of this functional or add a weight that varies on the domain in front of the gradient term. In physics this type of approach is called pinning. For more information on changes on the penalty term see~\cite{Rubinstein1995} and~\cite{Constraint1999} and on weight additions see~\cite{Andre1998}.
	
	Numerically, the problem is that those methods are even hard to solve than the method~\cite{Beaufort}.
	
	For cross fields, a classical formulation for their representation vector could be the following. We are looking for $u \in H^1(G, \mathbb{R}^2)$ that minimizes the functional
	\begin{equation*}
	\int_G |\nabla u|^2 + \frac{1}{4 \epsilon^2} \int_G \left( f(x) - |u|^2 \right)^2
	\end{equation*}
	with $\deg(u,\partial G) = 4$ and with
	\begin{equation*}
	f(x) = \exp\left(\|( \sqrt{x^2 + y^2} - 0.5)\|\right)
	\end{equation*}
	for a pin on a circle centered at the origin of radius $0.5$.
	
	\subsection{Quad mesh generation from cross fields}
	
	There are a lot of different methods for extracting quadrangular meshes from cross fields. Among the numerous utilizations, we can highlight 3 interesting approaches. The first one consists in generating streamlines from singularities and to use the underlying cutting to compute quad mesh (see, as examples,~\cite{Campen2017} and~\cite{Pietroni}). The second works with parametrization via scalar fields and mixed integer solvers (see, as examples,~\cite{Bommes} and~\cite{Pellenard}). The third approach uses Morse parameterization hybridization (see~\cite{Fang}).
	
	\section{Materials and methods}\label{MaterialsAndMethods}
	
	We will begin by presenting our method for placing singularities in a domain. We will see that this method does not allow us to explicitly set boundary singularities. The singularities that will appear on the boundaries will in fact be the ones that minimize a new energy that we construct in section~\ref{GLPiecewise}.
	This will imply that it is not easy to know how to respect the Poincaré-Hopf theorem since this theorem takes into consideration both inner and boundary singularity degrees.
	For the practical use of our method, we will propose a new interpretation of the generalizations of the Poincare-Hopf theorem in section~\ref{GenPH}.
	
	\subsection{Singularity placements via perforation strategy}\label{SingPlacement}
	
	As we have seen, regular vector fields can be computed on any surface as long as we have drill holes with prescribed degrees on their boundaries such as to respect the Poincar\'e-Hopf theorem. Furthermore, the analysis of Ginzburg-Landau energy have lead to simplified formulations that allows to have non-minimal number and non-minimizing (in terms of Ginzburg-Landau energy) configuration of singularities (see~\cite{BBH}).
	
	\subsubsection{Description of the problem} \label{mathprob}
	
	The first idea developed in this paper,
	following a similar approach presented in~\cite{BBH},
	is that the unavoidable singularities could be trapped in tiny holes, 
	i.e., excluded from the computational domain. 
	In order to do so, suitable conditions must be prescribed
	on the boundaries of these holes. 
	One would then be left with the computation of a regular (singular free) field that is numerically cheap and stable.
	
	% Those conditions imply that strictly less or strictly more than 4 directions reach them, 
	% depending on the type of singularities desired. 
	% For the sake of brevity, these holes will be referred to as ``singularities". Nonetheless, it does not imply any vanishing of the considered cross field. It has to be mentioned that our generated cross fields may be regular and smooth.
	%an alternative relaxation approach to the penalty term of \eqref{eq:GL} is given. 
	
	Let thus $B(P_i, \rho)$ be $N$ small holes drilled in the domain $G$,
	$\gamma_i = \partial B(P_i, \rho)$,
	and $G_\rho = G \backslash \bigcup_{i} B(P_i, \rho)$ 
	be the computational domain excluding the holes. 
	The points $P_1, P_2, \ldots, P_N $ at which the holes are centered 
	can be placed freely.
	
	% are free to be moved despite some configurations having lower Dirichlet energy on the modified domains.
	% The singularities of \eqref{eq:GL} of same sign repel each other and pairs of singularities of opposite sign tend to coalesce, but the boundary conditions with $d \neq 0$ on $\partial G$ produce a confinement effect (see~\cite{BBH}). Although this is not a limitation of our method, in practice we will rarely work with singularities of degree greater than 1 for one singularity of degree $m > 1$ carries more energy than $m$ singularities
	% of degree $1$.
	
	In~\cite{BBH} there is a proof of the existence of minimizers $v_\rho$ for the problem
	\begin{equation}
	\min\limits_{v \in \mathcal{E}} \int_{G_\rho} | \nabla u |^2.
	\label{eq:DirichletHoled}
	\end{equation}
	where
	\begin{equation*}
	\mathcal{E} = \left\lbrace v \in H^1(G_\rho; S^1) \left| \begin{aligned}
	v = g \text{  on  } \partial G \text{    and  }\\
	\deg(v, \gamma_i) = 1 \; \; \forall i
	\end{aligned} \right. \right\rbrace.
	\end{equation*}
	
	The formulation we will use to place singularities will be slightly different than \eqref{eq:DirichletHoled}.
	For the sake of clarity, the differences are detailed in the following bulleted list.
	\begin{itemize}
		\item{The function $g$ will be such that $u$ has two of its $4^{th}$ roots parallel to $\partial G$ so that the associated cross fields are tangent with $\partial G$.}
		\item{The used vector field will belong to $\mathbb{R}^2$ in the domain rather than $S^1$ for cheap implementation. However, the boundary representation vectors must not divert too much from $S^1$ in order to preserve a well-defined notion of degree. If this notion is not well enough preserved, the computation must be performed with an added constraint on their norms.}
		\item{In the same way as for the previous bullet, we will suppose that $u \in S^1$ on the circles $\gamma_i$ and use the same notion of degree.}
		\item{We will work on connected piecewise smooth bounded domain $G_\rho$.}
		\item{The configurations of $\gamma = (\gamma_1, \gamma_2, ..., \gamma_p)$ will be such as $d = \sum_{i=1}^p \deg(u, \gamma_i)$ but without imposing that $p = d$ with $d = \deg(g, \partial G)$ for smooth domains. Furthermore, $d$ will be equal to a slight generalization of the Brouwer degree for piecewise smooth domains that we will give later on (see section \ref{GenPH}).}
	\end{itemize}

	Note that we will now have holes in our fields instead of the singularities of the fields composed of unitary elements but we will keep the name ``singularity" since these holes always mimic classical singularities and correspond also to singular nodes (nodes with strictly less or more than 4 neighbors) in the meshes derived from these fields.
	
	Three advantages associated with this approach are worth mentioning.
	Firstly, unlike what is shown in~\cite{Beaufort} and~\cite{Viertel}, a parameter $\epsilon$, whose value is not clear and strongly dependent on the problem considered, is not needed anymore. 
	Note that the $\rho$ parameter has taken the place of the $\epsilon$ parameter but our problem stays more stable as long as we do not impose $\rho$ to be very small.
	Secondly, the use of holes to replace zero norm crosses as singularities in $G$ implies that the computed field may be non-singular (or regular) and smooth. 
	The third advantage is the ability to better adapt to the user's needs and demands for quadrangular meshing.
	
	However, one drawback of this approach is that, numerically, the work is performed on reference vectors belonging to $\mathbb{R}^2$, while the associated mathematical results have been obtained for vectors belonging to $S^1$. We will then have a disparity on the norms of the crosses in compensation for the rotations induced in the computed cross fields.

	\subsubsection{Discrete formulation and implementation}  \label{implem}
	
	Numerically, calculations are performed in the domain via Crouzeix-Raviart finite elements (see~\cite{Crouzeix1973} for more details) on a triangular mesh generated via GMSH (see~\cite{GMSH} for more details). It is then a question of building a finite element system. Its matrix of stiffness corresponds to a Laplacian within the domain. At the same time, on each node at the boundary, we construct a cross having two directions tangent to one of the two edges adjacent to the node. Note that the implementation was done in Python using the GMSH python API which can generate the Jacobians and shape functions of the initial triangular mesh.
	
	\subsubsection{Definition of the constraint for holed singularities}
	
	Let $u \in H^1_g(G_\rho, \mathbb{R}^2)$ where $G$ is a piecewise smooth bounded domain and $G_\rho$ equals $G$ perforated by $N$ circular holes of boundaries $\gamma_j, 1 \leq j \leq N$.
	The following constraint must be respected
	\begin{equation}
	\sum\limits_{j=1}^N \deg(u, \gamma_j) = \chi \deg(g, \partial G)
	\label{eq:generalConstraint}
	\end{equation}
	which is a particular form of the Poincar\'e-Hopf theorem.
	
	We start from the definition of degree~\eqref{eq:brouwerdeg}. It is important to note that, on each $\gamma_j$, the variation of the angle could be defined as the derivative of an arc-tangent. We have
	\begin{align*}
	d\phi &= \left( \frac{1}{1 + \left( \frac{v_1}{v_2} \right)^2} \right) \left( \frac{dv_1}{v_2} - \frac{dv_2 v_1}{v_2^2} \right)\\
	&= \frac{v_2^2}{v_1^2 + v_2^2} \left( \frac{dv_1}{v_2} - \frac{v_1}{v_2^2} dv_2 \right).
	\end{align*}
	On $\gamma_j$, supposing that the elements of the field have norm 1 on each circle $\gamma_j$, we have 
	\begin{equation*}
	d\phi = v_2 dv_1 -v_1 dv_2.
	\end{equation*}
	However, the degree of $u$ on each $ \gamma_j$ precisely equals to the closed integral of $d\phi$ on $ \gamma_j$ divided by $2\pi$. Thus, our general constraint \eqref{eq:generalConstraint} can be reformulated as follows:
	\begin{equation*}
	\frac{1}{2\pi} \sum_{j=1}^{N} \int_{\gamma_j} d\phi = \chi \deg(g, \partial G).
	\end{equation*}
	
	A formulation of the constraint imposed on a single hole will now be derived.
	On each $ \gamma_j$ we have local polar coordinates (see figure \ref{fig:localCoord}).
	\begin{equation*}
	r=\sqrt{v_1^2+v_2^2}  \text{ and }  \phi = \arctan\left( \frac{v_1}{v_2} \right)
	\end{equation*}
	where $v_1$ and $v_2$ are the Cartesian coordinates of the vector $v$ considered on $ \gamma_j$.
	
	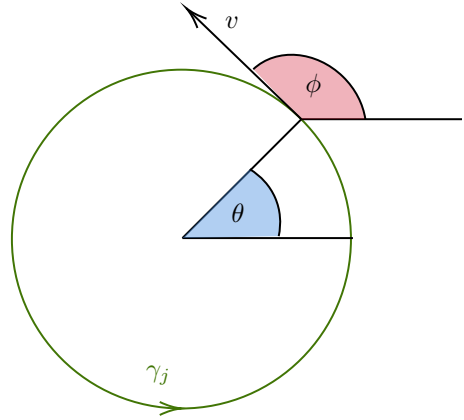
\begin{figure}[ht]
		\centering

		\tikzset{every picture/.style={line width=0.75pt}} %set default line width to 0.75pt        
		
		\begin{tikzpicture}[x=0.75pt,y=0.75pt,yscale=-1,xscale=1]
		
		%Shape: Circle [id:dp7982553092404495] 
		\draw  [color={rgb, 255:red, 65; green, 117; blue, 5 }  ,draw opacity=1 ] (224,131.5) .. controls (224,84.28) and (262.28,46) .. (309.5,46) .. controls (356.72,46) and (395,84.28) .. (395,131.5) .. controls (395,178.72) and (356.72,217) .. (309.5,217) .. controls (262.28,217) and (224,178.72) .. (224,131.5) -- cycle ;
		%Straight Lines [id:da23733261798997696] 
		\draw    (310,131) -- (396,131) ;

		%Straight Lines [id:da9962023652167233] 
		\draw    (310,131) -- (370,71) ;

		%Shape: Arc [id:dp9541892371233889] 
		\draw  [draw opacity=0][fill={rgb, 255:red, 74; green, 144; blue, 226 }  ,fill opacity=0.43 ] (358.57,130.03) .. controls (359.93,123.44) and (359.44,116.79) .. (356.77,110.66) .. controls (354.13,104.59) and (349.65,99.72) .. (343.98,96.23) -- (310,131) -- cycle ; \draw   (358.57,130.03) .. controls (359.93,123.44) and (359.44,116.79) .. (356.77,110.66) .. controls (354.13,104.59) and (349.65,99.72) .. (343.98,96.23) ;
		%Straight Lines [id:da2249353123374933] 
		\draw [color={rgb, 255:red, 0; green, 0; blue, 0 }  ,draw opacity=1 ][fill={rgb, 255:red, 0; green, 0; blue, 0 }  ,fill opacity=1 ]   (370,71) -- (451,71) ;

		%Shape: Arc [id:dp9747247666063398] 
		\draw  [draw opacity=0][fill={rgb, 255:red, 208; green, 2; blue, 27 }  ,fill opacity=0.3 ] (345.74,45.78) .. controls (357.44,34.77) and (377.42,36.33) .. (390.95,49.53) .. controls (397.29,55.71) and (401.11,63.37) .. (402.27,71) -- (370,71) -- cycle ; \draw  [color={rgb, 255:red, 0; green, 0; blue, 0 }  ,draw opacity=1 ] (345.74,45.78) .. controls (357.44,34.77) and (377.42,36.33) .. (390.95,49.53) .. controls (397.29,55.71) and (401.11,63.37) .. (402.27,71) ;
		%Straight Lines [id:da17649531618764192] 
		\draw [color={rgb, 255:red, 0; green, 0; blue, 0 }  ,draw opacity=1 ]   (370,71) -- (313.44,16.39) ;
		\draw [shift={(312,15)}, rotate = 403.99] [color={rgb, 255:red, 0; green, 0; blue, 0 }  ,draw opacity=1 ][line width=0.75]    (10.93,-3.29) .. controls (6.95,-1.4) and (3.31,-0.3) .. (0,0) .. controls (3.31,0.3) and (6.95,1.4) .. (10.93,3.29)   ;
		
		%Straight Lines [id:da882421962638485] 
		\draw [color={rgb, 255:red, 65; green, 117; blue, 5 }  ,draw opacity=1 ]   (309.5,217) ;
		\draw [shift={(309.5,217)}, rotate = 180] [color={rgb, 255:red, 65; green, 117; blue, 5 }  ,draw opacity=1 ][line width=0.75]    (10.93,-3.29) .. controls (6.95,-1.4) and (3.31,-0.3) .. (0,0) .. controls (3.31,0.3) and (6.95,1.4) .. (10.93,3.29)   ;

		% Text Node
		\draw (338,118) node   {$\theta $};
		% Text Node
		\draw (376,54) node   {$\phi $};
		% Text Node
		\draw (335,22) node   {$v$};
		% Text Node
		\draw (298,200) node   {$\textcolor[rgb]{0.25,0.46,0.02}{\gamma _{j}}$};

		\end{tikzpicture}

		\caption{Definition of local polar coordinates.}
		\label{fig:localCoord}
	\end{figure}

	With these new coordinates, $dr = 0$ on $ \gamma_j$. Furthermore, the derivatives of our vector components only involve partial derivatives with respect to the $\theta$ variable. For each $\gamma_j$, this leads to the following equation:
	\begin{equation*}
	d\phi = \left( v_2 \frac{\partial v_1}{\partial \theta} - v_1 \frac{\partial v_2}{\partial \theta} \right) d\theta,
	\end{equation*}
	where $\theta$ is the angle to the center of each $\gamma_j$.
	
	At this point, it is important to note that numerically, it will be necessary to pay particular attention to the orientation on $ \gamma_j$ in order to truly match the general constraint \eqref{eq:generalConstraint}.
	
	In practical terms, the following centered differences will be worked with, which can be understood by the figure \ref{fig:octo}:
	\begin{align*}
	&\left( \frac{\partial v_1}{\partial \theta} \right)^i = \frac{v_1^{i+1} - v_1^{i-1}}{2\Delta \theta},\\
	&\left( \frac{\partial v_2}{\partial \theta} \right)^i = \frac{v_2^{i+1} - v_2^{i-1}}{2\Delta \theta},
	\end{align*}
	
	\begin{align*}
	\left( \Delta \phi \right)^{i,i+1} &= \left(\frac{v_2^i \left( v_1^{i+1} - v_1^{i-1} \right) - v_1^i \left( v_2^{i+1} - v_2^{i-1} \right)}{2\Delta \theta} \right) \Delta \theta, \\
	& = \frac{1}{2} \left(v_2^i \left( v_1^{i+1} - v_1^{i-1} \right) - v_1^i \left( v_2^{i+1} - v_2^{i-1} \right)\right),
	\end{align*}
	where the indices placed in exponents indicate successive node numbers.
	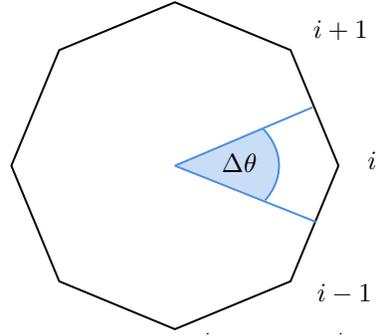
\begin{figure}[ht]
		\centering

		\tikzset{every picture/.style={line width=0.75pt}} %set default line width to 0.75pt        
		
		\begin{tikzpicture}[x=0.75pt,y=0.75pt,yscale=-1,xscale=1]
		
		%Shape: Regular Polygon [id:dp26806191934292656] 
		\draw   (390,156.5) -- (365.84,214.84) -- (307.5,239) -- (249.16,214.84) -- (225,156.5) -- (249.16,98.16) -- (307.5,74) -- (365.84,98.16) -- cycle ;
		%Straight Lines [id:da6914968999394927] 
		\draw [color={rgb, 255:red, 74; green, 144; blue, 226 }  ,draw opacity=1 ]   (307.5,156.5) -- (377,127) ;

		%Straight Lines [id:da10149458494591568] 
		\draw [color={rgb, 255:red, 74; green, 144; blue, 226 }  ,draw opacity=1 ]   (307.5,156.5) -- (379,185) ;

		%Shape: Arc [id:dp9050206369284926] 
		\draw  [draw opacity=0][fill={rgb, 255:red, 74; green, 144; blue, 226 }  ,fill opacity=0.3 ] (352.98,173.99) .. controls (357.45,168.85) and (360,162.87) .. (360,156.5) .. controls (360,149.49) and (356.91,142.96) .. (351.59,137.49) -- (307.5,156.5) -- cycle ; \draw  [color={rgb, 255:red, 74; green, 144; blue, 226 }  ,draw opacity=1 ] (352.98,173.99) .. controls (357.45,168.85) and (360,162.87) .. (360,156.5) .. controls (360,149.49) and (356.91,142.96) .. (351.59,137.49) ;
		
		% Text Node
		\draw (340,155) node   {$\Delta \theta $};
		% Text Node
		\draw (391,88) node   {$i+1$};
		% Text Node
		\draw (407,154) node   {$i$};
		% Text Node
		\draw (393,221) node   {$i-1$};

		\end{tikzpicture}
		\caption{Finite centered differences for $\left( \frac{\partial v_1}{\partial \theta} \right)^i$ and $ \left(\frac{\partial v_2}{\partial \theta} \right)^i$.}
		\label{fig:octo}
	\end{figure}

	Hence, for a hole $\gamma_j$ approximated by a regular polygon with $m$ edges, we have 
	\begin{equation*}
	\sum\limits_{i=1}^{m-1} \left( \Delta \phi \right)^{i,i+1} \approx \int_{\gamma_j} d\phi.
	\end{equation*}
	
	Finally, we have as a discrete approximation of the constraint equation \eqref{eq:generalConstraint}:
	\begin{multline*}
	\frac{1}{2\pi}\sum\limits_{j=1}^N \left( \sum\limits_{i=1}^{m-1} \left( \Delta \phi \right)^{i,i+1}  \right)_{\gamma_j} =\\
	\frac{1}{4\pi} \sum\limits_{j=1}^N \left( \sum\limits_{i=1}^{m-1} v_2^i \left( v_1^{i+1} - v_1^{i-1} \right) - v_1^i \left( v_2^{i+1} - v_2^{i-1} \right) \right)_{\gamma_j} 
	\\ \approx \chi \deg(g, \partial G).
	\end{multline*}
	
	In practice, we will separate the overall condition into one condition on each perforated hole in order to clearly choose the placement of the singularities.
	The resolution of a $\Delta u = 0$ with the imposition of conditions on the boundary crosses to align with $\partial G$ leads to a linear system 
	\begin{equation*}
	K x = b 
	\end{equation*}
	with $x = (v_1^1, v_2^1, \cdots v_1^n, v_2^n)$.
	
	We can rewrite the non-linear quadratic constraint as followed:
	\begin{equation*}
	x^TMx = \sum\limits_{i=1}^{2n} \sum\limits_{j=1}^{2n} m_{ij} x_i x_j = 2\pi \chi \deg(g, \partial G).
	\end{equation*}
	For example, if we have a single hole with nodes 4-38-39-40-41 on it, $M$ will have the following shape:
	\begin{table}[ht]
		\begin{center}
			\begin{adjustbox}{max width=.46\textwidth}
				\begin{tabular}{l|llllllllll}
					\cline{2-11}
					& \multicolumn{1}{l|}{\scriptsize $v_1^4$} & \multicolumn{1}{l|}{\scriptsize$v_2^4$} & \multicolumn{1}{l|}{\scriptsize $v_1^{38}$} & \multicolumn{1}{l|}{\scriptsize $v_2^{38}$} & \multicolumn{1}{l|}{\scriptsize$v_1^{39}$} & \multicolumn{1}{l|}{\scriptsize$v_2^{39}$} & \multicolumn{1}{l|}{\scriptsize$v_1^{40}$} & \multicolumn{1}{l|}{\scriptsize$v_2^{40}$} & \multicolumn{1}{l|}{\scriptsize$v_1^{41}$} & \multicolumn{1}{l|}{\scriptsize $v_2^{41}$} \\ \hline
					\multicolumn{1}{|l|}{\scriptsize$v_1^4$}    &                              &                              &                                 & -1                              &                                 &                                 &                                 &                                 &                                 & 1                               \\ \cline{1-1}
					\multicolumn{1}{|l|}{\scriptsize$v_2^4$}    &                              &                              & 1                               &                                 &                                 &                                 &                                 &                                 & -1                              &                                 \\ \cline{1-1}
					\multicolumn{1}{|l|}{\scriptsize$v_1^{38}$} &                              & 1                            &                                 &                                 &                                 & -1                              &                                 &                                 &                                 &                                 \\ \cline{1-1}
					\multicolumn{1}{|l|}{\scriptsize$v_2^{38}$} & -1                           &                              &                                 &                                 & 1                               &                                 &                                 &                                 &                                 &                                 \\ \cline{1-1}
					\multicolumn{1}{|l|}{\scriptsize$v_1^{39}$} &                              &                              &                                 & 1                               &                                 &                                 &                                 & -1                              &                                 &                                 \\ \cline{1-1}
					\multicolumn{1}{|l|}{\scriptsize$v_2^{39}$} &                              &                              & -1                              &                                 &                                 &                                 & 1                               &                                 &                                 &                                 \\ \cline{1-1}
					\multicolumn{1}{|l|}{\scriptsize$v_1^{40}$} &                              &                              &                                 &                                 &                                 & 1                               &                                 &                                 &                                 & -1                              \\ \cline{1-1}
					\multicolumn{1}{|l|}{\scriptsize$v_2^{40}$} &                              &                              &                                 &                                 & -1                              &                                 &                                 &                                 & 1                               &                                 \\ \cline{1-1}
					\multicolumn{1}{|l|}{\scriptsize$v_1^{41}$} &                              & -1                           &                                 &                                 &                                 &                                 &                                 & 1                               &                                 &                                 \\ \cline{1-1}
					\multicolumn{1}{|l|}{\scriptsize $v_2^{41}$} & 1                            &                              &                                 &                                 &                                 &                                 & -1                              &                                 &                                 &                                 \\ \cline{1-1}
				\end{tabular}
			\end{adjustbox}
		\end{center}
	\end{table}
	
	We then have
	\begin{multline*}
	x^TMx = \; v_2^4 v_1^{38} - v_2^4 v_1^{41} - v_1^4 v_2^{38} +v_1^4 v_2^{41} + v_2^{38} v_1^{39}\\ - v_2^{38} v_1^4 - v_1^{38}v_2^{39} 
	+ v_1^{38} v_2^4 + v_2^{39} v_1^{40}\\ - v_2^{39} v_1^{38} - v_1^{39} v_2^{40} + v_1^{39} v_2^{40} + v_1^{39} v_2^{38}\\ + v_2^{40} v_1^{41}
	- v_2^{40} v_1^{39} - v_1^{40}v_2^{41} + v_1^{40}v_2^{39}\\ + v_2^{41} v_1^4 - v_2^{41} v_1^{40} - v_1^{41} v_2^4 + v_1^{41} v_2^{40}.
	\end{multline*}
	
	In order to solve this problem, a Newton-Raphson scheme is computed with the utilization of a Lagrange multiplier for each of the constraint of each of our perforated holes. We then have each constraint $C_{\gamma_i}$ on each $\gamma_i$ and its derivatives which are described by
	\[
	C_{\gamma_i} := \lambda_i \left( x^T M x - 2\pi \deg(u, \gamma_i) \right),
	\]
	\[
	\frac{\partial C_{\gamma_i}}{\partial \lambda_i} = x^T M x - 2\pi \deg(u, \gamma_i),
	\]
	\[
	\frac{\partial C_{\gamma_i}}{\partial x} = 2 \lambda_i Mx,
	\]
	\[
	\frac{\partial^2 C_{\gamma_i}}{\partial \lambda_i \partial x} = 2Mx,
	\]
	\[
	\frac{\partial^2 C_{\gamma_i}}{\partial \lambda_i \partial \lambda_i} = 0,
	\]
	and
	\[
	\frac{\partial^2 C_{\gamma_i}}{\partial x \partial x} = 2M.
	\]
	The gradient of our problem becomes
	\[
	\tilde{b} = [b +  2 \lambda_i Mx, \; x^T M x - 2\pi \deg(u, \gamma_i)]^T
	\]
	and the Hessian becomes
	\[
	H = \begin{bmatrix}
	K + 2M & 2Mx \\
	(2Mx)^T  & 0
	\end{bmatrix}.
	\]
	We solve this via Newton-Raphson scheme until convergence and if the Hessian is singular at a particular iteration, we set for this particular iteration : 
	\[\frac{\partial^2 C_{\gamma_i}}{\partial \lambda_i \partial \lambda_i}=1.
	\]
	
	\subsection{New computations for placing boundary singularities}\label{GLPiecewise}
	
	% \subsubsection{Theory: singular energy minimization in case of angular domains}
	When cross fields are computed on piecewise smooth closed domains, in addition to the internal singularities we have already discussed, boundary singularities may appear on the boundary corners. 
	The corner angles to be considered correspond to the total angle formed around the boundary corner inside the domain.
	Thus, an outgoing corner will correspond to an angle between $0$ and $\pi$ radians while an incoming corner will correspond to an angle between $\pi$ and $2\pi$ radians. 
	
	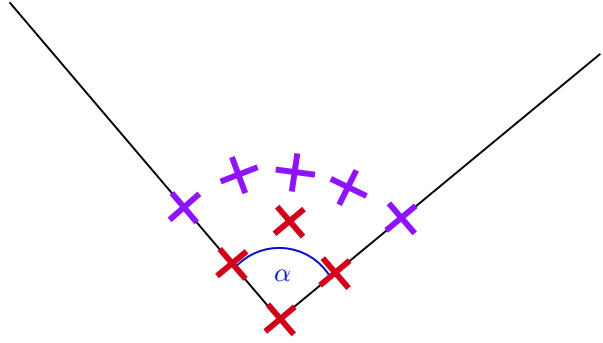
\begin{figure}[H]
		\begin{center}

			\tikzset{every picture/.style={line width=0.75pt}} %set default line width to 0.75pt        
			
			\begin{tikzpicture}[x=0.75pt,y=0.75pt,yscale=-1,xscale=1]
			
			%Straight Lines [id:da9136489217121772] 
			\draw    (64,61) -- (200,221) ;

			%Straight Lines [id:da05077191331855779] 
			\draw    (200,221) -- (362,87) ;

			%Shape: Arc [id:dp40266938749543135] 
			\draw  [draw opacity=0] (178.24,194.35) .. controls (184.1,188.18) and (192.53,184.53) .. (201.69,185.05) .. controls (211.9,185.62) and (220.64,191.25) .. (225.61,199.36) -- (200,215) -- cycle ; \draw  [color={rgb, 255:red, 2; green, 10; blue, 208 }  ,draw opacity=1 ] (178.24,194.35) .. controls (184.1,188.18) and (192.53,184.53) .. (201.69,185.05) .. controls (211.9,185.62) and (220.64,191.25) .. (225.61,199.36) ;
			%Straight Lines [id:da7695742396525193] 
			\draw [color={rgb, 255:red, 144; green, 19; blue, 254 }  ,draw opacity=1 ][line width=2.25]    (159.77,157.68) -- (144.7,170.82) ;

			%Straight Lines [id:da011859614007216912] 
			\draw [color={rgb, 255:red, 144; green, 19; blue, 254 }  ,draw opacity=1 ][line width=2.25]    (145.85,157.26) -- (158.41,172.09) ;

			%Straight Lines [id:da3156289741709519] 
			\draw [color={rgb, 255:red, 144; green, 19; blue, 254 }  ,draw opacity=1 ][line width=2.25]    (189.11,144.27) -- (170.45,151.47) ;

			%Straight Lines [id:da5037422870900554] 
			\draw [color={rgb, 255:red, 144; green, 19; blue, 254 }  ,draw opacity=1 ][line width=2.25]    (176.17,139.12) -- (182.9,157.35) ;

			%Straight Lines [id:da4461674294981194] 
			\draw [color={rgb, 255:red, 144; green, 19; blue, 254 }  ,draw opacity=1 ][line width=2.25]    (218.07,148.02) -- (198.24,145.43) ;

			%Straight Lines [id:da5402010606701668] 
			\draw [color={rgb, 255:red, 144; green, 19; blue, 254 }  ,draw opacity=1 ][line width=2.25]    (209.16,137.31) -- (206.37,156.54) ;

			%Straight Lines [id:da09077789898622435] 
			\draw [color={rgb, 255:red, 144; green, 19; blue, 254 }  ,draw opacity=1 ][line width=2.25]    (244.03,158.77) -- (226,150.11) ;

			%Straight Lines [id:da7691694694233236] 
			\draw [color={rgb, 255:red, 144; green, 19; blue, 254 }  ,draw opacity=1 ][line width=2.25]    (238.92,145.82) -- (230.25,163.21) ;

			%Straight Lines [id:da2929592502160311] 
			\draw [color={rgb, 255:red, 144; green, 19; blue, 254 }  ,draw opacity=1 ][line width=2.25]    (268.37,177.64) -- (255.32,162.48) ;

			%Straight Lines [id:da2458662181340291] 
			\draw [color={rgb, 255:red, 144; green, 19; blue, 254 }  ,draw opacity=1 ][line width=2.25]    (268.87,163.72) -- (253.96,176.18) ;

			%Straight Lines [id:da03212640907748965] 
			\draw [color={rgb, 255:red, 208; green, 2; blue, 27 }  ,draw opacity=1 ][line width=2.25]    (235.03,205.31) -- (221.99,190.15) ;

			%Straight Lines [id:da39402453602697707] 
			\draw [color={rgb, 255:red, 208; green, 2; blue, 27 }  ,draw opacity=1 ][line width=2.25]    (235.54,191.39) -- (220.63,203.85) ;

			%Straight Lines [id:da7021010969415014] 
			\draw [color={rgb, 255:red, 208; green, 2; blue, 27 }  ,draw opacity=1 ][line width=2.25]    (183.03,200.64) -- (169.99,185.48) ;

			%Straight Lines [id:da4409871578896002] 
			\draw [color={rgb, 255:red, 208; green, 2; blue, 27 }  ,draw opacity=1 ][line width=2.25]    (183.54,186.72) -- (168.63,199.18) ;

			%Straight Lines [id:da7740309757238505] 
			\draw [color={rgb, 255:red, 208; green, 2; blue, 27 }  ,draw opacity=1 ][line width=2.25]    (207.37,228.64) -- (194.32,213.48) ;

			%Straight Lines [id:da3228002824525311] 
			\draw [color={rgb, 255:red, 208; green, 2; blue, 27 }  ,draw opacity=1 ][line width=2.25]    (207.87,214.72) -- (192.96,227.18) ;

			%Straight Lines [id:da2959562242957292] 
			\draw [color={rgb, 255:red, 208; green, 2; blue, 27 }  ,draw opacity=1 ][line width=2.25]    (212.03,179.31) -- (198.99,164.15) ;

			%Straight Lines [id:da5956835377649442] 
			\draw [color={rgb, 255:red, 208; green, 2; blue, 27 }  ,draw opacity=1 ][line width=2.25]    (212.54,165.39) -- (197.63,177.85) ;

			% Text Node
			\draw (201.8,198.6) node  [color={rgb, 255:red, 2; green, 16; blue, 208 }  ,opacity=1 ] [align=left] {$\displaystyle \alpha $};

			\end{tikzpicture}
			
		\end{center}
		\caption{Comparison of the presence of a boundary singularity of degree $1/4$ in red on the apex of the $\alpha (\approx \pi/2)$ angle or the absence of singularity in purple.}
		\label{fig:angleCrosses}
	\end{figure}
	
	A boundary singularity implies a rotation of the boundary cross compared to the local outward normal at the domain boundaries. For example, if around a boundary corner, the border cross rotation is $\frac{\pi}{2}$ compared to the local outward normal, we say that a $1/4$ boundary singularity has been placed on this corner. The presence or the absence of a singularity on such an angle implies different rotations for the computed cross field (see Figure \ref{fig:angleCrosses}) since the field inside the domain must be smooth.
	
	The boundary corners are usually considered as singularity  for angle multiple of k$\pi/2$, considering that a singularity of index 0 is in fact a regular cross. For example the four corners of a square correspond to four 1/4 boundary singularities thus the Poincaré-Hopf theorem is respected without internal singularities.
	The conventional distribution proposed in the literature corresponds to the one given on Table \ref{tab:angles} with different tolerances around the given angles.
	
	\begin{table}[H]
		\centering
		\begin{tabular}{l|c|c|c|c|l}
			
			angles (in radians)         & $2\pi $ & $\frac{3\pi}{2}$ & $\pi$ & $\frac{\pi}{2}$ & 0\\[.2cm]
			index of the singularity & $-\frac{2}{4}$                             & $-\frac{1}{4}$                                                                                                             & $0$                                                                                               & $\frac{1}{4}$  & one boundary $\frac{1}{4}$ and one internal $\frac{1}{4}$                                                                   \\[.2cm]
		\end{tabular}
		\caption{Usual singularities angles distribution.}
		\label{tab:angles}
	\end{table}
	
	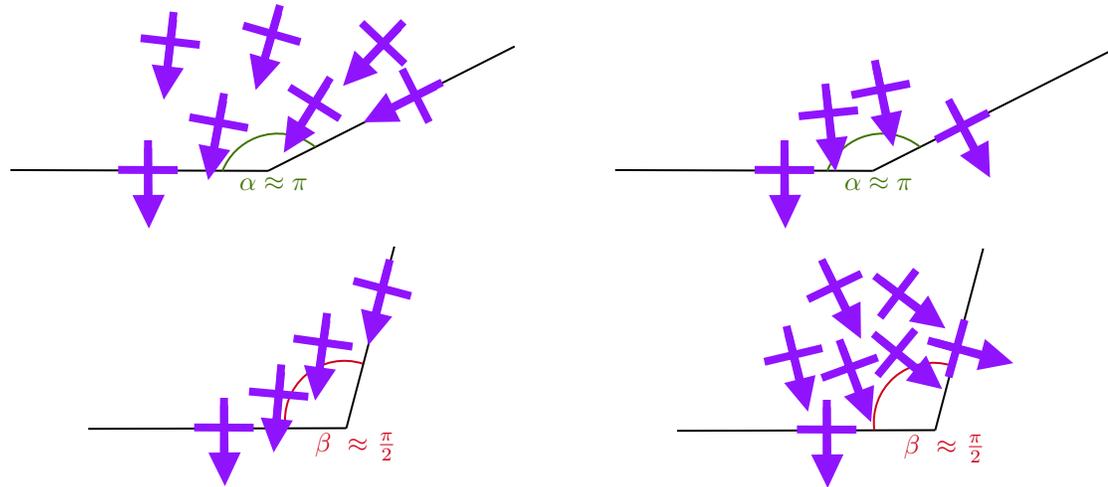
\begin{figure}[ht] 
		\centering

		\tikzset{every picture/.style={line width=0.75pt}} %set default line width to 0.75pt        
		
		\begin{tikzpicture}[x=0.75pt,y=0.75pt,yscale=-1,xscale=1]
		
		%Straight Lines [id:da4319569685214921] 
		\draw    (50.2,89.4) -- (180.2,89.8) ;
		%Straight Lines [id:da9908683688048057] 
		\draw    (180.2,89.8) -- (304.5,27) ;
		%Shape: Arc [id:dp7670093423049111] 
		\draw  [draw opacity=0] (157.19,89.74) .. controls (160.19,82.35) and (166.1,76.16) .. (174.12,73.04) .. controls (184.55,68.98) and (195.87,71.15) .. (203.97,77.75) -- (185,101) -- cycle ; \draw  [color={rgb, 255:red, 65; green, 117; blue, 5 }  ,draw opacity=1 ] (157.19,89.74) .. controls (160.19,82.35) and (166.1,76.16) .. (174.12,73.04) .. controls (184.55,68.98) and (195.87,71.15) .. (203.97,77.75) ;
		%Straight Lines [id:da7953096667468448] 
		\draw [color={rgb, 255:red, 144; green, 19; blue, 254 }  ,draw opacity=1 ][line width=3]    (119.64,78.07) -- (119.99,112.92) ;
		\draw [shift={(120.06,118.92)}, rotate = 269.40999999999997] [fill={rgb, 255:red, 144; green, 19; blue, 254 }  ,fill opacity=1 ][line width=0.08]  [draw opacity=0] (16.97,-8.15) -- (0,0) -- (16.97,8.15) -- cycle    ;
		\draw  [color={rgb, 255:red, 144; green, 19; blue, 254 }  ,draw opacity=1 ][line width=3]  (134.56,89.42) -- (104.56,89.3)(119.5,104.36) -- (119.62,74.36) ;
		%Straight Lines [id:da6679351436450531] 
		\draw [color={rgb, 255:red, 144; green, 19; blue, 254 }  ,draw opacity=1 ][line width=3]    (264.74,46.98) -- (233.59,62.61) ;
		\draw [shift={(228.23,65.3)}, rotate = 333.36] [fill={rgb, 255:red, 144; green, 19; blue, 254 }  ,fill opacity=1 ][line width=0.08]  [draw opacity=0] (16.97,-8.15) -- (0,0) -- (16.97,8.15) -- cycle    ;
		\draw  [color={rgb, 255:red, 144; green, 19; blue, 254 }  ,draw opacity=1 ][line width=3]  (261.34,65.26) -- (247.79,38.49)(241.18,58.64) -- (267.95,45.1) ;
		%Straight Lines [id:da6770625136984745] 
		\draw    (89.46,220) -- (219.64,219.67) ;
		%Straight Lines [id:da5599018321510362] 
		\draw    (219.64,219.67) -- (243.83,128) ;
		%Straight Lines [id:da4788531850440837] 
		\draw [color={rgb, 255:red, 144; green, 19; blue, 254 }  ,draw opacity=1 ][line width=3]    (157.97,208.07) -- (158.33,242.92) ;
		\draw [shift={(158.39,248.92)}, rotate = 269.40999999999997] [fill={rgb, 255:red, 144; green, 19; blue, 254 }  ,fill opacity=1 ][line width=0.08]  [draw opacity=0] (16.97,-8.15) -- (0,0) -- (16.97,8.15) -- cycle    ;
		\draw  [color={rgb, 255:red, 144; green, 19; blue, 254 }  ,draw opacity=1 ][line width=3]  (172.89,219.42) -- (142.89,219.3)(157.83,234.36) -- (157.95,204.36) ;
		%Straight Lines [id:da23409139624904685] 
		\draw [color={rgb, 255:red, 144; green, 19; blue, 254 }  ,draw opacity=1 ][line width=3]    (240.66,138.23) -- (232.21,172.04) ;
		\draw [shift={(230.76,177.87)}, rotate = 284.02] [fill={rgb, 255:red, 144; green, 19; blue, 254 }  ,fill opacity=1 ][line width=0.08]  [draw opacity=0] (16.97,-8.15) -- (0,0) -- (16.97,8.15) -- cycle    ;
		\draw  [color={rgb, 255:red, 144; green, 19; blue, 254 }  ,draw opacity=1 ][line width=3]  (252.23,152.98) -- (223.23,145.3)(233.89,163.64) -- (241.57,134.64) ;
		%Shape: Arc [id:dp11424793690370905] 
		\draw  [draw opacity=0] (188.93,219.66) .. controls (187.64,209.96) and (191.11,199.83) .. (199.06,192.96) .. controls (207.26,185.88) and (218.19,184.02) .. (227.8,187.09) -- (218.67,215.67) -- cycle ; \draw  [color={rgb, 255:red, 208; green, 2; blue, 27 }  ,draw opacity=1 ] (188.93,219.66) .. controls (187.64,209.96) and (191.11,199.83) .. (199.06,192.96) .. controls (207.26,185.88) and (218.19,184.02) .. (227.8,187.09) ;
		%Straight Lines [id:da19834326756155085] 
		\draw [color={rgb, 255:red, 144; green, 19; blue, 254 }  ,draw opacity=1 ][line width=3]    (209.78,46.05) -- (191.26,75.58) ;
		\draw [shift={(188.07,80.66)}, rotate = 302.09000000000003] [fill={rgb, 255:red, 144; green, 19; blue, 254 }  ,fill opacity=1 ][line width=0.08]  [draw opacity=0] (16.97,-8.15) -- (0,0) -- (16.97,8.15) -- cycle    ;
		\draw  [color={rgb, 255:red, 144; green, 19; blue, 254 }  ,draw opacity=1 ][line width=3]  (216.35,63.44) -- (190.88,47.59)(195.69,68.25) -- (211.54,42.78) ;
		%Straight Lines [id:da8806474773452976] 
		\draw [color={rgb, 255:red, 144; green, 19; blue, 254 }  ,draw opacity=1 ][line width=3]    (157.54,54.37) -- (151.07,88.62) ;
		\draw [shift={(149.96,94.51)}, rotate = 280.69] [fill={rgb, 255:red, 144; green, 19; blue, 254 }  ,fill opacity=1 ][line width=0.08]  [draw opacity=0] (16.97,-8.15) -- (0,0) -- (16.97,8.15) -- cycle    ;
		\draw  [color={rgb, 255:red, 144; green, 19; blue, 254 }  ,draw opacity=1 ][line width=3]  (170,68.16) -- (140.5,62.7)(152.52,80.18) -- (157.98,50.68) ;
		%Straight Lines [id:da21624715325280275] 
		\draw [color={rgb, 255:red, 144; green, 19; blue, 254 }  ,draw opacity=1 ][line width=3]    (209.64,165.31) -- (204.75,199.82) ;
		\draw [shift={(203.9,205.76)}, rotate = 278.07] [fill={rgb, 255:red, 144; green, 19; blue, 254 }  ,fill opacity=1 ][line width=0.08]  [draw opacity=0] (16.97,-8.15) -- (0,0) -- (16.97,8.15) -- cycle    ;
		\draw  [color={rgb, 255:red, 144; green, 19; blue, 254 }  ,draw opacity=1 ][line width=3]  (222.72,178.52) -- (193,174.41)(205.81,191.32) -- (209.91,161.6) ;
		%Straight Lines [id:da6628675658181477] 
		\draw [color={rgb, 255:red, 144; green, 19; blue, 254 }  ,draw opacity=1 ][line width=3]    (186.74,191.25) -- (183.54,225.95) ;
		\draw [shift={(182.98,231.92)}, rotate = 275.28] [fill={rgb, 255:red, 144; green, 19; blue, 254 }  ,fill opacity=1 ][line width=0.08]  [draw opacity=0] (16.97,-8.15) -- (0,0) -- (16.97,8.15) -- cycle    ;
		\draw  [color={rgb, 255:red, 144; green, 19; blue, 254 }  ,draw opacity=1 ][line width=3]  (200.45,203.8) -- (170.57,201.14)(184.18,217.41) -- (186.84,187.53) ;
		%Straight Lines [id:da9639717009787692] 
		\draw [color={rgb, 255:red, 144; green, 19; blue, 254 }  ,draw opacity=1 ][line width=3]    (132.33,13.32) -- (128.26,47.93) ;
		\draw [shift={(127.56,53.89)}, rotate = 276.71] [fill={rgb, 255:red, 144; green, 19; blue, 254 }  ,fill opacity=1 ][line width=0.08]  [draw opacity=0] (16.97,-8.15) -- (0,0) -- (16.97,8.15) -- cycle    ;
		\draw  [color={rgb, 255:red, 144; green, 19; blue, 254 }  ,draw opacity=1 ][line width=3]  (145.72,26.21) -- (115.91,22.81)(129.12,39.42) -- (132.52,9.61) ;
		%Straight Lines [id:da5296872685487715] 
		\draw [color={rgb, 255:red, 144; green, 19; blue, 254 }  ,draw opacity=1 ][line width=3]    (185.38,9.83) -- (175.6,43.28) ;
		\draw [shift={(173.92,49.04)}, rotate = 286.29] [fill={rgb, 255:red, 144; green, 19; blue, 254 }  ,fill opacity=1 ][line width=0.08]  [draw opacity=0] (16.97,-8.15) -- (0,0) -- (16.97,8.15) -- cycle    ;
		\draw  [color={rgb, 255:red, 144; green, 19; blue, 254 }  ,draw opacity=1 ][line width=3]  (196.43,24.77) -- (167.61,16.46)(177.86,35.03) -- (186.18,6.2) ;
		%Straight Lines [id:da5649492233039999] 
		\draw [color={rgb, 255:red, 144; green, 19; blue, 254 }  ,draw opacity=1 ][line width=3]    (246.15,17.42) -- (222.07,42.61) ;
		\draw [shift={(217.92,46.95)}, rotate = 313.71000000000004] [fill={rgb, 255:red, 144; green, 19; blue, 254 }  ,fill opacity=1 ][line width=0.08]  [draw opacity=0] (16.97,-8.15) -- (0,0) -- (16.97,8.15) -- cycle    ;
		\draw  [color={rgb, 255:red, 144; green, 19; blue, 254 }  ,draw opacity=1 ][line width=3]  (249.08,35.78) -- (227.33,15.12)(227.88,36.33) -- (248.53,14.57) ;
		%Straight Lines [id:da91084049762993] 
		\draw    (355.2,89.4) -- (485.2,89.8) ;
		%Straight Lines [id:da24306807822345822] 
		\draw    (485.2,89.8) -- (609.5,27) ;
		%Shape: Arc [id:dp4551390829719395] 
		\draw  [draw opacity=0] (462.19,89.74) .. controls (465.19,82.35) and (471.1,76.16) .. (479.12,73.04) .. controls (489.55,68.98) and (500.87,71.15) .. (508.97,77.75) -- (490,101) -- cycle ; \draw  [color={rgb, 255:red, 65; green, 117; blue, 5 }  ,draw opacity=1 ] (462.19,89.74) .. controls (465.19,82.35) and (471.1,76.16) .. (479.12,73.04) .. controls (489.55,68.98) and (500.87,71.15) .. (508.97,77.75) ;
		%Straight Lines [id:da6227147429297112] 
		\draw [color={rgb, 255:red, 144; green, 19; blue, 254 }  ,draw opacity=1 ][line width=3]    (440.64,78.07) -- (440.99,112.92) ;
		\draw [shift={(441.06,118.92)}, rotate = 269.40999999999997] [fill={rgb, 255:red, 144; green, 19; blue, 254 }  ,fill opacity=1 ][line width=0.08]  [draw opacity=0] (16.97,-8.15) -- (0,0) -- (16.97,8.15) -- cycle    ;
		\draw  [color={rgb, 255:red, 144; green, 19; blue, 254 }  ,draw opacity=1 ][line width=3]  (455.56,89.42) -- (425.56,89.3)(440.5,104.36) -- (440.62,74.36) ;
		%Straight Lines [id:da0727999166729918] 
		\draw [color={rgb, 255:red, 144; green, 19; blue, 254 }  ,draw opacity=1 ][line width=3]    (525.04,57.2) -- (541.46,87.95) ;
		\draw [shift={(544.29,93.24)}, rotate = 241.9] [fill={rgb, 255:red, 144; green, 19; blue, 254 }  ,fill opacity=1 ][line width=0.08]  [draw opacity=0] (16.97,-8.15) -- (0,0) -- (16.97,8.15) -- cycle    ;
		\draw  [color={rgb, 255:red, 144; green, 19; blue, 254 }  ,draw opacity=1 ][line width=3]  (543.52,60.38) -- (516.86,74.13)(537.07,80.59) -- (523.31,53.92) ;
		%Straight Lines [id:da38042972408468345] 
		\draw [color={rgb, 255:red, 144; green, 19; blue, 254 }  ,draw opacity=1 ][line width=3]    (486.93,37.58) -- (494.21,71.66) ;
		\draw [shift={(495.46,77.53)}, rotate = 257.94] [fill={rgb, 255:red, 144; green, 19; blue, 254 }  ,fill opacity=1 ][line width=0.08]  [draw opacity=0] (16.97,-8.15) -- (0,0) -- (16.97,8.15) -- cycle    ;
		\draw  [color={rgb, 255:red, 144; green, 19; blue, 254 }  ,draw opacity=1 ][line width=3]  (503.81,45.74) -- (474.39,51.58)(492.02,63.37) -- (486.17,33.95) ;
		%Straight Lines [id:da8892504038110778] 
		\draw [color={rgb, 255:red, 144; green, 19; blue, 254 }  ,draw opacity=1 ][line width=3]    (461.63,49.5) -- (465.67,84.12) ;
		\draw [shift={(466.36,90.08)}, rotate = 263.35] [fill={rgb, 255:red, 144; green, 19; blue, 254 }  ,fill opacity=1 ][line width=0.08]  [draw opacity=0] (16.97,-8.15) -- (0,0) -- (16.97,8.15) -- cycle    ;
		\draw  [color={rgb, 255:red, 144; green, 19; blue, 254 }  ,draw opacity=1 ][line width=3]  (477.67,59.22) -- (447.82,62.26)(464.27,75.66) -- (461.22,45.82) ;
		%Straight Lines [id:da9057583555298049] 
		\draw    (386.46,221) -- (516.64,220.67) ;
		%Straight Lines [id:da4861166304615411] 
		\draw    (516.64,220.67) -- (540.83,129) ;
		%Shape: Arc [id:dp7552844646253741] 
		\draw  [draw opacity=0] (485.93,220.66) .. controls (484.64,210.96) and (488.11,200.83) .. (496.06,193.96) .. controls (504.26,186.88) and (515.19,185.02) .. (524.8,188.09) -- (515.67,216.67) -- cycle ; \draw  [color={rgb, 255:red, 208; green, 2; blue, 27 }  ,draw opacity=1 ] (485.93,220.66) .. controls (484.64,210.96) and (488.11,200.83) .. (496.06,193.96) .. controls (504.26,186.88) and (515.19,185.02) .. (524.8,188.09) ;
		%Straight Lines [id:da30285088741806543] 
		\draw [color={rgb, 255:red, 144; green, 19; blue, 254 }  ,draw opacity=1 ][line width=3]    (461.97,209.07) -- (462.33,243.92) ;
		\draw [shift={(462.39,249.92)}, rotate = 269.40999999999997] [fill={rgb, 255:red, 144; green, 19; blue, 254 }  ,fill opacity=1 ][line width=0.08]  [draw opacity=0] (16.97,-8.15) -- (0,0) -- (16.97,8.15) -- cycle    ;
		\draw  [color={rgb, 255:red, 144; green, 19; blue, 254 }  ,draw opacity=1 ][line width=3]  (476.89,220.42) -- (446.89,220.3)(461.83,235.36) -- (461.95,205.36) ;
		%Straight Lines [id:da745625887057227] 
		\draw [color={rgb, 255:red, 144; green, 19; blue, 254 }  ,draw opacity=1 ][line width=3]    (516.53,176.39) -- (550.22,185.33) ;
		\draw [shift={(556.02,186.87)}, rotate = 194.86] [fill={rgb, 255:red, 144; green, 19; blue, 254 }  ,fill opacity=1 ][line width=0.08]  [draw opacity=0] (16.97,-8.15) -- (0,0) -- (16.97,8.15) -- cycle    ;
		\draw  [color={rgb, 255:red, 144; green, 19; blue, 254 }  ,draw opacity=1 ][line width=3]  (531.45,165.03) -- (523.34,193.92)(541.84,183.53) -- (512.95,175.42) ;
		%Straight Lines [id:da4829288625887497] 
		\draw [color={rgb, 255:red, 144; green, 19; blue, 254 }  ,draw opacity=1 ][line width=3]    (488.32,174.38) -- (515.5,196.2) ;
		\draw [shift={(520.18,199.95)}, rotate = 218.74] [fill={rgb, 255:red, 144; green, 19; blue, 254 }  ,fill opacity=1 ][line width=0.08]  [draw opacity=0] (16.97,-8.15) -- (0,0) -- (16.97,8.15) -- cycle    ;
		\draw  [color={rgb, 255:red, 144; green, 19; blue, 254 }  ,draw opacity=1 ][line width=3]  (506.56,170.04) -- (487.45,193.17)(508.57,191.16) -- (485.44,172.05) ;
		%Straight Lines [id:da07627332777970108] 
		\draw [color={rgb, 255:red, 144; green, 19; blue, 254 }  ,draw opacity=1 ][line width=3]    (469.63,178.53) -- (482.14,211.06) ;
		\draw [shift={(484.29,216.66)}, rotate = 248.95999999999998] [fill={rgb, 255:red, 144; green, 19; blue, 254 }  ,fill opacity=1 ][line width=0.08]  [draw opacity=0] (16.97,-8.15) -- (0,0) -- (16.97,8.15) -- cycle    ;
		\draw  [color={rgb, 255:red, 144; green, 19; blue, 254 }  ,draw opacity=1 ][line width=3]  (487.57,183.96) -- (459.42,194.33)(478.68,203.22) -- (468.31,175.07) ;
		%Straight Lines [id:da8106071925604755] 
		\draw [color={rgb, 255:red, 144; green, 19; blue, 254 }  ,draw opacity=1 ][line width=3]    (461.02,137.97) -- (476.53,169.18) ;
		\draw [shift={(479.2,174.56)}, rotate = 243.57] [fill={rgb, 255:red, 144; green, 19; blue, 254 }  ,fill opacity=1 ][line width=0.08]  [draw opacity=0] (16.97,-8.15) -- (0,0) -- (16.97,8.15) -- cycle    ;
		\draw  [color={rgb, 255:red, 144; green, 19; blue, 254 }  ,draw opacity=1 ][line width=3]  (479.39,141.69) -- (452.34,154.65)(472.35,161.7) -- (459.38,134.64) ;
		%Straight Lines [id:da9723260187518871] 
		\draw [color={rgb, 255:red, 144; green, 19; blue, 254 }  ,draw opacity=1 ][line width=3]    (489.01,144.97) -- (516.98,165.77) ;
		\draw [shift={(521.79,169.35)}, rotate = 216.64] [fill={rgb, 255:red, 144; green, 19; blue, 254 }  ,fill opacity=1 ][line width=0.08]  [draw opacity=0] (16.97,-8.15) -- (0,0) -- (16.97,8.15) -- cycle    ;
		\draw  [color={rgb, 255:red, 144; green, 19; blue, 254 }  ,draw opacity=1 ][line width=3]  (507.08,139.96) -- (488.83,163.77)(509.86,160.99) -- (486.05,142.75) ;
		%Straight Lines [id:da10016253063749248] 
		\draw [color={rgb, 255:red, 144; green, 19; blue, 254 }  ,draw opacity=1 ][line width=3]    (442.49,171.73) -- (451.01,205.52) ;
		\draw [shift={(452.47,211.34)}, rotate = 255.86] [fill={rgb, 255:red, 144; green, 19; blue, 254 }  ,fill opacity=1 ][line width=0.08]  [draw opacity=0] (16.97,-8.15) -- (0,0) -- (16.97,8.15) -- cycle    ;
		\draw  [color={rgb, 255:red, 144; green, 19; blue, 254 }  ,draw opacity=1 ][line width=3]  (459.66,179.27) -- (430.47,186.18)(448.52,197.32) -- (441.61,168.13) ;
		
		% Text Node
		\draw (183,96) node    {$\textcolor[rgb]{0.25,0.46,0.02}{\alpha \approx \pi }$};
		% Text Node
		\draw (224.33,229) node    {$\textcolor[rgb]{0.82,0.01,0.11}{\beta \ \approx \ }\textcolor[rgb]{0.82,0.01,0.11}{\frac{\pi }{2}}$};
		% Text Node
		\draw (488,96) node    {$\textcolor[rgb]{0.25,0.46,0.02}{\alpha \approx \pi }$};
		% Text Node
		\draw (521.33,230) node    {$\textcolor[rgb]{0.82,0.01,0.11}{\beta \ \approx \ }\textcolor[rgb]{0.82,0.01,0.11}{\frac{\pi }{2}}$};

		\end{tikzpicture}

		\caption{
			On the above left diagram, a singularity $1/4$ has been placed around an angle $\alpha$ close to $\pi$. On the above right diagram, no singularity has been placed around the same angle $\alpha$. We see that for this angle $\alpha$, placing no, i.d. a zero degree, boundary singularity on it implies less rotations on the computed cross field.
			Conversely, on the below left diagram, a singularity $1/4$ has been placed around an angle $\beta$ close to $\frac{\pi}{2}$. On the below right diagram, no singularity has been placed around the same angle $\beta$. We see that for this angle $\beta$, placing a 1/4 boundary singularity on it implies less rotations on the computed cross field.}
		\label{fig:WithOrWithout} 
	\end{figure}
	
	To deal with boundary singularities there are two possibilities: to place the singularities as we do with the problem described in the section about further thoughts~\ref{FurThou} or to focus only on the internal singularities but to count the boundary singularities via some boundary angular jumps as we will do in the section where we give a new interpretation of the Poincaré-Hopf theorem \ref{GenPH}.
	
	The angles proposed in the literature to place boundary singularities only try to limit the rotation of the cross field around these boundary corners as it is illustrated on  Figure~\ref{fig:WithOrWithout}.
	However, to respect the Poincaré-Hopf theorem while limiting the rotation on a computed cross field, we believe that internal singularities must be taken into account.
	
	In this section, informal calculations are used to justify the range of boundary angles that are considered as singularities as well as their orders depending on the boundary angles and the computational geometry. We believe that the proposed energy to be minimized is a natural extension of the Ginzburg-Landau energy for angular domains.
	This energy let the internal singularities affect the angle ranges on which a boundary singularity must be placed in order to minimize this boundary Ginzburg-Landau energy.
	The classical k$\pi/2$ angles as boundary singularities belongs to the expected singularity degrees. As we will see, the proposed ranges depend on whether there are too many or too few boundary singularities to match the Euler-Poincaré characteristic of the domain. Intuitively, if we have not enough positive boundary singularities to match the Euler-Poincaré characteristic, we must place boundary positive singularities on larger range of angles and vice-versa. Keep in mind that the sums of internal and boundary singularities are strongly linked by the Poincaré-Hopf theorem \eqref{eq:PoinHopf}.
	
	We will see three scenarios appear. 
	In terms of implementation, the idea is to compute the energy implied by those three scenarios and to choose the ranges of angle for boundary singularities that correspond to the scenario of less energy.
	
	Note that for quad mesh purpose, boundary singularity of degree 1/2 are of no interest and must be replaced by one 1/4 boundary singularity and one 1/4 internal singularity close to the boundary one.
	
	\subsubsection{Ginzburg-Landau energy for piecewise-smooth closed boundaries}
	
	Given a domain $\Omega$ with a piecewise $C^1$ boundary $\partial \Omega$, we consider a cross field, whose elements are in $SO(2)/C_4$, where $C_4$ is the rotation group of the square (square is preserved by rotation of $\{\frac{j}{4}2\pi | j \in \{0, ..., 3 \} \}$).\
	
	We assume that the angles formed at the boundary of $\Omega$ are $\alpha_1, \ldots, \alpha_m$ and taken between $0$ and $2\pi$. 
	We let $k_1, \ldots, k_m$ be the indices of the singularities inserted in those angles.

	Note that this may include singularities of index $0$ which are just a regular crosses. To respect the Poincar\'e-Hopf theorem, certain singularities have to be placed in the domain. We let $l_1, \ldots l_n$ be their indices.
	
	On one hand, the singular energy induced by each featured angle $\alpha_i$ of the boundary is given by
	\[
	\sum\limits_{i=1}^m \frac{(\pi - \alpha_i - k_i 2\pi )^2}{2 \alpha_i}.
	\]
	The idea is that the rotational energy needed to not place a singularity on some angle is given by the linear variation induced on the cross. This variation is squared and normed with the size of the angle in order to construct an energy. The $1/2$ term is arbitrary and stems from the fact that we described some fictive kinetic energy.
	This way, placing a 1/4 singularity on a flat edge will cost: $\frac{\pi^2}{4} \frac{1}{2\pi} = \frac{\pi}{8}$ in terms of our proposed Ginzburg-Landau like energy. This energy also implies that put no singularity on a flat angle or a 1/4 singularity on a $\pi/2$ angle or a -1/4 singularity on a $3\pi/2$ angle of the external boundary will result in 0 singular energy (see Figure \ref{fig:angleCrosses}).\\
	
	On the other hand, the singular energy induced by all the inner singularity is given by
	\[
	\frac{1}{2\pi}\left( \sum\limits_{i=1}^n \left(2\pi| l_i |\right)^2 \right).
	\]

	since this is the total rotational energy, computed as above, created in the domain.
	This way, placing an extra 1/4 singularity in the domain will cost: $\frac{1}{2\pi} \left(2\pi \frac{1}{4}\right)^2 = \frac{\pi}{8}$ in terms of our proposed Ginzburg-Landau like energy.\\
	
	The total singular energy is then given by
	\[
	\sum\limits_{i=1}^m \frac{(\pi - \alpha_i - k_i 2\pi)^2}{2 \alpha_i} + \frac{1}{2\pi}\left( \sum\limits_{i=1}^n \left(2\pi| l_i |\right)^2 \right).
	\]
	
	Note that if we have $m$ boundary singularities of index $k_i$ and $n$ internal singularities of degrees $l_j$, the Poincaré-Hopf theorem implies that
	\[
	\sum_{i=1}^m k_i + \sum_{j=1}^n l_j = \chi
	\]
	where $\chi$ is the Euler-Poincaré characteristic of the computational domain.\\
	
	\subsubsection*{Optimality conditions}
	
	We first consider the case of balanced boundary singularities: $\sum_{i=1}^m k_i = \chi$. It is the case for example if we have only four boundary angles of $90^\circ$ that we associate with four 1/4 boundary singularities with $\chi = 1$.
	
	\begin{theo}[Balanced boundary singularities]\label{theo:balanced}
		If we have an optimal choice of boundary singularities in a case where the sum of the index of the boundary equals the Euler-Poincaré characteristic of the computational domain such as the Poincaré-Hopf theorem is fulfilled, the limits angles $\alpha$ between which a boundary singularity of degree $k$ mst be placed are given by the following relationship:
		\begin{equation}
			\frac{ 1 - \frac{1}{4}}{2} - \left(1 + \frac{1}{8}\right)\frac{\alpha}{2\pi} \leq  k \leq \frac{ 1 + \frac{1}{4}}{2}  - \left(1 - \frac{1}{8} \right)\frac{\alpha}{2\pi}.
		\end{equation}
	\end{theo}
	\begin{proof}
		See the appendix section \ref{demo:balanced}.
	\end{proof}

	In particular, we can decide to place no singularity ($k_f = 0$) when
	\[
	\frac{2\pi}{3}\leq \alpha_f \leq \frac{10\pi}{7}
	\]
	but the intervals overlap so we can decide to place a 1/4 singularity ($k_f = 1/4$) when
	\[
	\frac{2\pi}{9}\leq \alpha_f \leq \frac{6\pi}{7}
	\]
	thus in the interval $2\pi/3 \leq x \leq 6\pi/7$ we can place either a 1/4 or no singularity. Note in particular that the heuristic classical limit angle of $3\pi/4$ belongs to that interval.\\
	
	For the two other cases, there is no overlaping between the intervals on which place a singularity.\\

	We now consider the case of rare boundary singularities: $\sum_{i=1}^m k_i < \chi$. It is the case for example if we have three boundary angles of $90^\circ$ that we associate with three 1/4 boundary singularities with $\chi = 1$. 
	\begin{theo}[Scarce boundary singularities]\label{theo:rare}
		If we have an optimal choice of boundary singularities in a case where the sum of the index of the boundary singularities is lower than the Euler-Poincaré characteristic of the computational domain, the limits angles $\alpha$ between which a boundary singularity of degree $k$ must be placed are given by the following relationship:
		\begin{equation}
		\frac{ 1 - \frac{1}{4}}{2} - \left(1 - \frac{1}{8}\right)\frac{\alpha_{f}}{2\pi} \leq  k_f \leq \frac{ 1 + \frac{1}{4}}{2}  - \left(1 - \frac{1}{8} \right)\frac{\alpha_{f}}{2\pi}.
		\end{equation}
	\end{theo}
	\begin{proof}
		See the appendix section \ref{demo:rare}.
	\end{proof}
	In particular, we can decide to place no singularity ($k_f = 0$) when
	\[
	\frac{6\pi}{7}\leq \alpha_f \leq \frac{10\pi}{7}
	\]
	and we decide to place a 1/4 singularity ($k_f = 1/4$) when
	\[
	\frac{2\pi}{7}\leq \alpha_f \leq \frac{6\pi}{7}.
	\]
	Note in particular that the heuristic classical limit angle of $3\pi/4$ belongs to that interval.
	This time the intervals follow each other without overlapping:
	\[
		0 \overset{k=\frac{2}{4}}{\longleftrightarrow} \frac{2\pi}{7} \overset{k=\frac{1}{4}}{\longleftrightarrow} \frac{6\pi}{7} \overset{k=0}{\longleftrightarrow} \frac{10\pi}{7} \overset{k=-\frac{1}{4}}{\longleftrightarrow} 2\pi.
	\]
	We see that we place more easily positive singularity than negative ones.\\
	
	We now consider the case of plentiful boundary singularities: $\sum_{i=1}^m k_i > \chi$. It is the case for example if we have five boundary angles of $90^\circ$ that we associate with five 1/4 boundary singularities with $\chi = 1$.
	
	\begin{theo}[Excess of boundary singularities]\label{theo:excess}
		If we have an optimal choice of boundary singularities in a case where the sum of the index of the boundary singularities is greater than the Euler-Poincaré characteristic of the computational domain, the limits angles $\alpha$ between which a boundary singularity of degree $k$ must be placed are given by the following relationship:
		\begin{equation}
		\frac{ 1 - \frac{1}{4}}{2} - \left(1 + \frac{1}{8}\right)\frac{\alpha_{f}}{2\pi} \leq  k_f \leq \frac{ 1 + \frac{1}{4}}{2}  - \left(1 + \frac{1}{8} \right)\frac{\alpha_{f}}{2\pi}.
		\end{equation}
	\end{theo}
	\begin{proof}
		See the appendix section \ref{demo:excess}.
	\end{proof}
	
	In particular, we can decide to place no singularity ($k_f = 0$) when
	\[
	\frac{2\pi}{3}\leq \alpha_f \leq \frac{10\pi}{9}
	\]
	and we decide to place a 1/4 singularity ($k_f = 1/4$) when
	\[
	\frac{2\pi}{9}\leq \alpha_f \leq \frac{2\pi}{3}.
	\]
	Note in particular that the heuristic classical limit angle of $3\pi/4$ belongs to that interval.
	This time the intervals follow each other without overlapping:
	\[
	0 \overset{k=\frac{2}{4}}{\longleftrightarrow} \frac{2\pi}{9} \overset{k=\frac{1}{4}}{\longleftrightarrow} \frac{2\pi}{3} \overset{k=0}{\longleftrightarrow} \frac{10\pi}{9} \overset{k=-\frac{1}{4}}{\longleftrightarrow}\frac{14\pi}{9} \overset{k=-\frac{2}{4}}{\longleftrightarrow} 2\pi.
	\]
	We see that we place more easily negative singularities than positive ones.\\
	
%	In terms of implementation, on one hand, if we expect to have too few boundary singularities so that the sum of their indices is smaller than the Euler-Poincaré characteristic of the domain but that when we do the computation, we have too many of them, changes the choices of ranges will always help to get closer to the Euler-Poincaré characteristic. For example, if we have three angles of $\pi/4$ with $\chi = 1$, and that we suppose wrongly that we will have too few boundary singularities, we will associates them with three $2/4$ singularities and have $|\chi - \sum k_i|= 1/2$. When we change the ranges, they become three $1/4$ and we have $|\chi - \sum k_i|= 1/4$.\\
%	
%	On the other hand, if we expect to have too many boundary singularities so that the sum of their indices is greater than the Euler-Poincaré characteristic of the domain but that when we do the computation, we have not enough of them, changes the choices of ranges will always help to get closer to the Euler-Poincaré characteristic. For example, if we have one angle of $\pi/4$ with $\chi = 1$, and that we suppose wrongly that we will have too many boundary singularities, we will associates it with one $1/4$ singularity and have $|\chi - \sum k_i|= 3/4$. When we change the ranges, the singularity become one $1/2$ and we have $|\chi - \sum k_i|= 1/2$.\\
	
	The conclusion of our calculations is that when we have enough boundary singularities, the user can choose the type of the singularities on some ranges; when we have not enough boundary singularities, the placement of positive singularities occurs more often and when we have more than enough boundary singularities, the placement of negative singularities occurs more often. With the implementation, if we made a wrong supposition, choose the another set of ranges will always be a better choice or an equivalent one if we are in the situation where $\sum k_i = \chi$.
%	\comment{Plutôt que deviner le nombre de singularités, on essaie de trouver les singularités dans les trois scénarios (singularités rares, équilibrées et abondantes) et on voit lequel fonctionne. En pratique, il n’y en aura en général qu’un.}	
%	
	% For figure citations, please use "Fig" instead of "Figure".

	% Place figure captions after the first paragraph in which they are cited.

	% Results and Discussion can be combined.
	\section{Results and Discussion}
	
	We will show and compare different results, as well as seeing that it is sometimes necessary to add constraints to our basic formulation to obtain the desired results. We will also give a new interpretation of the Poincaré-Hopf theorem that is compatible with the optimizations proposed in the section \ref{GLPiecewise}. This interpretation helps a user to properly apply our method without necessarily knowing explicitly the configuration of the boundary singularities.
	
	\subsection{Analysis of the numerical results} \label{resultsAnalysis}
	
	%All this is encoded in Python using the GMSH python API which can generate the Jacobians and shape functions of the initial triangular mesh.
	
	In figure \ref{fig:firstFormulationDisk}, we can see that our formulation in $\mathbb{R}^2$ does not result in a vector field in $S^1$ in the computed
	cross field and in the corresponding vector field for which our algorithm has converged in 13 fast iterations. Furthermore, we can easily move the singularities (see figure \ref{fig:diskV4}).
	
	\begin{figure}[ht] 
		\begin{minipage}[b]{.5\linewidth}
			\centering
			\includegraphics[width=.8\linewidth]{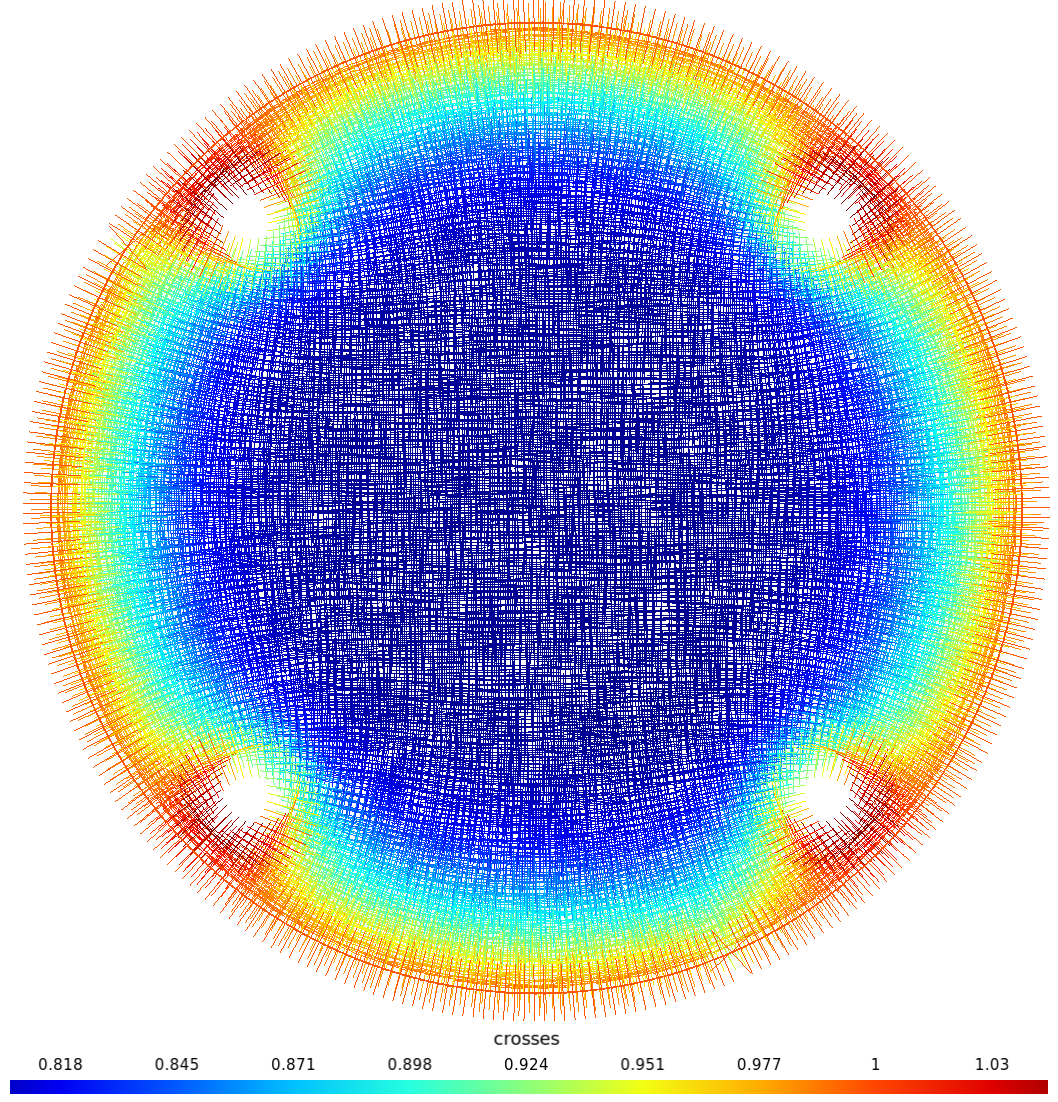}
		\end{minipage}%%
		\begin{minipage}[b]{.5\linewidth}
			\centering
			\includegraphics[width=.8\linewidth]{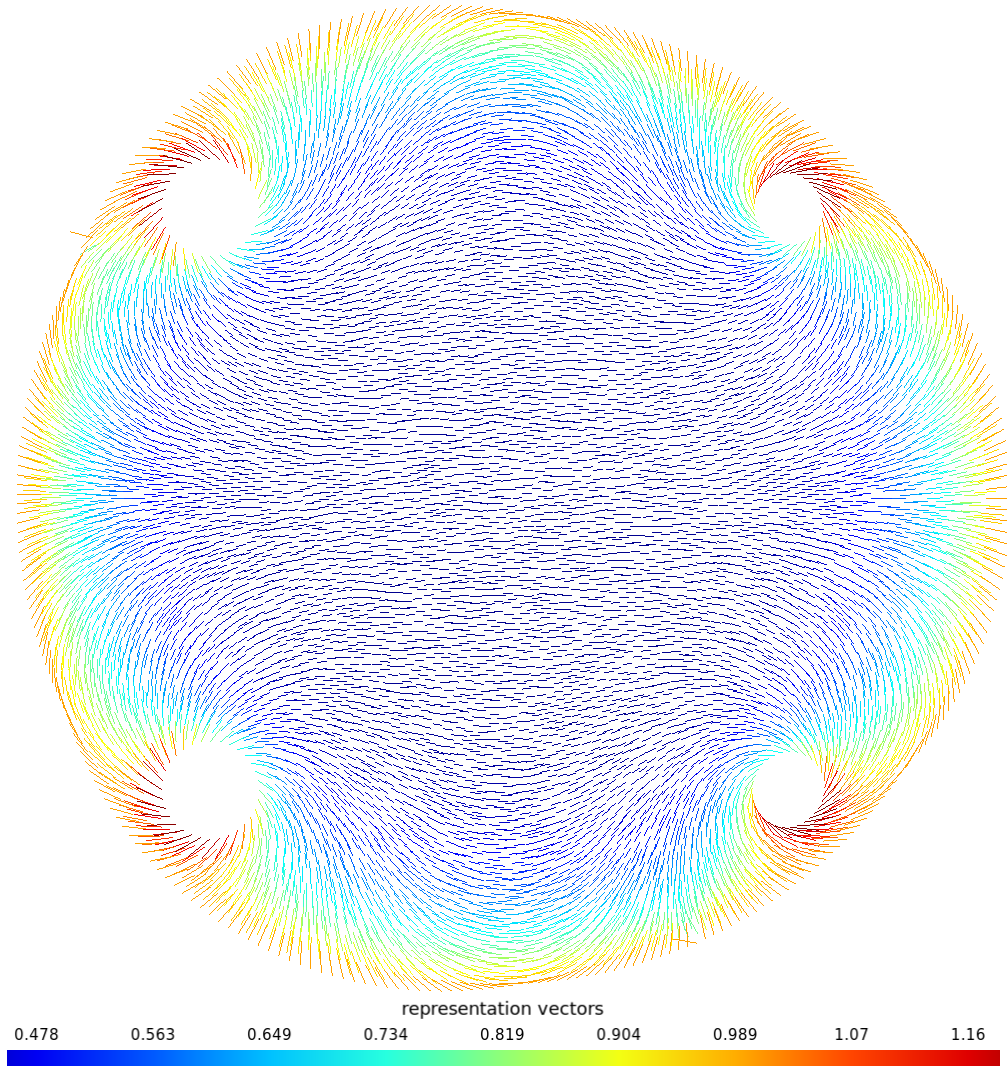} 
			%\vspace{4ex}
		\end{minipage}%% 
		\caption{Cross field for a disk of radius 1 with four holes of radius $0.1$ with centers on a circle of radius $0.6$ at left and the corresponding vector field at right.}
		\label{fig:firstFormulationDisk} 
	\end{figure}
	
	\begin{figure}[ht] 
		\begin{minipage}[b]{.5\linewidth}
			\centering
			\includegraphics[width=.8\linewidth]{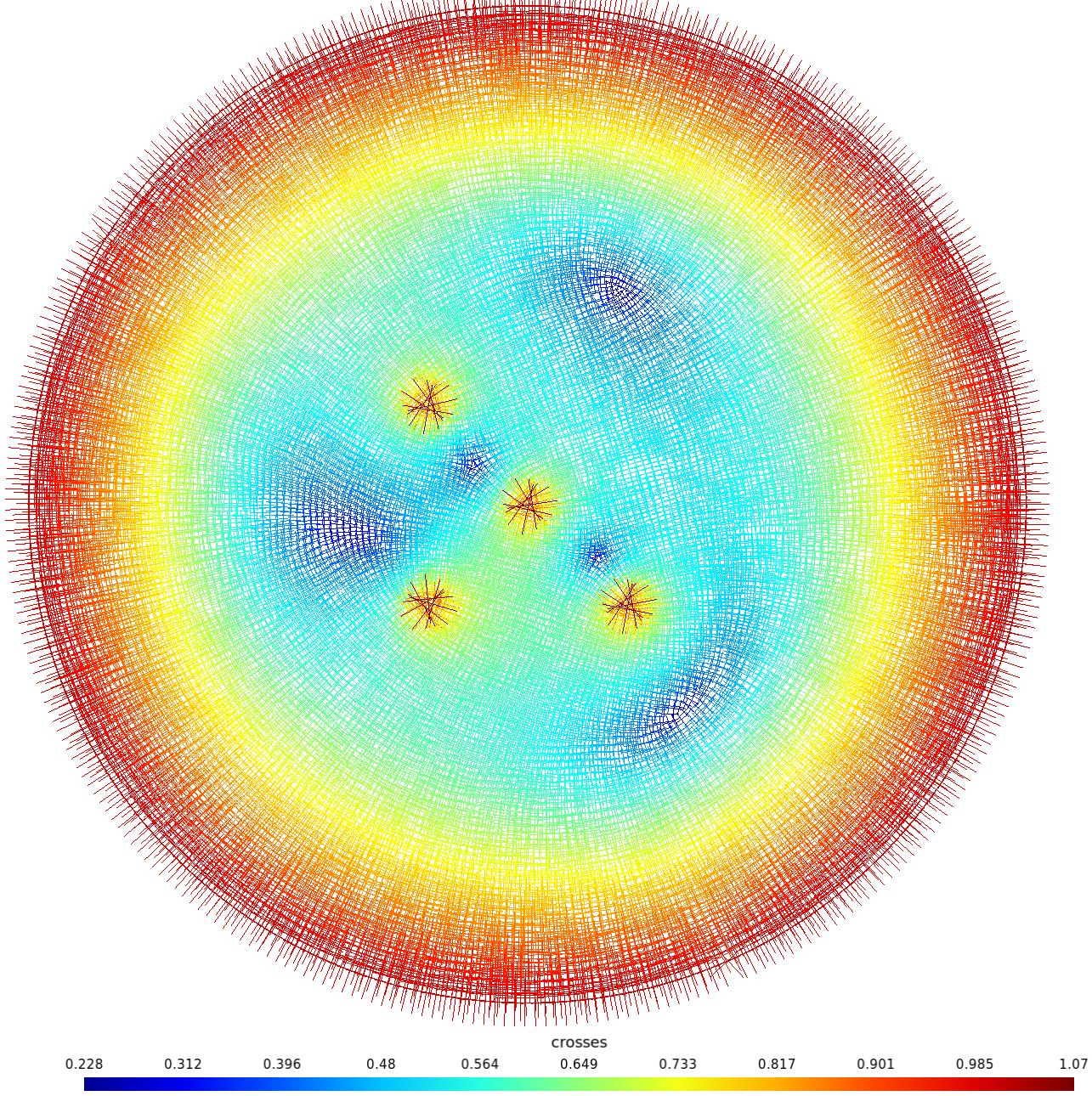} 
			%\vspace{4ex}
		\end{minipage}%%
		\begin{minipage}[b]{.5\linewidth}
			\centering
			\includegraphics[width=.8\linewidth]{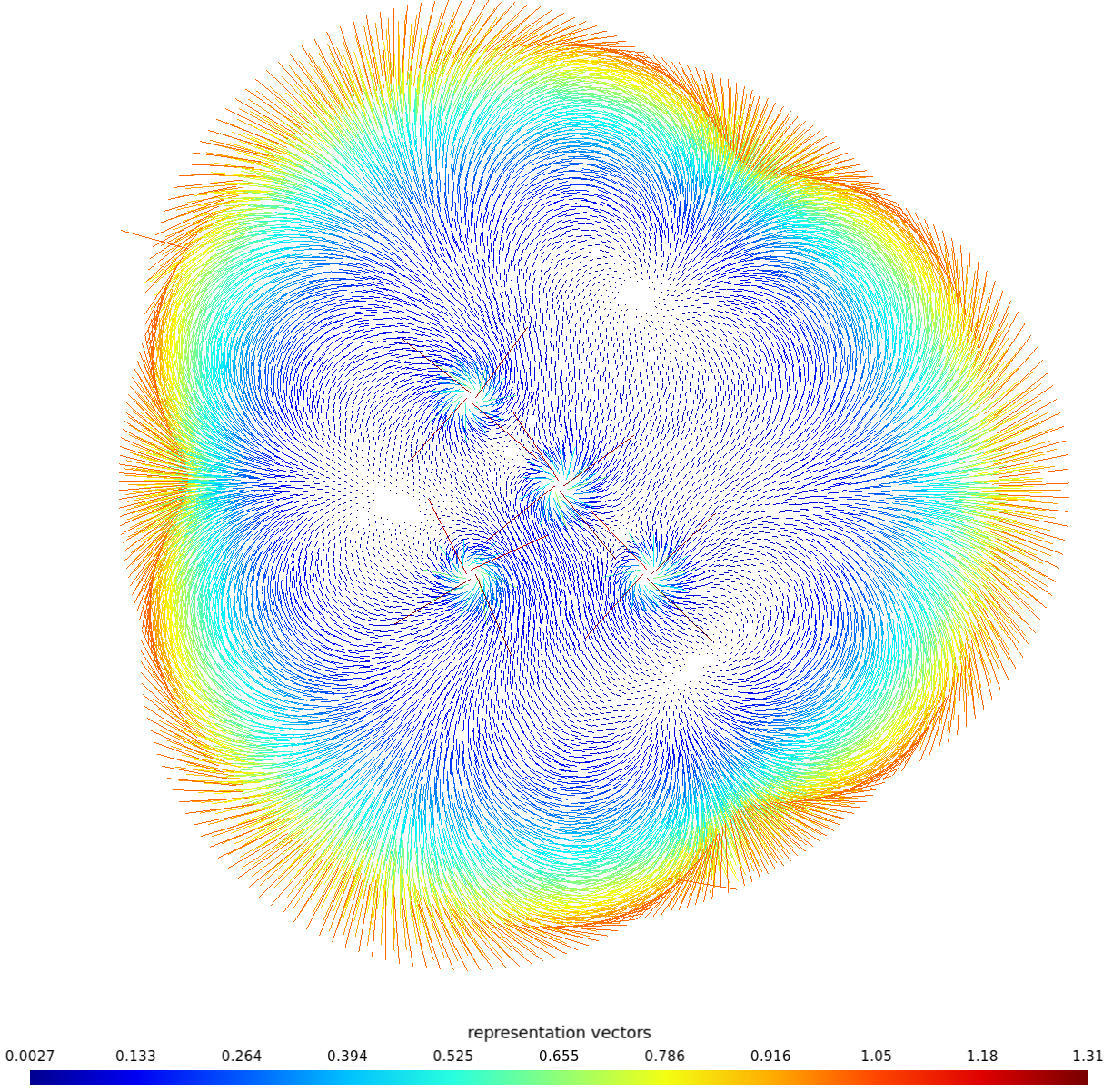} 
			%\vspace{4ex}
		\end{minipage}%% 
		\caption{Singularities placed via disks of radius equals to $0.01$ placed on $(0, 0)$, $(-0.2, 0.2)$, $(0.2, -0.2)$ and $(-0.2, -0.2)$.}
		\label{fig:diskV4} 
	\end{figure}

	The method converges quickly and in never more than 30 iterations in the tested cases. Furthermore, the Newton-Raphson scheme could be replaced with methods better adapted to solve linear objective with quadratic constraints. Even with the Newton-Raphson scheme the method can converge for any topologically consistent configuration of the hole centers and sizes. On one hand, larger holes will imply less rotations and thus less energy in the computed cross fields but also less crosses in the computed field which will thus imply quicker convergence as long as the configuration of singularity is topologically coherent. On the other hand, large holes could imply more norm disparity of the crosses on their outline and thus a worse estimation of the degree on it.
	Furthermore, it is more difficult to extract a block-structured quandrangular mesh from a domain perforated with large holes.
	
	%\begin{figure}[ht] 
	%	\begin{minipage}[b]{.5\linewidth}
	%		\centering
	%		\includegraphics[width=\linewidth]{diskV3Crosses.png} 
	%		%\vspace{4ex}
	%	\end{minipage}%%
	%	\begin{minipage}[b]{.5\linewidth}
	%		\centering
	%		\includegraphics[width=\linewidth]{diskV3Vectors.png} 
	%		%\vspace{4ex}
	%	\end{minipage}%% 
	%	\caption{Singularities placed via disks of radius equals to $0.01$ placed on $(0.3, 0.3)$, $(-0.2, 0.2)$, $(0.3, -0.3)$, $(0.3, 0.3)$.}
	%	\label{fig:diskV3} 
	%\end{figure}
	Among the results obtained, two are particularly worth mentioning.
	
	First, for methods based on classical Ginzburg-Landau theories, the generation of cross fields in a flat ellipse leads to singularities aligned on a single line while the perforation method allows to keep a rectangular configuration of singularities similar to that of the disc (see figure \ref{fig:ellipse}).
	
	\begin{figure}[ht] 
		\begin{minipage}[b]{\linewidth}
			\centering
			\includegraphics[width=.8\linewidth]{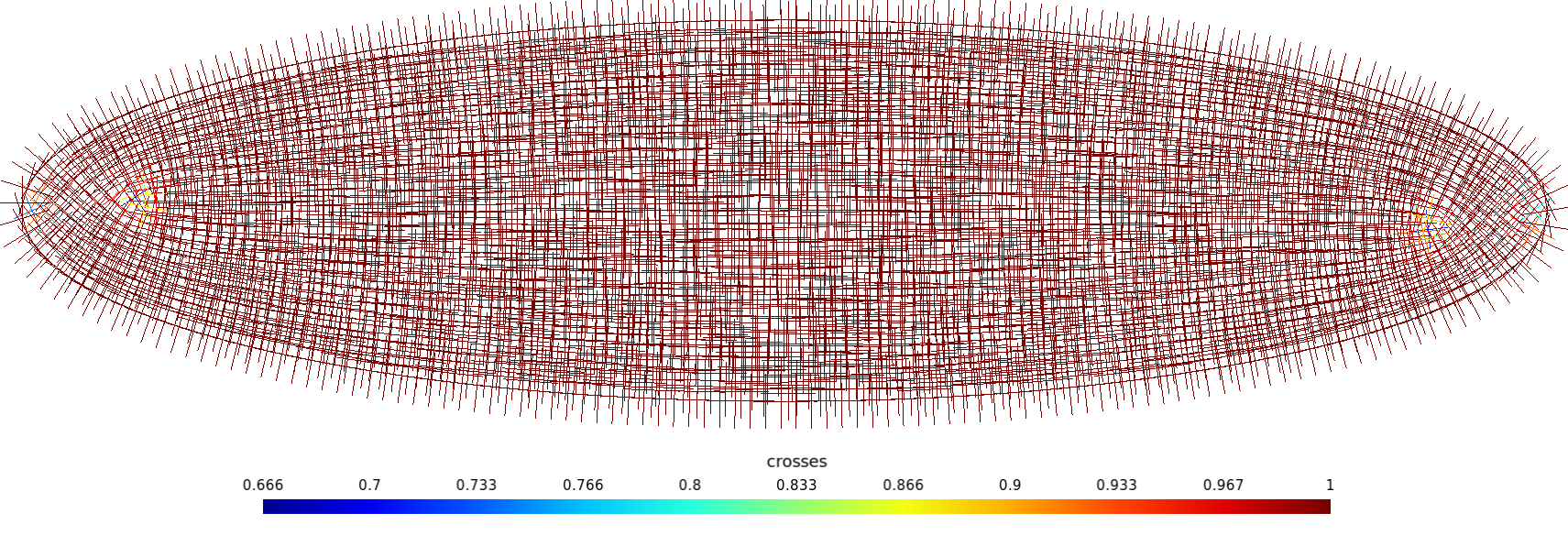}
			\vspace{.3cm}
		\end{minipage}
		\begin{minipage}[b]{\linewidth}
			\centering
			\includegraphics[width=.8\linewidth]{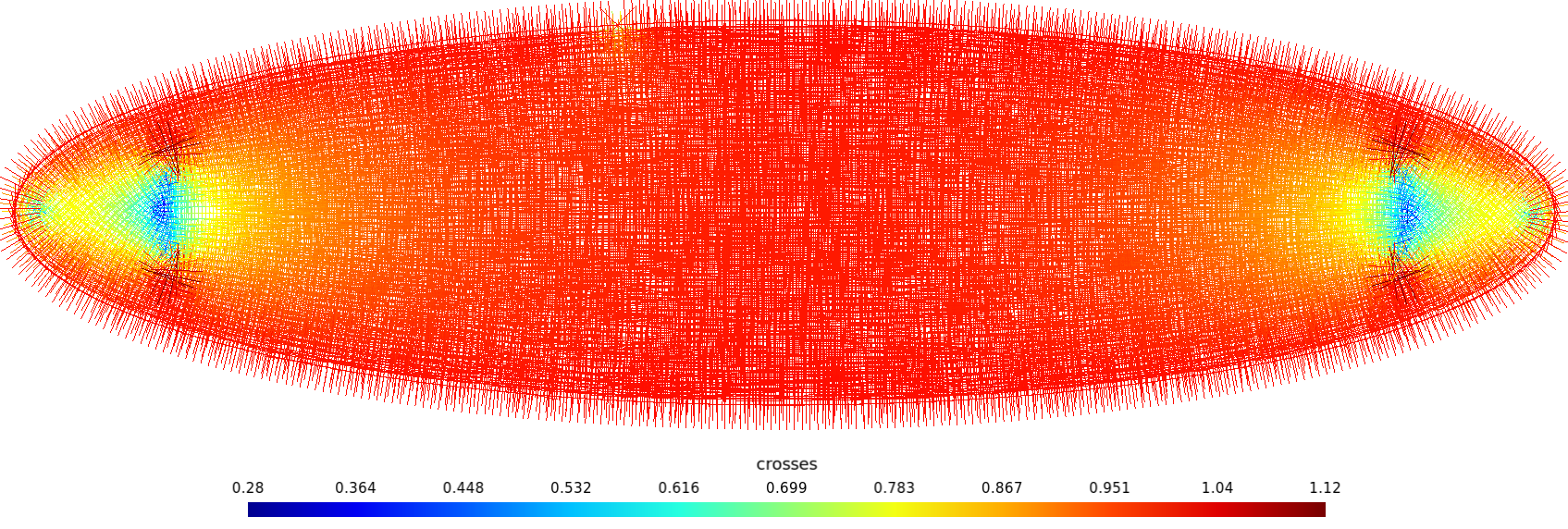} 
		\end{minipage}%% 
		\caption{Ginzburg-Landau based singularity configuration (above) compared with a perforated singularity configuration (below).}
		\label{fig:ellipse} 
	\end{figure}
	
	Second, we noticed that for a choice of singularities coherent and close to the configuration prescribed by the Ginzburg-Landau functional \eqref{eq:GL}, the condition 
	on a hole containing a
	singularity can be imposed on only one triangle of the starting triangular mesh (see figure \ref{fig:oneTriangle}). If one departs too far from the Ginzburg-Landau configuration, disparities of norms between the crosses of the triangle may imply an imposition of unsatisfactory degree.
	
	\begin{figure}[ht] 
		\begin{minipage}[b]{.5\linewidth}
			\centering
			\includegraphics[width=.8\linewidth]{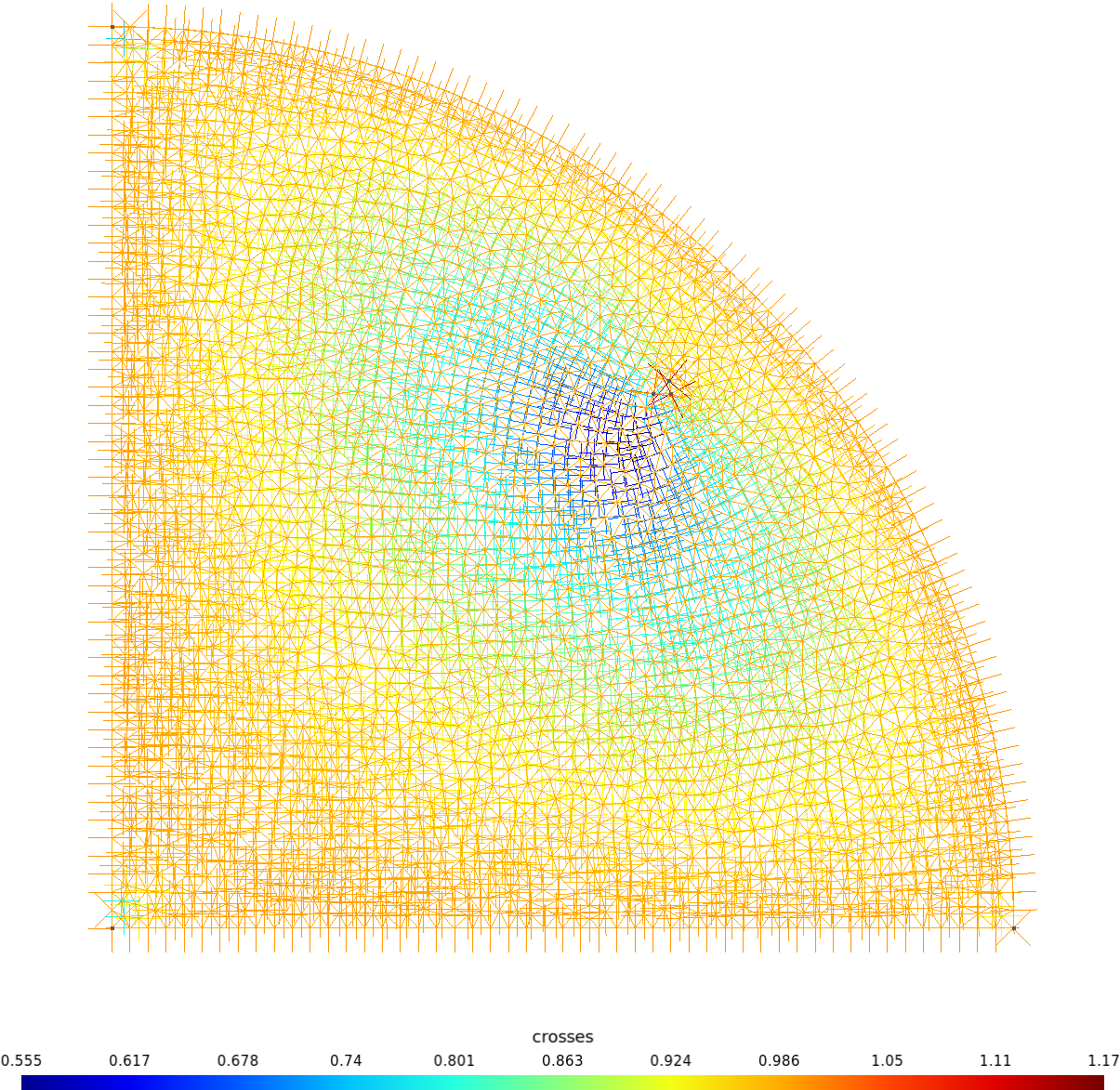} 
		\end{minipage}%%
		\begin{minipage}[b]{.5\linewidth}
			\centering
			\includegraphics[width=.8\linewidth]{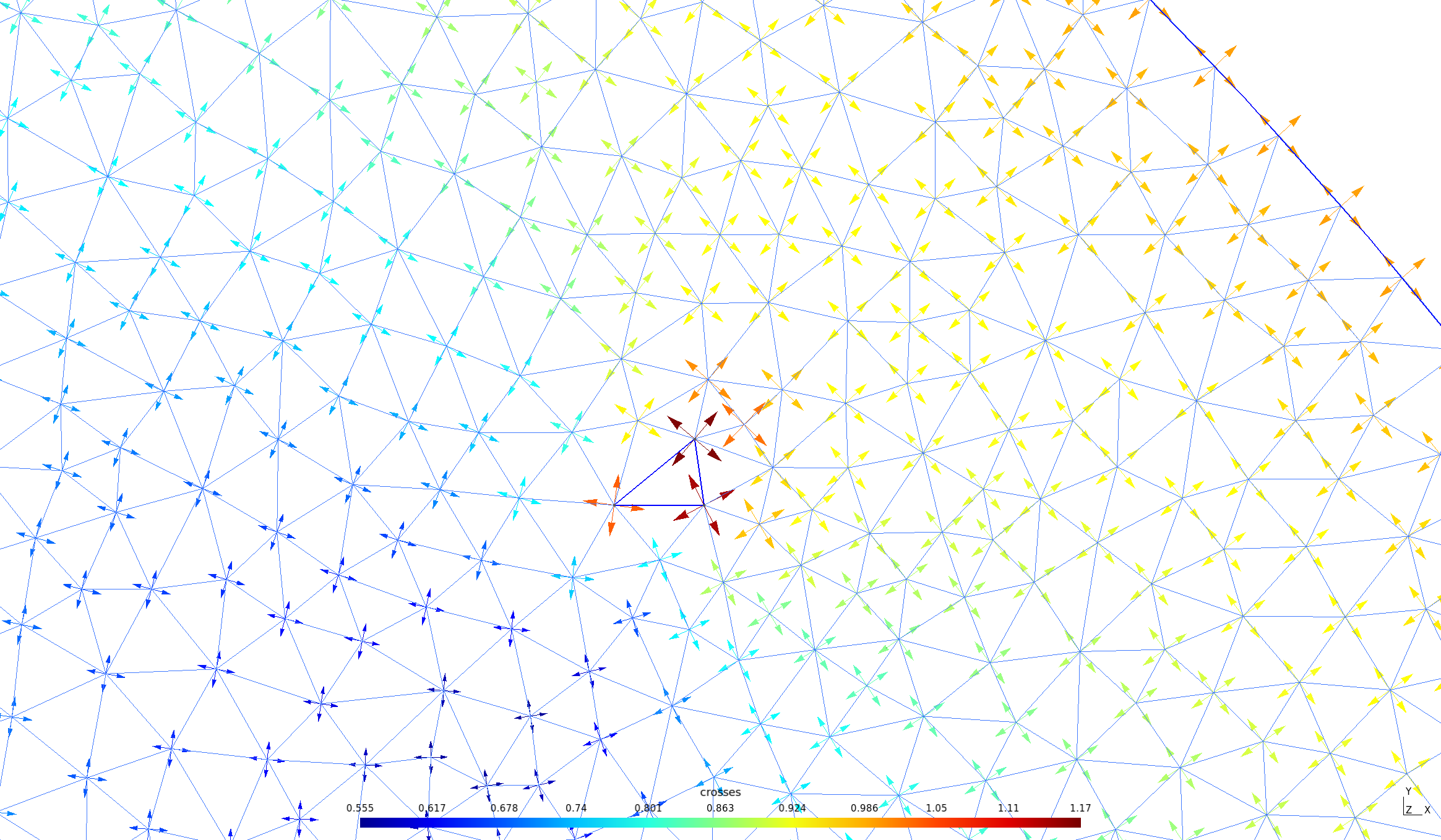}
			\vspace{.3cm}
		\end{minipage}%% 
		\caption{Generation of a cross field in which a singularity is placed on a single triangle of the underlying mesh.}
		\label{fig:oneTriangle} 
	\end{figure}
	
	%Unfortunately, the effectiveness of this method is linked to its non-generality, which implies at least two flaws.
	
	When we want to place a non-minimal number of singularities, we may have to force the crosses on some holes to have a norm equal to 1 via additional Lagrange multipliers (see the blue circle of the figure \ref{fig:normHole}) to make the solution converges towards the prescribed degrees. We exceptionally decided to work with the domain $G$ (without perforated holes as singularities) rather than $G_\rho$ in order to better visualize the different types of singularities on figure \ref{fig:normHole}. We also see that a large blue circle implies more flexibility for the actual position of the placed singularity within the hole. Note that we have here singularities on the corners boundary $\partial G$ and that the Poincar\'e-Hopf theorem \eqref{eq:PoinHopf} is well respected thanks to those singularities. Being very localized, the added constraints slow down the method only moderately.
	
	\begin{figure}[ht] 
		\centering
		\includegraphics[width=\linewidth]{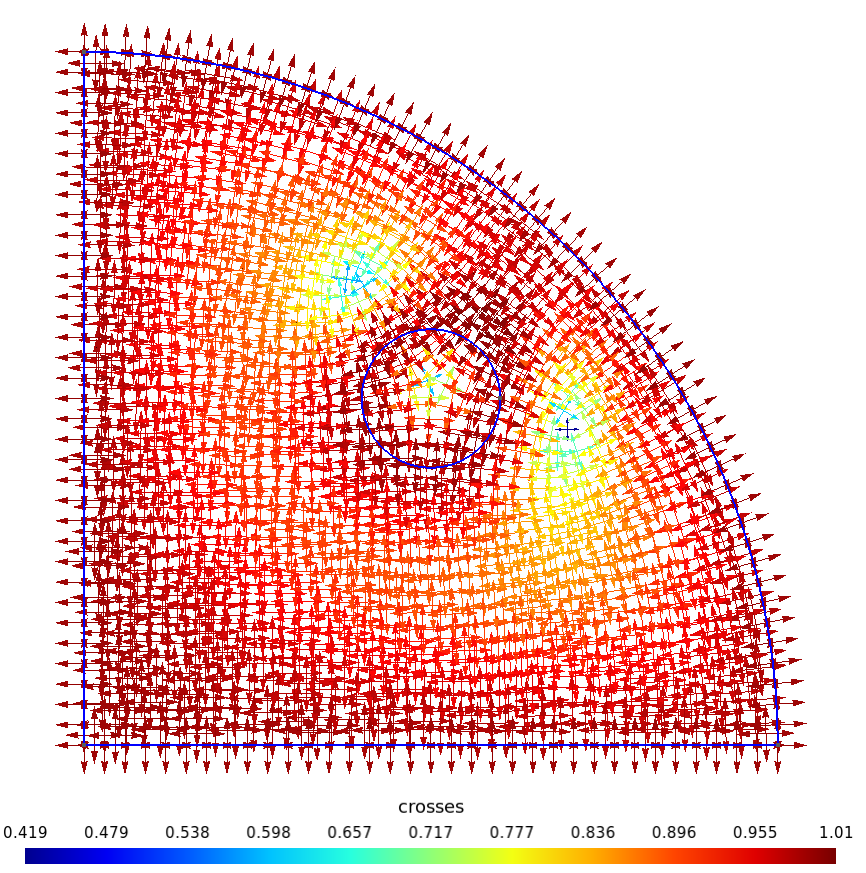} 
		\caption{Forced singularities on a disk quarter. The constraint to obtain a singularity -1/4 has been placed on the blue circle shown above. All crosses are unitary on it. We see two compensatory 1/4 singularities appearing.}
		\label{fig:normHole} 
	\end{figure}
	A flaw of the method is that the more we move from the solution associated with the minimizers of the Ginzburg-Landau functional, the more we have disparities on the norms of our crosses which can make the process of extracting quadrangular meshes from it more difficult (see figure and \ref{fig:diskGL}).	
	%\begin{figure}[ht] 
	%	\centering
	%	\includegraphics[width=.9\linewidth]{diskGL.png} 
	%	\caption{Singularities obtained via the Ginzburg-Landau equation \eqref{eq:GL}. The non-linear part of the equation is added via a linearization of the Ginzburg-Landau functional and a Newton-Raphson scheme.}
	%\label{fig:diskGL} 
	%\end{figure}
	A way to counteract this is to add to our method the non-linear penalty of \eqref{eq:GL} via a linearization and a Newton-Raphson scheme that is performed until convergence. All the computations implying this added terms have been performed with $\epsilon = 1/100$. There are two main drawbacks of this added penalty. First, the modified method is far slower. Second, when we try to place a holed singularity too far from the best location for the Ginzburg-Landau functional, a zero norm singular cross of opposed degree sticks to it and the natural zero norm singular cross of \eqref{eq:GL} appears (see figures \ref{fig:diskV4} and \ref{fig:holeAndGL}).
	
	%\begin{figure}[ht] 
	%	\centering
	%	\includegraphics[width=.9\linewidth]{holeAndGL.png} 
	%	\caption{Singularities obtained via the Ginzburg-Landau equation \eqref{eq:GL}. Holes set to obtain singularities of degree 1/4 were placed on $(-0.6, 0.6)$, $(-0.6, -0.6)$ and $(0.6, -0.6)$ and those set for degree -1/4 were placed on  $(-0.2, 0.2)$, $(-0.2, -0.2)$, $(0.5, -0.5)$ and $(0.6, 0.6)$.}
	%	\label{fig:holeAndGL} 
	%\end{figure}
	
	\begin{figure}[ht] 
		\begin{minipage}[b]{.5\linewidth}
			\centering
			\includegraphics[width=.9\linewidth]{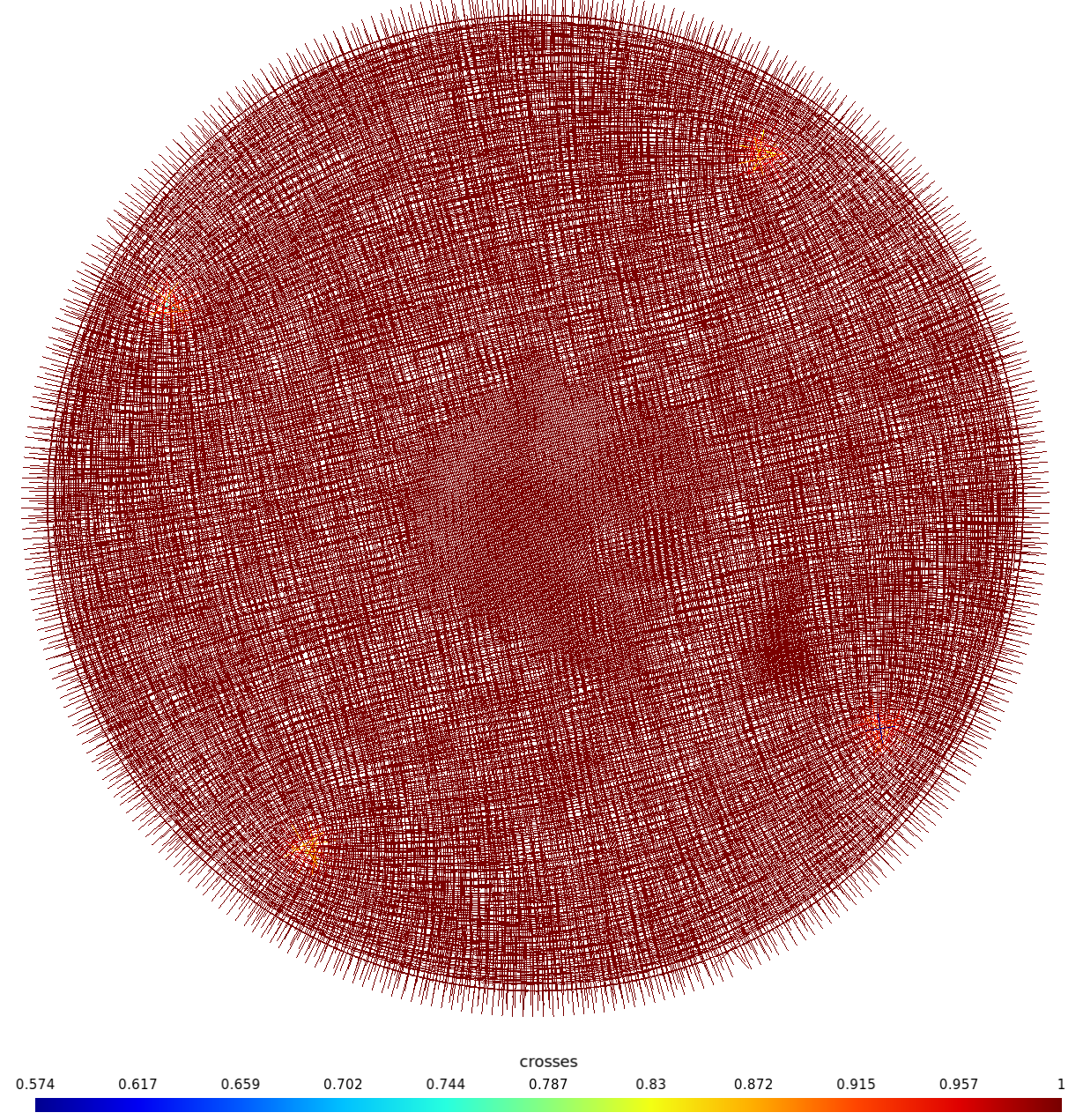}
			\captionsetup{width=.8\textwidth}
			\caption{Singularities obtained via the Ginzburg-Landau equation \eqref{eq:GL} on a disk centered at the origin and of radius 1.}
			\label{fig:diskGL}
			\vspace{.4cm}
		\end{minipage}%%
		\begin{minipage}[b]{.5\linewidth}
			\centering
			\includegraphics[width=\linewidth]{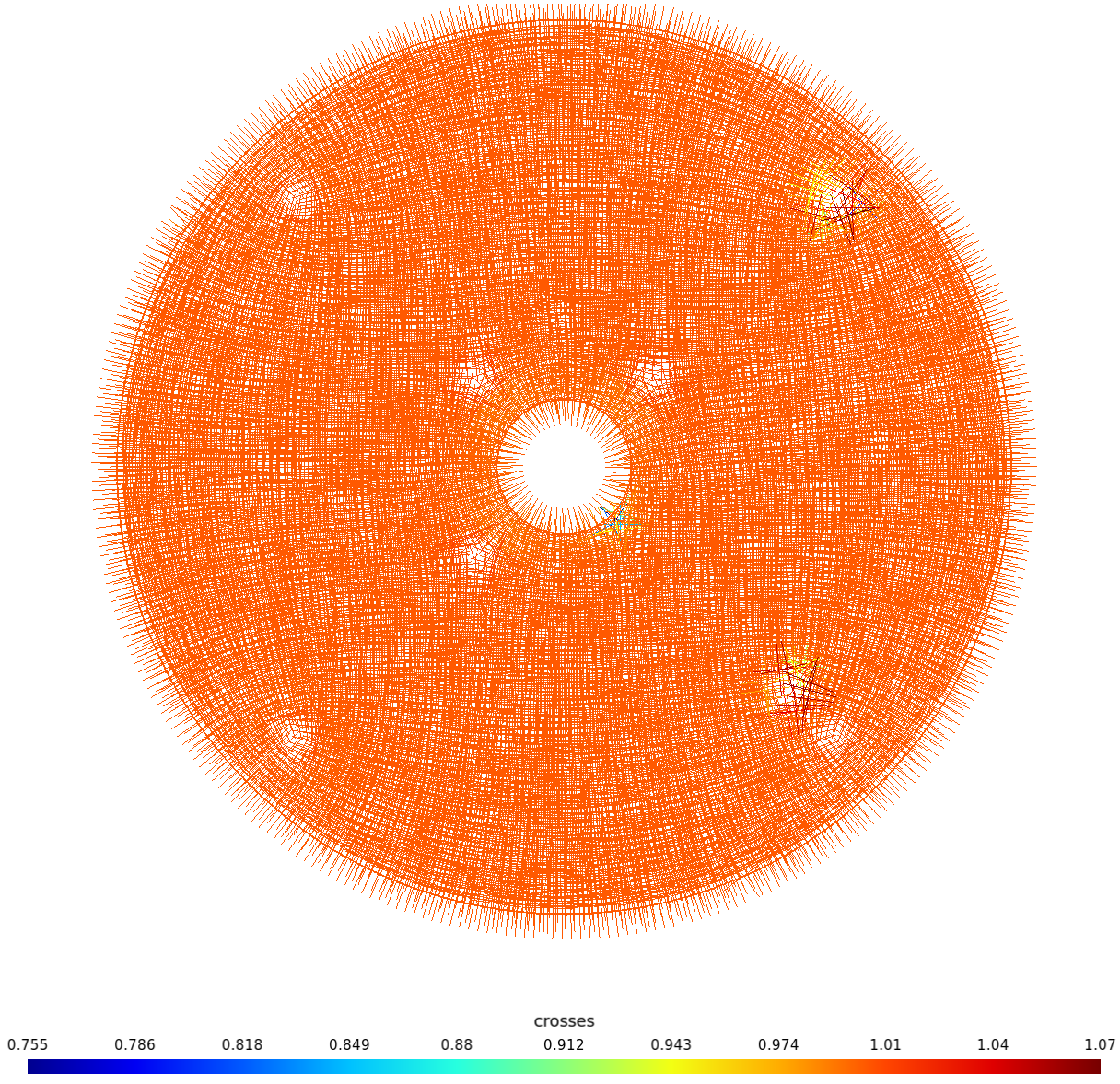}
			\captionsetup{width=\textwidth} 
			\caption{Singularities obtained via the Ginzburg-Landau equation \eqref{eq:GL} with some large visible holes set to obtain singularities of degree 1/4 on a similar disk as the one of figure \ref{fig:diskGL}.}
			\label{fig:holeAndGL}
		\end{minipage}%% 
	\end{figure}
	
	Still there is at least one benefit with this added penalty term. With it, the choices of singularities that do not respect the Poincar\'e-Hopf theorem are permitted without divergence of the method. We can see on the figure \ref{fig:holeAndGL} that a wrong singularity has been placed on a circle centered on $(0.6, 0.6)$ but two opposite singularities appears on the hole to compensate it so as to fulfill the Poincaré-Hopf theorem. 
	We can also see that the wrongly placed singularities on $(0.5, -0.5)$ have been canceled by sticking an opposite singularity on it. Moreover, another singularity was placed on the inner circle. Obviously, such cross fields are no longer regular.
	
	To conclude this section, we propose some comments on the use of the method and its possible extensions.
	On one hand, for a basic use one is invited to set a size for the perforated holes similar to the mesh size of the underlying triangular mesh and with one constraint on each boundary $\gamma_i$ of the perforated holes.
	On the other hand, if you do not have a precise idea of the configuration of the singularities, you can use a single global constraint \eqref{eq:generalConstraint} as well as large circles as the boundaries of the singularities that you may not perforate (as done in the figure \ref{fig:normHole}) in order to more easily extract a quadrangular cut.
	Note also that the choice of a Newton-Raphson scheme is certainly not the most suited and that the alternating direction method of multipliers (ADMM) may be more appropriate given our linear objective with quadratic constraints. A very important application will be to compare energy of different configuration (possibly non-minimal) such as to minimize them (for example with positive and negative singularities).
	It is also important to point out that our method does not allow us to impose boundary singularities or to give information about them. To remedy this, we have extended the Ginzburg-Landau theory in the section \ref{GLPiecewise} and we will give a new interpretation of the Poincaré-Hopf theorem in the section \ref{GenPH}.
	
	\subsection{Generalized Poincar\'e-Hopf theorem}\label{GenPH}
	
	We now give a generalization of the Poincaré-Hopf theorem which is compatible with the previously given optimizations of the boundary singularity degrees. Given a geometry with boundary angles, this generalization allows to quickly know what should be the sum of the degrees of the interior singularities.
	
	The theoretical results obtained in the rich analysis of the energetic Ginzurg-Landau functional do not apply to piecewise smooth domains (see~\cite{BBH}).
	One way to get around this problem is to smooth the corners of a piecewise smooth domain so as to be able to apply the Ginzburg-Landau theory. Unfortunately, clear information about singularity degree for boundary singularities which is seek in the literature (see~\cite{Viertel}) can not be extracted in this way. But in the section \ref{GLPiecewise} we have created a new energy to get over this fact.
	
	On the other hand, multiple generalizations of the Poincar\'e-Hopf theorem have been given in the literature (see~\cite{Beaufort},~\cite{Viertel} and~\cite{Ray2008}). In~\cite{Ray2008} an extension of the Poincar\'e-Hopf theorem to rational indices has first been developed.
	%Crosses are objects invariant by rotation of $k\frac{\pi}{2}$ with $k \in \mathbb{Z}$.
	% Topology is th study of properties that are invariant by continuous deformations (without cutting or gluing anythong), called homotopies. What is even more interesting is the structure of the set of homotopy classes (classes of all objects with the same topology).
	However, these generalizations do not allow to have a priori information on the configuration of singularities to be placed inside angular domains. Indeed, the angles present on the boundary of such domains may or may not correspond to singularities in the resulting mesh.

	In this section, we propose a new formulation of the Poincar\'e-Hopf theorem for N-directional fields to choose topologically coherent singularity configuration to place inside angular domains.
	
	In order to get information on internally placed singularity configurations, we propose a strategy based on a valence term (illustrated on Figure \ref{fig:3pi4}) that sums up the smallest oriented angles such that the object is mapped into an equivalent object on each side of the discontinuity points of the boundary.
	
	\begin{figure}[ht] 
		\centering

		\tikzset{every picture/.style={line width=0.6pt}} %set default line width to 0.75pt        
		
		\begin{tikzpicture}[x=0.75pt,y=0.75pt,yscale=-1,xscale=1]
		
		%Straight Lines [id:da4319569685214921] 
		\draw    (50.2,89.4) -- (180.2,89.8) ;

		%Straight Lines [id:da9908683688048057] 
		\draw    (180.2,89.8) -- (304.5,27) ;

		%Shape: Arc [id:dp7670093423049111] 
		\draw  [draw opacity=0] (157.19,89.74) .. controls (160.19,82.35) and (166.1,76.16) .. (174.12,73.04) .. controls (184.55,68.98) and (195.87,71.15) .. (203.97,77.75) -- (185,101) -- cycle ; \draw  [color={rgb, 255:red, 65; green, 117; blue, 5 }  ,draw opacity=1 ] (157.19,89.74) .. controls (160.19,82.35) and (166.1,76.16) .. (174.12,73.04) .. controls (184.55,68.98) and (195.87,71.15) .. (203.97,77.75) ;
		%Straight Lines [id:da7953096667468448] 
		\draw [color={rgb, 255:red, 144; green, 19; blue, 254 }  ,draw opacity=1 ][line width=3]    (119.64,78.07) -- (120.01,113.92) ;
		\draw [shift={(120.06,118.92)}, rotate = 269.40999999999997] [fill={rgb, 255:red, 144; green, 19; blue, 254 }  ,fill opacity=1 ][line width=3]  [draw opacity=0] (16.97,-8.15) -- (0,0) -- (16.97,8.15) -- cycle    ;
		
		\draw  [color={rgb, 255:red, 144; green, 19; blue, 254 }  ,draw opacity=1 ][line width=3]  (134.56,89.42) -- (104.56,89.3)(119.5,104.36) -- (119.62,74.36) ;
		%Straight Lines [id:da6679351436450531] 
		\draw [color={rgb, 255:red, 144; green, 19; blue, 254 }  ,draw opacity=1 ][line width=3]    (247.99,43.41) -- (264.34,75.31) ;
		\draw [shift={(266.62,79.76)}, rotate = 242.87] [fill={rgb, 255:red, 144; green, 19; blue, 254 }  ,fill opacity=1 ][line width=3]  [draw opacity=0] (16.97,-8.15) -- (0,0) -- (16.97,8.15) -- cycle    ;
		
		\draw  [color={rgb, 255:red, 144; green, 19; blue, 254 }  ,draw opacity=1 ][line width=3]  (266.41,46.9) -- (239.52,60.19)(259.61,66.99) -- (246.32,40.1) ;
		%Straight Lines [id:da6770625136984745] 
		\draw    (89.46,220) -- (219.64,219.67) ;

		%Straight Lines [id:da5599018321510362] 
		\draw    (219.64,219.67) -- (243.83,128) ;

		%Straight Lines [id:da4788531850440837] 
		\draw [color={rgb, 255:red, 144; green, 19; blue, 254 }  ,draw opacity=1 ][line width=3]    (157.97,208.07) -- (158.34,243.92) ;
		\draw [shift={(158.39,248.92)}, rotate = 269.40999999999997] [fill={rgb, 255:red, 144; green, 19; blue, 254 }  ,fill opacity=1 ][line width=3]  [draw opacity=0] (16.97,-8.15) -- (0,0) -- (16.97,8.15) -- cycle    ;
		
		\draw  [color={rgb, 255:red, 144; green, 19; blue, 254 }  ,draw opacity=1 ][line width=3]  (172.89,219.42) -- (142.89,219.3)(157.83,234.36) -- (157.95,204.36) ;
		%Straight Lines [id:da23409139624904685] 
		\draw [color={rgb, 255:red, 144; green, 19; blue, 254 }  ,draw opacity=1 ][line width=3]    (240.66,138.23) -- (231.97,173.01) ;
		\draw [shift={(230.76,177.87)}, rotate = 284.02] [fill={rgb, 255:red, 144; green, 19; blue, 254 }  ,fill opacity=1 ][line width=3]  [draw opacity=0] (16.97,-8.15) -- (0,0) -- (16.97,8.15) -- cycle    ;
		
		\draw  [color={rgb, 255:red, 144; green, 19; blue, 254 }  ,draw opacity=1 ][line width=3]  (252.23,152.98) -- (223.23,145.3)(233.89,163.64) -- (241.57,134.64) ;
		%Shape: Arc [id:dp11424793690370905] 
		\draw  [draw opacity=0] (188.93,219.66) .. controls (187.64,209.96) and (191.11,199.83) .. (199.06,192.96) .. controls (207.26,185.88) and (218.19,184.02) .. (227.8,187.09) -- (218.67,215.67) -- cycle ; \draw  [color={rgb, 255:red, 208; green, 2; blue, 27 }  ,draw opacity=1 ] (188.93,219.66) .. controls (187.64,209.96) and (191.11,199.83) .. (199.06,192.96) .. controls (207.26,185.88) and (218.19,184.02) .. (227.8,187.09) ;
		
		% Text Node
		\draw (166,49) node   {$\textcolor[rgb]{0.25,0.46,0.02}{\frac{3\pi }{4}}\textcolor[rgb]{0.25,0.46,0.02}{\leq \alpha < \pi }$};
		% Text Node
		\draw (165.33,168) node   {$\textcolor[rgb]{0.82,0.01,0.11}{\beta \ < \ }\textcolor[rgb]{0.82,0.01,0.11}{\frac{3\pi }{4}}$};

		\end{tikzpicture}
		
		\caption{On one hand, placing a $1/4$ singularity on a boundary angle $\alpha$ such as $3\pi/4 \leq \alpha < \pi$  implies a lot of rotation (and therefore of energy) in the computed cross field as seen on the section \ref{GLPiecewise}. On the other hand, placing a 1/4 singularity on a small angle $\beta < 3\pi/4$ implies much less rotational energy in the cross field. The valence $V(X, \partial G)$ on the two angles correspond to the angle between the successive oriented arrows around the corners which is the minimal rotation that the crosses can perform on those angles to keep their alignment to the boundary. In the case of $\alpha$, this angle implies that no boundary singularity has been placed since the arrow keeps pointing outside the domain. In the case of $\beta$ this angle implies that a boundary singularity has been placed since the arrow points outwards on one side of the angle and is parallel to the boundary on the other side.}
		\label{fig:3pi4} 
	\end{figure}
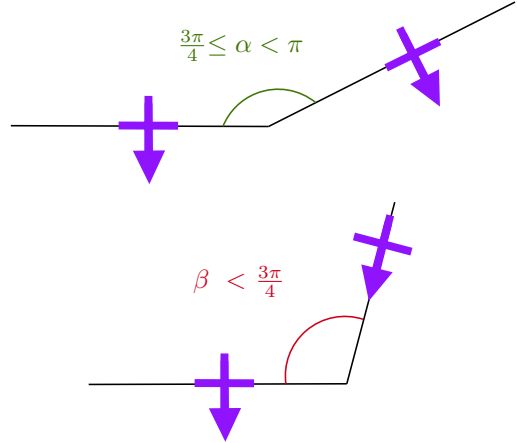

	\begin{figure}[ht] 
		\centering
		\includegraphics[width=.7\linewidth]{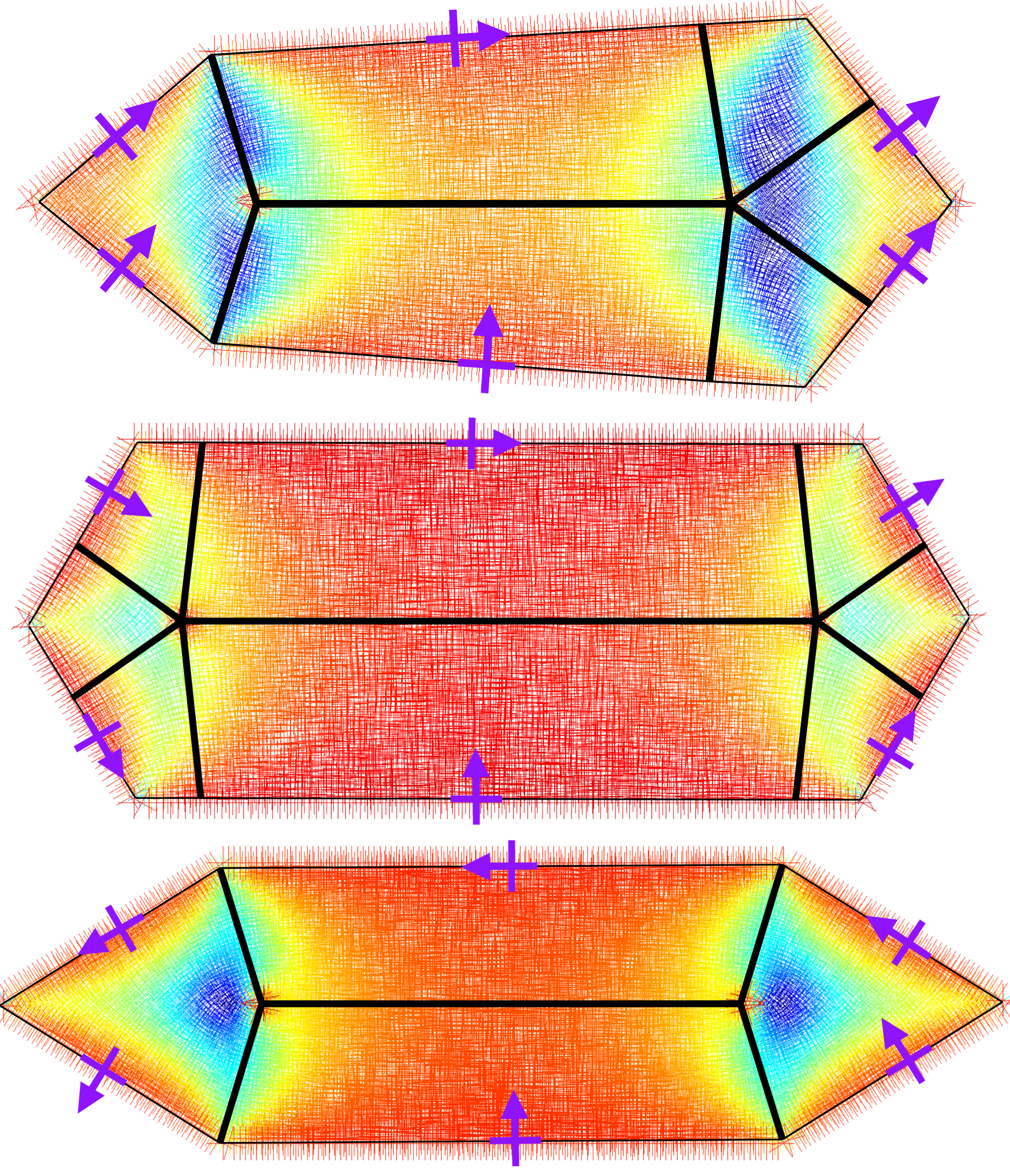} 
		\caption{In the first domain the purple orientation vector has a sum of its smallest variations on the angles on the boundary such as to be mapped on an equivalent cross equals to 0. Furthermore the sum of the indices of the internal singularities equals 0 and since there are angles greater than $\frac{3\pi}{2}$, there must be at least one singularity in the field replacing them. For the second, the sum of the variations equals -1/2 and there is 2 singularities -1/4 in the domain. For the third, the sum of the variations equals 1/2 and there is 2 singularities 1/4 in the domain.}
		\label{fig:threeAngularDomsMeshes} 
	\end{figure} 
	
	Let $\rho(M)$ be the set of the $M^{th}$ roots of the unity and $X$ be a M-directional field 
%	(see~\cite{Viertel}) 
	defined on a piecewise smooth domain $G$ with a finite number of angular points~(jump points of the derivative of the parametrization). We set
	\begin{align*}
	& \gamma : [0, 1] \rightarrow \partial G, \\
	& \gamma(0) = \gamma(1) \text{  and  }\\
	& \gamma \in W^{1, \infty}([0, 1]).
	\end{align*}
	Let $t_j$ be the angular points of $\partial G$ with $1 \leq j \leq N-1$, $t_0 = 0$ and $t_N = 1$. Then we define the valence of a directional field $X$ on the boundary $\partial G$ as 
	\[
		V(X, \partial G) = \sum\limits_{j=1}^{N} \left(X^+(\gamma(t_j)) - X^- (\gamma(t_j))\right) \left[\frac{2\pi}{\rho(M)}\right].
	\]
	Let $P_1, P_2, \ldots, P_N$ be the $N$ singularities of a field $X$ in the interior of a domain $G$.
	Let $\phi$ be the angle between the tangent of the domain boundary and an arbitrary fixed axis.
	For any piecewise smooth domains we speculate that the following equality must hold :
	\begin{equation}
	\sum\limits_{k=1}^N \text{index}(P_k) = V(X, \partial G) +  \frac{1}{2\pi} \int_{\partial G \backslash \{P_k\} }  d \phi. 
	\label{eq:PoinHopfGen}
	\end{equation}
	
	Therefore, the above figure of \ref{fig:threeAngularDomsMeshes} presents a cross field with $V(X, \partial G) = 0$, the central figure a cross field with $V(X, \partial G) = -1/2$, and the below figure a cross field with $V(X, \partial G) = 1/2$. All the figures of this article also corroborate the equation \eqref{eq:PoinHopfGen}.
	
	At this point, it is also important to note that the minimal number of singularities on the first figure of \ref{fig:threeAngularDomsMeshes} has been obtained even if the sum of the internal singularities equals 0. With six boundary segments, we must place singularities in order to end with 4-sided blocks. This configuration minimizes the energy proposed in \ref{GLEnergies}. These results corroborate the claims made by Viertel and Osting \cite{Viertel} that Ginzburg-Landau based cross field generation methods ensure the possibility to extract block-structured quadrangular meshes from the computed cross fields.
	Furthermore, the presence of inside singularities allow to  ``cancel" some of the angular singularities that become regular boundary points from which originate three directions towards the inside of the domain.
	It is cheaper to place one $1/4$ and one $-1/4$ inside singularities than let the two angles close to $\pi$ at the left of the domain be $1/4$ singularities and place two inside $-1/4$ singularities.
	
	Let introduce an other cross field to corroborate equation \eqref{eq:PoinHopfGen} for domains with Euler-Poincar\'e characteristics, $\chi$, different from 1.
	
	%Figure \ref{fig:twoBigAngles} is such that $V(X, \partial G) = 1/4$. Still we need 3 singularities in the domain (one -1/4 and two 1/4). That comes from the presence of two flat angles (greater than $3\pi/4$). Those angles cannot be seen as singularities so they must be "canceled out" via internal singularities. Furthermore, a single 1/4 singularity cannot be enough for then we will have a triangle in the resulting quad layout. The fact that our method does not converge with a single 1/4 singularity imposed must be linked with the fact that it is deeply linked with Ginzburg-Landau approaches that proves to place singularities as appropriate with quad mesh generation (see~\cite{Viertel}). To avoid the production of a triangle in the resulting quad layout we must thus introduce two other opposite singularities that could result in the configuration proposed in \ref{fig:twoBigAngles}.  

	\begin{figure}[ht] 
		\centering
		\includegraphics[width=.8\linewidth]{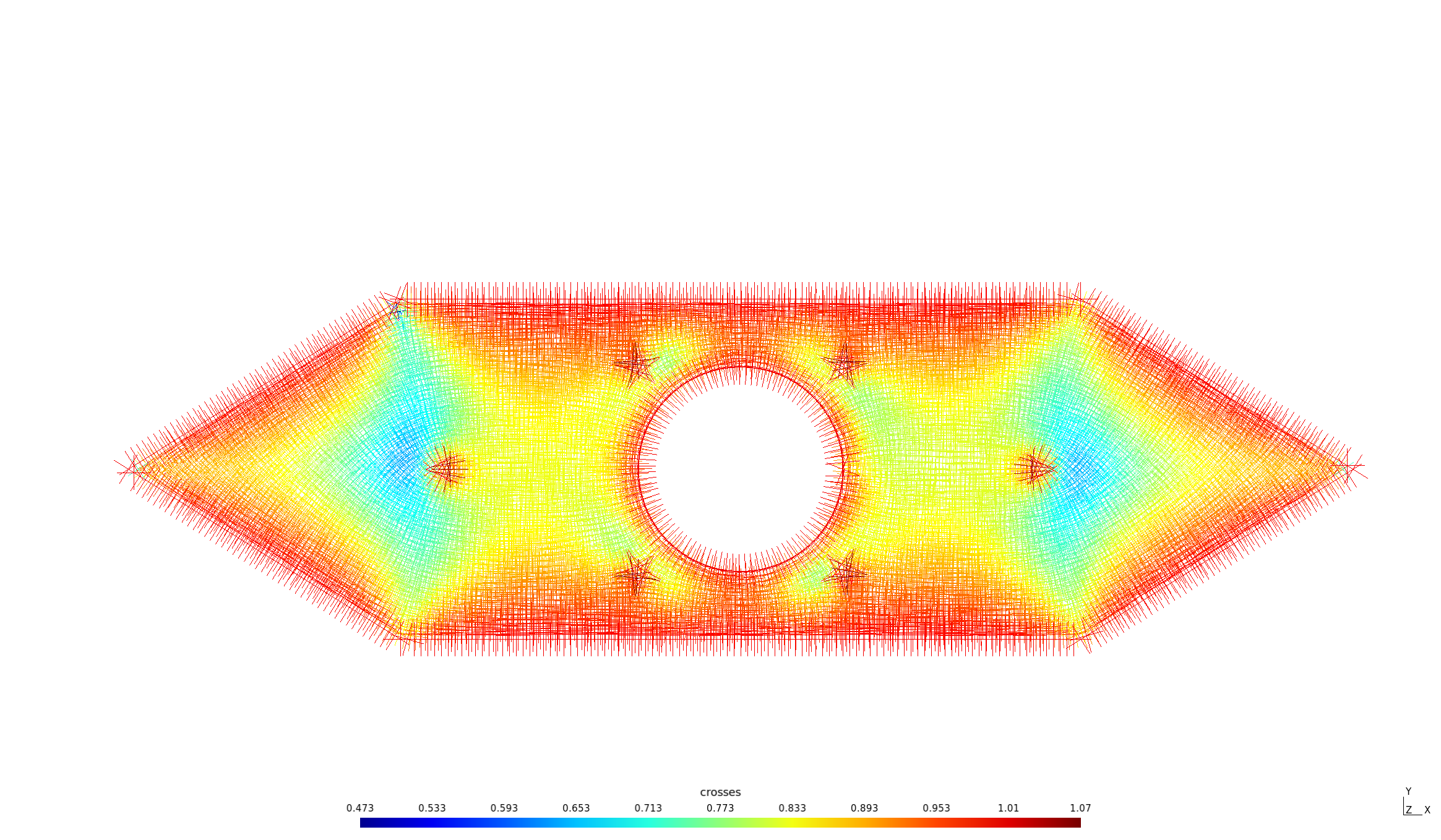} 
		\caption{Domain containing a centered regular hole and 6 singular holes (four 1/4 singular holes around the regular hole and two -1/4 between large angles).}
		\label{fig:test2} 
	\end{figure}
	
	Figure \ref{fig:test2} is such that $V(X, \partial G) = 1/2$ thanks to the external boundary and $\int_{\partial G \backslash \{x_k\} }  \frac{d \phi}{2\pi} = - 1$ thanks to the internal boundary.
	We must then have an internal singularity configuration such that 
	\[ \sum\limits_{k=1}^n \text{index}(x_k) = -1/2. \]
	However we have 4 flat angles that could be ``canceled out" by 1/4 singularities. To complete those singularities four 1/4 singularities could be placed to isolate the internal hole as in the configuration of holed singularities proposed in \ref{fig:test2}. 
	
	In conclusion, the following algorithm could be use to place singularities inside piecewise smooth closed domains. First, the equation \eqref{eq:PoinHopfGen} is computed. Then, all the angles between $3\pi/4$ and $\pi$ not included are outlined and the Euler-Poincar\'e characteristic is computed. Finally, a configurations of holed singularities must be proposed. This configuration could be as to ``cancel out" the angles between $3\pi/4$ and $\pi$ not included. 
	As our methods does not allow to explicitly place boundary singularities, one must verify that its configuration is such as to minimize the energy proposed in \ref{GLPiecewise} for the boundary singularities while respecting the equation \eqref{eq:PoinHopfGen} and the generalized Poincar\'e-Hopf theorem \eqref{eq:PoinHopf} with its imposed inside singularities.

	\subsection{Further thoughts}\label{FurThou}
	
	The minimizers of the Ginzburg-Landau energy 
	\[
	E_\epsilon (u) = \frac{1}{2} \int_G |\nabla u|^2 + \frac{1}{4\epsilon^2} \int_G (|u|^2 -1)^2
	\]
	are functions describing vector fields on surfaces.\\
	
	The study of this energy has lead in its earlier developments to the formulation of the problems used previously in this paper but the most notable conclusion, as we have already seen, is that an equivalent common simpler renormalized energy exists. This energy is about points and degrees of those points on the surface of computation.
	We have seen that its computation involves a scalar field $\Phi$.
	This scalar field $\Phi$ can be seen as the harmonic conjugate of the phase $\theta$ of a complex valued vector field
	\[
	v(x,y) = \Phi(x,y) + i \theta(x,y).
	\]
	This is an interesting information for the computation of cross field since the use of a function
	\[
	w(x,y) = H(x,y) + i \theta(x,y)
	\]
	is possible~\cite{Remacle2011} to construct perpendicular fields $u$ and $v$ that form together a cross field but the computation of $\Phi$ must be adapted to the cross field generation. 
	Notably, we must ensure alignment with the domain boundaries. 
	Let $H$ be such a field that we will give a way to calculate it later.
	As it is the real part of a complex field, the scalar field $H$ is directly linked with the size field of $w$.
	On the one hand, mathematicians will be interested in the positions and degrees of a minimum number of singularities that minimize the Ginzburg-Landau energy.
	On the other hand, experts in quadrangular block-structured meshing will rather seek to impose enough singularities in suited places and with suited degrees to produce a cross field from which a quadrangular block-structured mesh respecting, for example, a certain mesh size can be extracted.
	The computation of the field $H$ is therefore very interesting in terms of size map of the subsequent block-structured quadrangular mesh.
	
	As far as cross-field generation is concerned, the scalar field $H$ is half of a complex field from which a parameterization could be extracted (see \cite{Remacle2011}). Moreover, we can place singularities in it at will while respecting the Poincaré-Hopf theorem.
	$H(x,y) = \text{Re}(w(z))$ is solution of
	\begin{equation}
	\left\{
	\begin{array}{ccccc}
	\nabla^2 H &=& 2 \pi \sum_{j=1}^N {k_j \over 4}  \delta
	(x_j,y_j)& \text{in}& S \\
	\partial_n H &=& \kappa& \text{in}& \partial S, 
	\end{array}
	\right.
	\label{eq:pb}
	\end{equation}
	where $\kappa$ is the local curvature of $\partial S$.\\
	
	To construct the corresponding cross fields, ones can rely on a suited set of cutting and the relation
	\[
	\nabla H = \nabla^T \theta
	\]
	but this remains a challenge.\\
	
	A very interesting point of this new formulation is that it allows the placement of boundary and interior singularities in one and the same way. This allows a user to more easily and freely respect the Poincaré-Hopf theorem.
	
	Here are an example of the computation of the field $H$ given the following suitable set of singularities (the red points are $1/4$ singularities and the green ones are $-1/4$ singularities such as $\sum_{j=1}^N k_j = 4 \chi(S)$).
	\begin{figure}[ht]
		\centering
		\includegraphics[width=0.7\linewidth]{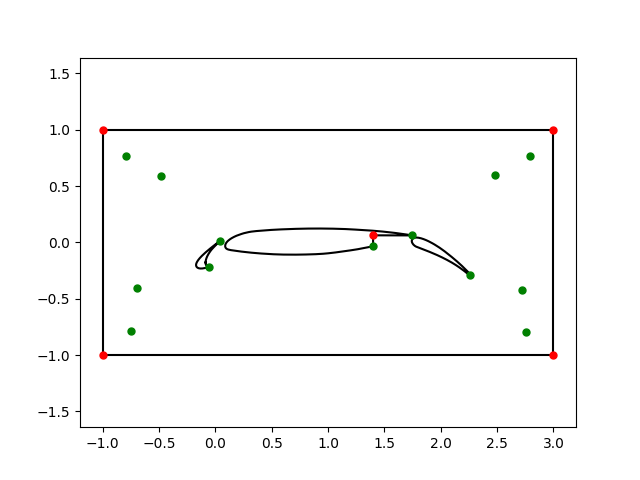}
		\caption{Position of the singularities for the computation of the scalar field $H$.}
		\label{fig:singdeltas}
	\end{figure}

	\begin{figure}[ht] 
		\begin{minipage}[b]{.5\linewidth}
			\centering
			\includegraphics[width=\linewidth]{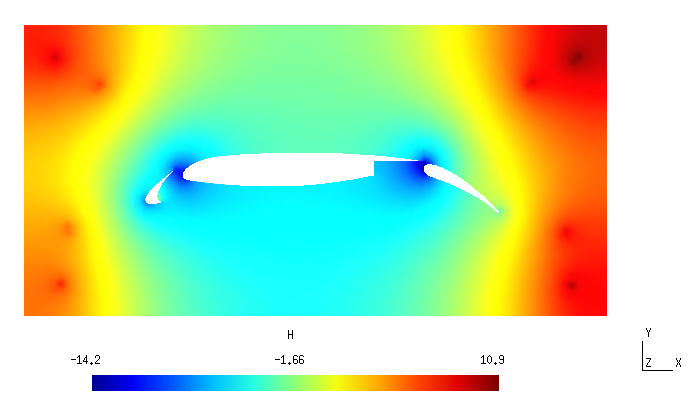}
			\label{fig:h}
			\vspace{-0cm}
		\end{minipage}%%
		\begin{minipage}[b]{.5\linewidth}
			\centering
			\includegraphics[width=\linewidth]{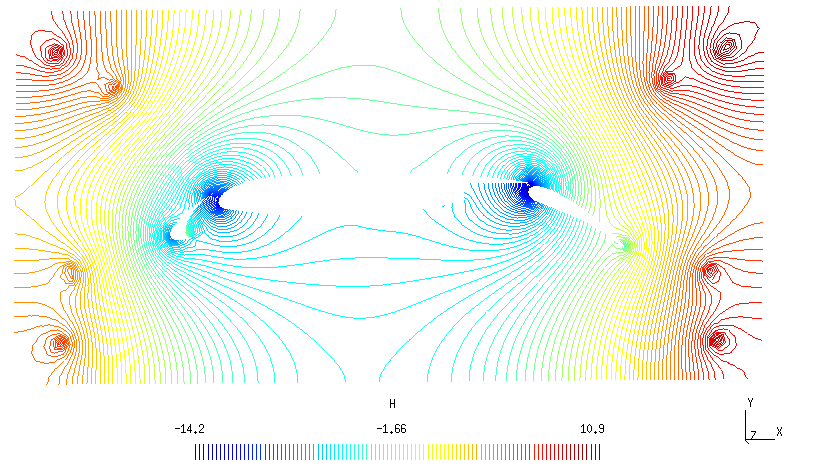}
			\label{fig:hiso}
		\end{minipage}%% 
		\caption{Computation of the scalar field $H$ from a linear Neumann problem given the positions and orders of the deltas given in Figure \ref{fig:singdeltas} is given on the left. On the right, we can see that the iso-values of the computed H are perpendicular to the boundaries of the domain.}
	\end{figure}

%	\begin{figure}[ht]
%		\centering
%		\includegraphics[width=0.5\linewidth]{H}
%		\caption{Computation of the scalar field $H$ from a linear Neumann problem given the positions and orders of the deltas given in Figure \ref{fig:singdeltas}.}
%		\label{fig:h}
%	\end{figure}
%	
%	\begin{figure}[ht]
%		\centering
%		\includegraphics[width=0.6\linewidth]{H_iso}
%		\caption{The iso-values of the computed H are perpendicular to the boundaries of the domain.}
%		\label{fig:hiso}
%	\end{figure}
	The problems that remain to be addressed are the extraction of a cross field, and a size map based on a calculated $H$ field.
	
	On this subject, note that size maps or local refinements already appear in the literature which is a good thing for high-precision numerical simulation. 
	
	Jiang et al.~\cite{Jiang2015} use them not as a goal but as a tool. 
	They compute a discrete metric on
	the input surface to obtain a cross field.
	
	Lyon at al.~\cite{Lyon2020CostML} work directly on block-structured quadrangular meshes that they locally refine.
	For this local refinement, they impose new singularities i.e. new irregular vertices on the mesh.
	They are able to perform local anisotropic refinement with 	few singularities.
	To do so they use a binary program.
	They concentrate on local split of local edges or elements.
	Thus, their method makes it possible to perform local surgeries on meshes already generated where we recommend to take into account the size map as soon as the cross-field is generated in order to directly produce the desired mesh.
	
	\newpage
	\section*{Conclusion}
	
	%%% Contexte
	Most methods generating block-structured quadrangular meshes, with potential user-based inputs, rely on the generation of smooth cross fields.
	In this paper, we have broadened the theoretical context and given new methods to generate such cross fields in the context of Ginzburg-Landau theory.
	In particular, we have defined a new Ginzburg-Landau energy for the boundaries of piecewise smooth domains containing singularities which was formally missing in the scientific literature as noted by Viertel and Osting \cite{Viertel}.
	Moreover, we have given a method to place singularities that does not correspond to those of Ginzburg-Landau energy minimizers and offered a way to evaluate them in terms of this energy.
	Practically, there are three ways to decrease the Ginzburg-Landau energy while respecting the Poincaré-Hopf theorem: (i) decrease the number of singularities, (ii) bring their degrees closer to zero, or (iii) move them to make the cross field smoother.
%	Those methods are already providing interesting results but their are unable to base their solutions on criteria that are really interesting for numerical simulation.
%	In addition to this absence of quality criteria to be used to optimize the computed cross field, the current formulations, although highly optimized, are often complex, involving, for example, mixed-integer optimization.\\

	%%% Ce que j'ai fait

	%%% Nuémrique : 

	% 1) construire avec trous 
	What are the methods and theoretical elements presented in this article and what can be said of their strengths and weaknesses.
	First, we have given a method to generate a cross field with prescribed inner singularity configurations.
	Said configurations consist of a set of holes with prescribed degrees implying a different number of directions from the usual four that meet on a cross.
	Numerically, the method corresponds to the optimization of a linear objective function with very localized quadratic constraints. Another constraint could be needed to ensure the degrees on these holes.
	Indeed, a degree must be computed with crosses of a norm close to one.
	This method works only with singularity configurations that respect the Poincaré-Hopf theorem.
	Unfortunately, this method only allows to impose inner singularities and does not give us clear information concerning the eventual boundary singularities.
	Other generalizations of this method for singularity configurations that do not respect the topology of the domain have been proposed but they are less efficient, generate singular crosses in the field and unintended singularity configurations.
	To help a user to respect the Poincaré-Hopf theorem we have also proposed in this article new theoretical contributions.
	
	% 2) évaluer la position des trous 
	In addition, we have given a way to evaluate the positions and degrees of such holed singularities in terms of Ginzburg-Landau energy but without computing an actual cross field which takes us away from the block-structured quadrangular mesh generation.
	This evaluation gives the same Ginzburg-Landau energy as the asymptotic minimizers of the Ginzburg-Landau functional if the singularities configuration coincide and the radius of the holes tend towards zero.
	
	Hence, the proposed perforation method, as it both produces quality non-singular cross fields and gives the option to choose the holed placements mimicking inner singularities, could be a promising path to efficiently and flexibly handle block-structured quadrangular mesh generation while keeping quality criterion via Ginzburg-Landau energy.
	
	% 3) construire le champ scalaire.
	We have ended the paper with a Neumann problem with zero-radius holes, similar to the ones of the asymptotic Ginzburg-Landau energy for holed domains. 
	In its formulation both inner and boundary singularities are seen in the same way. Furthermore, its input consists of configuration of both inner and boundary singularities.
	This problem allows to construct a single-valued scalar field strongly linked to the multivalued angles of a cross field. 
	This field corresponds to the harmonic conjugate of the phase field because when it is added with i times the phase field, it gives a complex harmonic field almost everywhere.
	Furthermore, the gradients of these two scalar fields are perpendicular.  
	Therefore, based on a given suited singularity configuration, this scalar linear Neumann problem allows to construct a field of the harmonic conjugates of the phases of a cross field. This problem may lead to a particularly simple and informative way to compute cross field. Indeed, this scalar field, as it correspond to the real part of a complex field, could be strongly linked with the local size of the cross field.\\

	%	We present a Neumann problem with these zero-radius holes as singularities in the perspectives.
	%	This scalar field is strongly linked to the multivalued angles of a cross field. 
	%	Indeed this field and the phase field of a cross field have perpendicular gradients and form together a complex harmonic field almost everywhere. This could lead to a new simpler way to represent cross fields.
	
	%%% Théorique : 
	
	% 1) Energy de G-L pour les bords d'un domaine piecewise smooth 
	Through the theoretical developments, we have constructed a Ginzburg-Landau energy for the boundaries of piecewise smooth domains. Minimizing this energy leads to new ranges of limit angles on which boundary singularities appear.
	% 2) nvlle interpretation du th de P-H
	Besides, we have presented a new interpretation of the Poincaré-Hopf theorem. This interpretation makes it possible to apprehend in a simpler way the total degree of the configuration of the internal singularities.
	Although our main method does not allow to place boundary singularities, those two contributions help a user to gain crucial information about them.
	
	%%% Qu'est-ce que ça apporte
%	The presented perforation strategy enables the production of cross fields with chosen singularity configurations on piecewise smooth bounded domains. This method linked with Ginzburg-Landau theory, is efficient and generate regular cross fields for topologically coherent choices of singular hole configurations. 
%	
	%%% Quelles sont les limitations et quel est le futur de mon travail
	
	Throughout this paper, we have used the energy of Ginzburg-Landau as a quality criterion.	
	We expect that the presented methods and theories can later be adapted to other quality criteria, for instance, the adequacy of the future block-structured quadrangular mesh to a prescribed size field.
	
	We hope that our new contributions, unified under the Ginzburg-Landau theory, could contribute to the construction of cross field with singularity configuration focused on the needs of subsequent high performance numerical simulations on block-structured quadrangular meshes.

	\section*{Acknowledgments}
	
	%Many thanks to R\'emy Rodiac (UCLouvain - Orsay Paris-Sud) for his help in understanding and using many mathematical results, and also for his advice in writing this article. Thanks to Fran\c{c}ois Henrotte (UCLouvain) for his reviews and wise inputs. Thanks to Yi-Ann Hagelstein (UCLouvain - National Central University) and Pierre De Wael (UCLouvain) for their spelling corrections and their many tips on discursive construction. Also, thanks to my thesis promoters Jean-Fran\c{c}ois Remacle (UCLouvain) and Jean Van Schaftingen (UCLouvain) for their deep knowledge, their wise insights and their help. Thanks to Pierre-Alexandre Beaufort (UCLouvain) for code sharing and to Alexandre Chemin (UCLouvain) for his help with the implementations.
	%Last but not least, 
	%many thanks to the UCLouvain for the subvention of my work.\\
	
	I would like to thank my dear friend Pierre de Wael for his generous help with English writing. I would also like to thank Rémy Rodiac for his friendship and his many clarifications.
	
	% Either type in your references using
	% \begin{thebibliography}{}
	% Text
	% \end{thebibliography}
	%
	% or
	%
	% Compile your BiBTeX database using our plos2015.bst
	% style file and paste the contents of your .bbl file
	% here. See http://journals.plos.org/plosone/s/latex for 
	
	\section{Appendix}
	
	\subsection{Demonstrations for the optimal ranges of boundary singularities}
	
	In the following demonstrations, we choose to add a $\pm 1/4$ singularity or to change the order of a boundary singularity by adding $\pm 1/4$ to it. These are concision choices but the argumentation holds true for other ones.
	
	\subsubsection{Demonstration of the theorem \ref{theo:balanced}}\label{demo:balanced}
	\begin{proof}
		If we have an optimal choice of boundary singularities, we then should have by optimality for every $i \in \{ 1, \ldots, m \}$,
		\begin{multline*}
		\sum_{i\neq f} \left( \frac{\left( \frac{1}{2} - \frac{\alpha_i}{2\pi} - k_i \right)^2}{\frac{\alpha_i}{2 \pi}}\right) + \frac{\left( \frac{1}{2} - \frac{\alpha_{f}}{2\pi} - \left(k_f \pm \frac{1}{4}\right) \right)^2}{\frac{\alpha_{f}}{2 \pi}} + \sum_j |l_j|^2 + \left(\frac{1}{4}\right)^2 \\ \geq \sum_{i}\left( \frac{\left( \frac{1}{2} - \frac{\alpha_i}{2\pi} - k_i \right)^2}{\frac{\alpha_i}{2 \pi}}\right) +  \sum_j |l_j|^2,
		\end{multline*}
		since if we have the optimal order of singularities, increase or decrease the order of a boundary singularity (and add one compensating internal singularity to respect the Poincaré-Hopf theorem) must imply more energy. Some simplifications gives us the following inequality:
		\[
		\frac{\left( \frac{1}{2} - \frac{\alpha_{f}}{2\pi} - \left(k_f \pm \frac{1}{4}\right) \right)^2}{\frac{\alpha_{f}}{2 \pi}} + \left(\frac{1}{4}\right)^2 \geq  \frac{\left( \frac{1}{2} - \frac{\alpha_f}{2\pi} - k_f\right)^2}{\frac{\alpha_f}{2 \pi}} .
		\]
		Developing this last inequality gives
		\begin{multline*}
		\frac{ \left(\frac{1}{2} - \frac{\alpha_{f}}{2\pi}\right)^2 - 2\left(\frac{1}{2} - \frac{\alpha_{f}}{2\pi}\right)\left(k_f \pm \frac{1}{4}\right) + \left(k_f \pm \frac{1}{4}\right)^2}{\frac{\alpha_{f}}{2 \pi}} + \left(\frac{1}{4}\right)^2 \geq\\ \frac{ \left(\frac{1}{2} - \frac{\alpha_{f}}{2\pi}\right)^2 - 2\left(\frac{1}{2} - \frac{\alpha_{f}}{2\pi}\right)k_f + {k_f}^2}{\frac{\alpha_{f}}{2 \pi}},
		\end{multline*}
		which simplifies successively as
		\[
		\frac{\mp 2\left(\frac{1}{2} - \frac{\alpha_{f}}{2\pi}\right)\frac{1}{4} + \left(k_f \pm \frac{1}{4}\right)^2}{\frac{\alpha_{f}}{2 \pi}} + \left(\frac{1}{4}\right)^2 \geq \frac{{k_f}^2}{\frac{\alpha_{f}}{2 \pi}},
		\]
		\[
		\frac{\mp 2\left(\frac{1}{2} - \frac{\alpha_{f}}{2\pi}\right)\frac{1}{4} \pm 2k_f\frac{1}{4} + \left(\frac{1}{4}\right)^2}{\frac{\alpha_{f}}{2 \pi}} + \left(\frac{1}{4}\right)^2 \geq 0,
		\]	
		\[
		\frac{\left(\frac{1}{4}\right)^2 \mp 2\left(\frac{1}{2} - \frac{\alpha_{f}}{2\pi} - k_f \right)\frac{1}{4}}{\frac{\alpha_{f}}{2 \pi}} + \left(\frac{1}{4}\right)^2 \geq 0,
		\]
		that we divide by $(1/4)^2$ and multiply by $\alpha_{f}/(2\pi)$ which gives
		\[
		1 \mp 2\left(\frac{1}{2} - \frac{\alpha_{f}}{2\pi} - k_f \right)4 + \frac{\alpha_{f}}{2 \pi} \geq 0,
		\]
		\[
		\frac{1}{8} \left(1 + \frac{\alpha_{f}}{2 \pi}\right) \geq \pm \left(\frac{1}{2} - \frac{\alpha_{f}}{2\pi} - k_f \right),
		\]
		\[
		-\frac{1}{8} \left(1 + \frac{\alpha_{f}}{2 \pi}\right) \leq \left(\frac{1}{2} - \frac{\alpha_{f}}{2\pi} - k_f \right) \leq \frac{1}{8} \left(1 + \frac{\alpha_{f}}{2 \pi}\right),
		\]
		\[
		- \frac{1}{2} + \frac{\alpha_{f}}{2\pi} -\frac{1}{8} \left(1 + \frac{\alpha_{f}}{2 \pi}\right) \leq  - k_f \leq - \frac{1}{2} + \frac{\alpha_{f}}{2\pi} + \frac{1}{8} \left(1 + \frac{\alpha_{f}}{2 \pi}\right),
		\]
		\[
		\frac{1}{2} - \frac{\alpha_{f}}{2\pi} - \frac{1}{8} \left(1 + \frac{\alpha_{f}}{2 \pi}\right) \leq k_f \leq \frac{1}{2} - \frac{\alpha_{f}}{2\pi} +\frac{1}{8} \left(1 + \frac{\alpha_{f}}{2 \pi}\right) ,
		\]
		\[
		\frac{ 1 - \frac{1}{4}}{2} - \left(1 + \frac{1}{8}\right)\frac{\alpha_{f}}{2\pi} \leq  k_f \leq \frac{ 1 + \frac{1}{4}}{2}  - \left(1 - \frac{1}{8}\right)\frac{\alpha_{f}}{2\pi}.
		\]
	\end{proof}

	\subsubsection{Demonstration of the theorem \ref{theo:rare}}\label{demo:rare}

	If we have an optimal choice of boundary singularities, we then should have by optimality for every $i \in \{ 1, \ldots, m \}$,
	\begin{multline*}
	\sum_{i\neq f} \left( \frac{\left( \frac{1}{2} - \frac{\alpha_i}{2\pi} - k_i \right)^2}{\frac{\alpha_i}{2 \pi}}\right) + \frac{\left( \frac{1}{2} - \frac{\alpha_{f}}{2\pi} - \left(k_f \pm \frac{1}{4}\right) \right)^2}{\frac{\alpha_{f}}{2 \pi}} + \sum_j |l_j|^2 \\ \leq \sum_{i}\left( \frac{\left( \frac{1}{2} - \frac{\alpha_i}{2\pi} - k_i \right)^2}{\frac{\alpha_i}{2 \pi}}\right) + \sum_j |l_j|^2 \pm \left(\frac{1}{4}\right)^2,
	\end{multline*}
	since we have boundary singularity degrees that sum up to less than $\chi$ but we have choose the optimal choice of degrees by hypothesis, increase or decrease the order of a boundary singularity must imply less energy than adding or removing an internal one in order to closer respect the Poincaré-Hopf theorem. Some simplifications gives us the following inequality:
	\[
	\frac{\left( \frac{1}{2} - \frac{\alpha_{f}}{2\pi} - \left(k_f \pm \frac{1}{4}\right) \right)^2}{\frac{\alpha_{f}}{2 \pi}} \leq  \frac{\left( \frac{1}{2} - \frac{\alpha_f}{2\pi} - k_f\right)^2}{\frac{\alpha_f}{2 \pi}} \pm \left(\frac{1}{4}\right)^2.
	\]
	Developing this last inequality gives
	\begin{multline*}
	\frac{ \left(\frac{1}{2} - \frac{\alpha_{f}}{2\pi}\right)^2 - 2\left(\frac{1}{2} - \frac{\alpha_{f}}{2\pi}\right)\left(k_f \pm \frac{1}{4}\right) + \left(k_f \pm \frac{1}{4}\right)^2}{\frac{\alpha_{f}}{2 \pi}} \leq\\ \frac{ \left(\frac{1}{2} - \frac{\alpha_{f}}{2\pi}\right)^2 - 2\left(\frac{1}{2} - \frac{\alpha_{f}}{2\pi}\right)k_f + {k_f}^2}{\frac{\alpha_{f}}{2 \pi}} \pm \left(\frac{1}{4}\right)^2 ,
	\end{multline*}
	which simplifies successively as
	\[
	\frac{\mp 2\left(\frac{1}{2} - \frac{\alpha_{f}}{2\pi}\right)\frac{1}{4} + \left(k_f \pm \frac{1}{4}\right)^2}{\frac{\alpha_{f}}{2 \pi}} \leq \frac{{k_f}^2}{\frac{\alpha_{f}}{2 \pi}} \pm \left(\frac{1}{4}\right)^2,
	\]
	\[
	\frac{\mp 2\left(\frac{1}{2} - \frac{\alpha_{f}}{2\pi}\right)\frac{1}{4} \pm 2k_f\frac{1}{4} + \left(\frac{1}{4}\right)^2}{\frac{\alpha_{f}}{2 \pi}}  \leq \pm \left(\frac{1}{4}\right)^2,
	\]	
	\[
	\frac{\left(\frac{1}{4}\right)^2 \mp 2\left(\frac{1}{2} - \frac{\alpha_{f}}{2\pi} - k_f \right)\frac{1}{4}}{\frac{\alpha_{f}}{2 \pi}}  \leq \pm \left(\frac{1}{4}\right)^2,
	\]
	that we divide by $(1/4)^2$ and multiply by $\alpha_{f}/(2\pi)$ which gives
	\[
	1 \mp 2\left(\frac{1}{2} - \frac{\alpha_{f}}{2\pi} - k_f \right)4  \leq \pm \frac{\alpha_{f}}{2 \pi},
	\]
	\[
	1 \mp 2\left(\frac{1}{2} - \frac{\alpha_{f}}{2\pi} - k_f \right)4 \mp \frac{\alpha_{f}}{2 \pi} \leq 0,
	\]
	\[
	\mp 2\left(\frac{1}{2} - \frac{\alpha_{f}}{2\pi} - k_f + \frac{1}{8} \frac{\alpha_{f}}{2 \pi} \right)4 + 1 \leq 0,
	\]
	\[
	\frac{1}{8}  \leq \pm \left(\frac{1}{2} + \frac{\alpha_{f}}{2\pi}\left( \frac{1}{8} - 1 \right) - k_f \right),
	\]
	\[
	-\frac{1}{8} \leq \left(\frac{1}{2} + \frac{\alpha_{f}}{2\pi}\left( \frac{1}{8} - 1 \right) - k_f \right) \leq \frac{1}{8}  ,
	\]
	\[
	-\frac{1}{8} - \frac{1}{2} + \frac{\alpha_{f}}{2\pi} \left( 1 - \frac{1}{8} \right)  \leq  - k_f \leq  \frac{1}{8}  - \frac{1}{2} + \frac{\alpha_{f}}{2\pi} \left( 1 - \frac{1}{8} \right) ,
	\]
	\[
	\frac{ 1 - \frac{1}{4}}{2} - \left(1 - \frac{1}{8}\right)\frac{\alpha_{f}}{2\pi} \leq  k_f \leq \frac{ 1 + \frac{1}{4}}{2}  - \left(1 - \frac{1}{8} \right)\frac{\alpha_{f}}{2\pi}.
	\]
	
	\subsubsection{Demonstration of the theorem \ref{theo:excess}}\label{demo:excess}
	
	If we have an optimal choice of boundary singularities, we then should have by optimality for every $i \in \{ 1, \ldots, m \}$,
	\begin{multline*}
	\sum_{i\neq f} \left( \frac{\left( \frac{1}{2} - \frac{\alpha_i}{2\pi} - k_i \right)^2}{\frac{\alpha_i}{2 \pi}}\right) + \frac{\left( \frac{1}{2} - \frac{\alpha_{f}}{2\pi} - \left(k_f \mp \frac{1}{4}\right) \right)^2}{\frac{\alpha_{f}}{2 \pi}} + \sum_j |l_j|^2 \\ \geq \sum_{i}\left( \frac{\left( \frac{1}{2} - \frac{\alpha_i}{2\pi} - k_i \right)^2}{\frac{\alpha_i}{2 \pi}}\right) + \sum_j |l_j|^2 \pm \left(\frac{1}{4}\right)^2,
	\end{multline*}
	since we have boundary singularity degrees that sum up to more that $\chi$ but we have choose the optimal choice of degrees by hypothesis, increase or decrease the order of a boundary singularity must imply more energy than adding or removing an internal one in order to closer respect the Poincaré-Hopf theorem. Some simplifications gives us the following inequality:
	\[
	\frac{\left( \frac{1}{2} - \frac{\alpha_{f}}{2\pi} - \left(k_f \mp \frac{1}{4}\right) \right)^2}{\frac{\alpha_{f}}{2 \pi}} \geq  \frac{\left( \frac{1}{2} - \frac{\alpha_f}{2\pi} - k_f\right)^2}{\frac{\alpha_f}{2 \pi}} \pm \left(\frac{1}{4}\right)^2.
	\]
	Developing this last inequality gives
	\begin{multline*}
	\frac{ \left(\frac{1}{2} - \frac{\alpha_{f}}{2\pi}\right)^2 - 2\left(\frac{1}{2} - \frac{\alpha_{f}}{2\pi}\right)\left(k_f \mp \frac{1}{4}\right) + \left(k_f \mp \frac{1}{4}\right)^2}{\frac{\alpha_{f}}{2 \pi}} \geq\\ \frac{ \left(\frac{1}{2} - \frac{\alpha_{f}}{2\pi}\right)^2 - 2\left(\frac{1}{2} - \frac{\alpha_{f}}{2\pi}\right)k_f + {k_f}^2}{\frac{\alpha_{f}}{2 \pi}} \pm \left(\frac{1}{4}\right)^2 ,
	\end{multline*}
	which simplifies successively as
	\[
	\frac{\pm 2\left(\frac{1}{2} - \frac{\alpha_{f}}{2\pi}\right)\frac{1}{4} + \left(k_f \pm \frac{1}{4}\right)^2}{\frac{\alpha_{f}}{2 \pi}} \geq \frac{{k_f}^2}{\frac{\alpha_{f}}{2 \pi}} \pm \left(\frac{1}{4}\right)^2,
	\]
	\[
	\frac{\pm 2\left(\frac{1}{2} - \frac{\alpha_{f}}{2\pi}\right)\frac{1}{4} \pm 2k_f\frac{1}{4} + \left(\frac{1}{4}\right)^2}{\frac{\alpha_{f}}{2 \pi}}  \geq \pm \left(\frac{1}{4}\right)^2,
	\]	
	\[
	\frac{\left(\frac{1}{4}\right)^2 \pm 2\left(\frac{1}{2} - \frac{\alpha_{f}}{2\pi} - k_f \right)\frac{1}{4}}{\frac{\alpha_{f}}{2 \pi}}  \geq \pm \left(\frac{1}{4}\right)^2,
	\]
	that we divide by $(1/4)^2$ and multiply by $\alpha_{f}/(2\pi)$ which gives
	\[
	1 \pm 2\left(\frac{1}{2} - \frac{\alpha_{f}}{2\pi} - k_f \right)4  \geq \pm \frac{\alpha_{f}}{2 \pi},
	\]
	\[
	1 \pm 2\left(\frac{1}{2} - \frac{\alpha_{f}}{2\pi} - k_f \right)4 \mp \frac{\alpha_{f}}{2 \pi} \geq 0,
	\]
	\[
	\pm 2\left(\frac{1}{2} - \frac{\alpha_{f}}{2\pi} - k_f - \frac{1}{8} \frac{\alpha_{f}}{2 \pi} \right)4 + 1 \geq 0,
	\]
	\[
	\frac{1}{8}  \geq \pm \left(\frac{1}{2} + \frac{\alpha_{f}}{2\pi}\left( - \frac{1}{8} - 1 \right) - k_f \right),
	\]
	\[
	-\frac{1}{8} \leq \left(\frac{1}{2} + \frac{\alpha_{f}}{2\pi}\left( - \frac{1}{8} - 1 \right) - k_f \right) \leq \frac{1}{8}  ,
	\]
	\[
	-\frac{1}{8} - \frac{1}{2} + \frac{\alpha_{f}}{2\pi} \left( 1 + \frac{1}{8} \right)  \leq  - k_f \leq  \frac{1}{8}  - \frac{1}{2} + \frac{\alpha_{f}}{2\pi} \left( 1 + \frac{1}{8} \right) ,
	\]
	\[
	\frac{ 1 - \frac{1}{4}}{2} - \left(1 + \frac{1}{8}\right)\frac{\alpha_{f}}{2\pi} \leq  k_f \leq \frac{ 1 + \frac{1}{4}}{2}  - \left(1 + \frac{1}{8} \right)\frac{\alpha_{f}}{2\pi}.
	\]
	
	\nolinenumbers
	
	\bibliography{references}
	
	%\begin{thebibliography}{10}
	%
	%\bibitem{bib1}
	%Conant GC, Wolfe KH.
	%\newblock {{T}urning a hobby into a job: how duplicated genes find new
	%  functions}.
	%\newblock Nat Rev Genet. 2008 Dec;9(12):938--950.
	%
	%\bibitem{bib2}
	%Ohno S.
	%\newblock Evolution by gene duplication.
	%\newblock London: George Alien \& Unwin Ltd. Berlin, Heidelberg and New York:
	%  Springer-Verlag.; 1970.
	%
	%\bibitem{bib3}
	%Magwire MM, Bayer F, Webster CL, Cao C, Jiggins FM.
	%\newblock {{S}uccessive increases in the resistance of {D}rosophila to viral
	%  infection through a transposon insertion followed by a {D}uplication}.
	%\newblock PLoS Genet. 2011 Oct;7(10):e1002337.
	%
	%\end{thebibliography}

\end{document}